\documentclass{article}

\usepackage[hmargin=3.5cm,vmargin=3cm]{geometry}
\usepackage{graphicx} 
\usepackage{amsmath,amsfonts,amssymb,amsthm}
\usepackage{mathtools}
\usepackage{xcolor}
\usepackage{hyperref}
\hypersetup{
    colorlinks,
    linkcolor={red!50!black},
    citecolor={blue!50!black},
    urlcolor={blue!80!black}
}
\usepackage{cleveref}
\usepackage{enumitem}
\usepackage{footmisc}

\theoremstyle{plain}
\newtheorem{thm}{Theorem}
\Crefname{thm}{Theorem}{Theorems}
\newtheorem{cor}[thm]{Corollary}
\newtheorem{prop}[thm]{Proposition}
\newtheorem{lem}[thm]{Lemma}
\Crefname{lem}{Lemma}{Lemmas}
\newtheorem{claim}[thm]{Claim}
\newtheorem*{hyp*}{Hypothesis}

\theoremstyle{definition}
\newtheorem{defn}[thm]{Definition}
\Crefname{defn}{Definition}{Definitions}
\newtheorem{rmk}[thm]{Remark}
\newtheorem{exa}[thm]{Example}

\crefname{thm}{theorem}{theorems}
\Crefname{thm}{Theorem}{Theorems}
\crefname{cor}{corollary}{corollaries}
\Crefname{cor}{Corollary}{Corollaries}
\crefname{prop}{proposition}{propositions}
\Crefname{prop}{Proposition}{Propositions}
\crefname{lem}{lemma}{lemmas}
\Crefname{lem}{Lemma}{Lemmas}
\crefname{claim}{claim}{claims}
\Crefname{claim}{Claim}{Claims}
\crefname{defn}{definition}{definitions}
\Crefname{defn}{Definition}{Definitions}
\crefname{rmk}{remark}{remarks}
\Crefname{rmk}{Remark}{Remarks}
\crefname{exa}{example}{examples}
\Crefname{exa}{Example}{Examples}
\AddToHook{env/cor/begin}{\crefalias{thm}{cor}}
\AddToHook{env/prop/begin}{\crefalias{thm}{prop}}
\AddToHook{env/lem/begin}{\crefalias{thm}{lem}}
\AddToHook{env/claim/begin}{\crefalias{thm}{claim}}
\AddToHook{env/defn/begin}{\crefalias{thm}{defn}}
\AddToHook{env/rmk/begin}{\crefalias{thm}{rmk}}
\AddToHook{env/exa/begin}{\crefalias{thm}{exa}}

\newcommand{\hypref}[1]{(\nameref{#1})}

\renewcommand{\d}{\mathrm{d}}

\DeclareMathOperator{\rk}{rk}

\DeclareMathOperator{\End}{End}
\DeclareMathOperator{\Fol}{Fol}

\DeclareMathOperator{\CS}{CS}
\DeclareMathOperator{\CpS}{CpS}

\newcommand{\bS}{\mathbb{S}}
\newcommand{\bZ}{\mathbb{Z}}
\newcommand{\bN}{\mathbb{N}}
\newcommand{\bR}{\mathbb{R}}
\newcommand{\bC}{\mathbb{C}}
\newcommand{\bT}{\mathbb{T}}

\newcommand{\cF}{\mathcal{F}}
\newcommand{\cM}{\mathcal{M}}

\newcommand{\cK}{\mathcal{K}}

\newcommand{\cS}{\mathcal{S}}

\newcommand{\cP}{\mathcal{P}}

\newcommand{\Vol}{\mathrm{Vol}}

\newcommand{\VertBdle}{\mathrm{Vert}}
\newcommand{\HorBdle}{\mathrm{Hor}}

\DeclareMathOperator{\ev}{ev}

\title{On some rigidity results for \\ high-dimensional confoliations}
\author{Fabio Gironella \and Seungook Yu}
\date{}

\newcommand{\Addresses}{{
  \bigskip
  \footnotesize

  \textsc{CNRS - Laboratoire de Mathématiques Jean Leray, Nantes Université}\par\nopagebreak
  \textit{E-mail address}: \texttt{fabio.gironella@cnrs.fr}

  \medskip

  \textsc{Department of Mathematics, Pohang University of Science and Technology (POSTECH)}\par\nopagebreak
  \textit{E-mail address}: \texttt{yso1460@postech.ac.kr}

}}

\begin{document}

\maketitle

\begin{abstract}
The main result is an obstruction to the presence of contact regions for c-symplectic confoliations in any odd ambient dimension.
As a first application, we prove a $5$-dimensional analogue (under stronger assumptions) of a Reeb stability theorem for confoliations in dimension $3$ due to Eliashberg and Thurston.
Refining work of the first author with Cardona, we further give two new notions of contact approximations of c-symplectic confoliations in any odd ambient dimension generalizing the $3$-dimensional notion of Eliashberg--Thurston. 
As a second application of our main theorem, we then give an example, in any ambient odd dimension, of a c-symplectic (and in fact strong symplectic) foliation that cannot be contact approximated, in either of these two senses.
This relies on the property, of independent interest, that contact structures approximating (in these senses) a weakly c-symplectic confoliation are also weakly fillable.
The proofs of all these statements rely on pseudo-holomorphic curve theory.
Lastly, we describe several natural classes of examples of these types of contact approximations of c-symplectic confoliations.
\end{abstract}




\section{Introduction}
\label{sec:intro}

There are two types of hyperplane distributions on smooth odd-dimensional manifolds that appear naturally in several contexts in mathematics and physics: (regular, corank $1$) \emph{foliations}, and \emph{contact structures}.
The first ones are \emph{integrable}, i.e.\ admit through every point a local hypersurface which is tangent to the foliation at all points. 
Contact structures are instead \emph{maximally non-integrable}, i.e.\ the maximal dimension of a local submanifold everywhere tangent to the contact structure is equal to half the rank of the latter, and hence of codimension strictly higher than $1$ (in ambient odd dimension at least $3$). 
In an equivalent and more computation-friendly manner, (coorientable, regular, corank $1$) foliations and (positive, coorientable) contact structures on an ambient (oriented) $(2n+1)$-manifold can be described as kernels of a $1$-form $\alpha$ satisfying, respectively, $\alpha\wedge \d\alpha = 0$ and $\alpha\wedge \d\alpha^n>0$.

Despite this extreme difference in local behavior, in ambient dimension $3$ Eliashberg and Thurston found in \cite{EliashbergThurston_Confoliations} a very striking connection between these two types of hyperplane fields.
Namely, it is proven in \cite{EliashbergThurston_Confoliations} that, besides the foliation by $2$-spheres on $S^1\times S^2$, any (codimension $1$) $C^2$-foliation on a $3$-dimensional manifold can be approximated, in the space of plane fields equipped with its $C^0$-topology, by contact structures.
The techniques used are very much of topological nature, exploiting deep structural results of corank $1$ foliations, such as Sacksteder's theorem \cite{Sacksteder_FoliationsPseudogroups} and Reeb stability \cite[Theorem 2.4.1]{CandelConlon_FoliationsIBook}, as well as their specializations to the case of foliations (and ``laminations'') by surfaces on $3$-manifolds.

In fact, this approximation is also geometrically meaningful.
More precisely, and of particular interest for this work, Eliashberg--Thurston also proved that, in ambient dimension $3$, a ``rigid'' foliation can only be approximated by ``rigid'' contact structures:
if $(\xi_n)_{n\in \bN}$ is a sequence of contact structures approximating a \emph{taut} foliation $\cF$, then, for $n$ big enough, $\xi_n$ is \emph{weakly symplectically semi-fillable}, and hence also \emph{tight}.
Here, ``taut'' means that, through each point of the manifold, there is a closed loop transverse to the foliation at each of its points; ``weakly symplectically semi-fillable'' means that the contact structure lies at the boundary of a symplectic manifold $(W,\omega)$ with disconnected boundary, and that $\omega\vert_{\xi_n}>0$ for $n>0$ large enough; lastly, ``tight'' means that it admits no embedded \emph{overtwisted disk}, as defined in \cite{Eliashberg_ClassificationOvertwisted3D}. 
(The fact that weakly fillable contact structures are necessarily tight is due to  Eliashberg and Gromov \cite{Eliashberg_FillingByHoloDiscs}.)

In fact, \cite{EliashbergThurston_Confoliations} explore also other relationships between the geometric properties of a foliation $\cF$ and those of the approximating family of contact structures $(\xi_n)_{n\in \bN}$.
These results generally go in the direction of formalizing the idea that known rigid (i.e.\ geometric) phenomena/notions in foliation theory are ``inherited'', via contact approximations, to known rigid phenomena/notions in contact topology.

\smallskip

In order to make the bridge between foliations and contact structures, Eliashberg--Thurston introduce in their work the notion of $3$-dimensional \emph{confoliation}.
This is a plane field on a $3$-manifold which is defined by a one-form $\alpha$ satisfying $\alpha\wedge\d\alpha\geq 0$.
(Here, the plane field is assumed to be coorientable, hence $\alpha$ is defined on all of the ambient manifold.) 
Note that this algebraic condition on differential forms clearly generalizes the foliated and contact ones in this ambient dimension.
In fact, the above recalled approximation result from \cite{EliashbergThurston_Confoliations} holds more generally for confoliations.

In this context, and also of particular interest for the present work, Eliashberg and Thurston \cite{EliashbergThurston_Confoliations} proved a ``Reeb stability result'' for $3$-dimensional confoliations: on an ambient $3$-dimensional manifold, if a confoliation admits a $2$-sphere tangent to it at all points, then it must be the foliation by $2$-spheres on the ambient $3$-manifold $\bS^1\times \bS^2$.
The methods used are very low-dimensional in nature: for instance a key point is the analysis of the induced foliation on a certain (local) surface transverse to the confoliation.

\smallskip

After the pioneering work \cite{EliashbergThurston_Confoliations}, many authors have worked on elucidating both this approximation phenomenon and the topological properties of confoliations themselves; see e.g.\  \cite{Hind_FillingByHoloDiscsWeaklyPseudoConvex,Colin_TightContactToroidalAndApproxOfFol,Vogel_RigidityVersusFlexibilityTightConfol,Vogel_UniquenessApproximatingContact,Bowden_ApproximatingC0FolByCont,Bowden_ContactDeformTautFol,ColinKazezRoberts_TautFoliations,ColinHonda_FoliationsContactTheirInteractions}.

\bigskip

The focus of this work is the higher dimensional setting.
While the literature on high-dimensional confoliations is certainly much less developed than the $3$-dimensional one, there are some isolated works \cite{AltschulerWu_DeformingConfoliations,DatheRukimbira_FoliationsAndContactStructures,DatheKhoule_DeformationsFoliationsToContact} where the focus is on \textbf{smooth confoliations}, i.e.\ hyperplane fields that can be defined by $1$-forms $\alpha$ satisfying $\alpha\wedge \d\alpha^n\geq 0$, and their \textbf{contact smooth approximations and deformations}, i.e.\ families of contact structures parametrized by $\bN$ or by $(0,1]$ and converging smoothly to the smooth confoliation as the parameter goes, respectively, to $+\infty$ and to $0$.
However, these notions and results seem to not consider any symplectic-type data on the smooth confoliations, and as a consequence to not ask anything, in the case of contact smooth approximations, about the convergence (or at least an asymptotic control) of the canonical conformal symplectic structures on the contact structures as the parameter goes to $+\infty$ (or to $0$, in the case of deformations).
In this work, we follow instead the approach from \cite{CardonaGironella_ConfoliationsHighDim}, aimed at remedying this issue by equipping smooth confoliations with a certain symplectic-type data that allows for an effective study of their relations with symplectic manifolds analogously to symplectic fillability for contact structures.
We jump here straight to recalling the main notions from \cite{CardonaGironella_ConfoliationsHighDim} that we will need in this paper, and refer the reader to \cite{CardonaGironella_ConfoliationsHighDim} for a more in-depth comparison with previous works and description of the motivations behind the notion introduced there.

\medskip

The approach of \cite{CardonaGironella_ConfoliationsHighDim} is that of considering \textbf{c-symplectic confoliations}, that are smooth confoliations $\eta$ equipped with \emph{cone-symplectic structures}, \textbf{c-symplectic structures} in short.
These are defined to be a cone $\CS_{\eta,\mu}$ of non-degenerate $2$-forms on $\eta$ spanned by the \textbf{conformal partial-symplectic structure} $\CpS_\eta$ induced by the restriction $\d\beta\vert_\eta$ of the differential of a defining $1$-form $\beta$ for $\eta$, and by a non-degenerate $2$-form $\mu$ on $\eta$ whose restriction to the \textbf{characteristic foliation} $\cK_\eta:=\ker(\d\beta\vert_\eta)$ of $\eta$ is also non-degenerate.

Let us make some quick comments to slightly clarify this notion.
First, $\CpS_\eta$ is independent of the choice of defining $1$-form $\beta$ for $\eta$. 
Second, $\CpS_\eta$ is in general degenerate: for instance, if $\eta$ is a foliation $\CpS_\eta$ has kernel coinciding with the whole $\eta$.
Third, for $\CS_{\eta,\mu}$ to be a cone made entirely by non-degenerate $2$-forms, $\mu$ and $\CpS_\eta$ need to be \textbf{symplectically compatible}, i.e.\ we must have that $\mu+t\d\beta\vert_\eta$ is non-degenerate on $\eta$ for all $t\geq 0$.
(Written in this form, it is in particular clear that the definition of c-symplectic confoliation is directly inspired by that of weak symplectic filling for contact structures from \cite{MassotNiederkrugerWendl_WeakFillings}.)
Lastly, the characteristic foliation is a distribution that can be associated to any hyperplane field and not just to (smooth) confoliations; in general, it is singular, in the sense that the rank of $\ker(\d\beta\vert_\eta)$ can change from point to point, and the fact that it is a foliation simply means that it is \textbf{involutive}, i.e.\ the Lie bracket of any pair of vector fields tangent to it at all points is also tangent to it at all points.

With this clarified, we also recall that in \cite{CardonaGironella_ConfoliationsHighDim} a c-symplectic confoliation is furthermore called
\textbf{geodesible} if there is a defining $1$-form $\alpha$ for the underlying hyperplane distribution and a vector field $R$ so that $\alpha(R)=1$ and $\iota_R\d\alpha=0$,
and \textbf{tame} if there is a complex structure $J$ on $\eta$ that is tamed by $\mu$ and by $\CpS_\eta$ at the same time.
Here, ``tamed by $\mu$'' means that $\mu(v,Jv)>0$ for all non-vanishing $v\in \eta$, as is usual in symplectic topology; ``tamed by $\CpS_\eta$'' means instead that $\d\beta(v,Jv)\geq 0$, with equality if and only if $v\in \cK_\eta$.
(Recall from \cite{CardonaGironella_ConfoliationsHighDim} that tameness is in general not automatic for a c-symplectic confoliation.)

We point out that \cite{EliashbergThurston_Confoliations} already proposed a tentative notion of confoliations in higher dimensions, asking for the existence of a complex structure $J$ satisfying $\d\beta(v,Jv)\geq 0$ for all $v\in \eta$.
As explained in \cite{CardonaGironella_ConfoliationsHighDim}, the condition of being geodesible and the additional control asked by the notion of tame (c-symplectic) confoliation on the vectors in the kernel of the quadratic form given by $\d\beta(\cdot, J\cdot)$ allows, with a \textbf{weak maximum principle} (from \cite[Lemma 3.4]{CardonaGironella_ConfoliationsHighDim}, and recalled in \Cref{lem:weak_maximum_principle}) for an effective control of pseudo-holomorphic curves near the boundary in the presence of a ``symplectic filling of a confoliation''.
This control allowed for instance to obtain in \cite{CardonaGironella_ConfoliationsHighDim} a certain generalization to high-dimensional c-symplectic confoliations of the fact that symplectically fillable contact structures are tight.

There are several notions of symplectic fillability for confoliations introduced in \cite{CardonaGironella_ConfoliationsHighDim}, mimicking the distinct fillability notions for contact structures.
A symplectic manifold $(W,\Omega)$ was called a \textbf{weak symplectic filling} of a connected c-symplectic confoliated manifold $(M=\partial W,\eta,\CS_{\eta,\mu})$ if $\Omega\vert_{\cK_\eta}=f\mu\vert_{\cK_\eta}$ for some positive function $f$, $\Omega\vert_\eta$ is non-degenerate and $\Omega\vert_{\eta}$ is symplectically compatible with $\CS_{\eta,\mu}$, i.e.\ the (open) cone generated by $\Omega\vert_{\eta}$, $\mu$ and $\CpS_\eta$ is made of non-degenerate $2$-forms on $\eta$.
One talks instead of \textbf{quasi-strong symplectic filling} if instead $\Omega\vert_\eta\in \CS_{\eta,\mu}$, i.e.\ $\Omega\vert_{\eta} = \mu + f\d\beta\vert_\eta$ for some positive function $f$ and defining $1$-form $\beta$ for $\eta$ (note that here we could ask the form $\mu + \d\beta\vert_\eta$, up to replacing $\beta$ by $f\beta$).
If a preferred extension $\omega\in \Omega^2(M)$ of $\mu$ to $M$ is available, e.g.\ if $(\eta,\CS_{\eta,\mu})$ comes from a \textbf{strong symplectic foliation} $(\cF,\omega)$ (i.e.\ a corank $1$ regular foliation $\cF$ equipped with $\omega\in \Omega^2(\cF)$ so that $\d\omega = 0$ and $\omega\vert_\cF$ is non-degenerate), then we also talk about \textbf{strong symplectic filling} of $(\eta,\CS_{\eta,\mu},\omega)$ if instead $\Omega = \omega + f\d\beta$ as $2$-forms on $M$, again for some function $f>0$ and some defining $1$-form $\beta$ for $\eta$. 
Moreover, analogous notions can be defined for \textbf{semi-fillings}, i.e.\ fillings where the boundary is possibly disconnected, and for \textbf{cobordisms}, where one also allows for \emph{concave} type boundary components, i.e.\ components where the coorientation of $\eta$ and the orientation given by $\Omega\vert_\eta$ induce the opposite of the ambient orientation as boundary component of the filling.

\bigskip

With the context and notions from \cite{CardonaGironella_ConfoliationsHighDim} recalled, we are now ready to enter into the new contributions in this paper.
Our main aim is to give a high-dimensional analogue of two of the results recalled above from Eliashberg and Thurston's pioneering work, namely the fact that contact structures approximating a taut foliation in ambient dimension $3$ are weakly symplectically fillable \cite[Corollary 3.2.5]{EliashbergThurston_Confoliations}, and the Reeb stability result for $3$-dimensional confoliations \cite[Proposition 1.3.9]{EliashbergThurston_Confoliations}.

Both of these high-dimensional results rely on a rigidity property, of independent interest, for weakly (semi-)fillable tame c-symplectic confoliations admitting a closed leaf foliated by symplectic spheres: namely such confoliations have no contact regions.

We now give the precise statements of these results, and comment on each of them, starting with this last mentioned rigidity property.

\subsection{An obstruction to the presence of contact regions}
\label{sec:intro_obstruction_to_presence_contact_regions}

We work here under one of the following two hypotheses:
\begin{hyp*}[WeakSemiFill]
    $(M,\eta,\CS_{\eta,\mu})$ admits a semi-positive weak symplectic semi-filling $(W,\Omega)$.
    \label{hyp:weak_filling}
\end{hyp*}
\begin{hyp*}[CobordSHS]
    There is a semi-positive symplectic manifold $(W,\Omega)$ with weak convex boundary $(M,\eta,\CS_{\eta,\mu})$ and with strong concave boundary a stable Hamiltonian structure with Reeb vector field without contractible orbits.
    \label{hyp:cobordism_from_hypertight_shs}
\end{hyp*}

Recall that a symplectic manifold $(X^{2n},\Omega)$ is called \textbf{semi-positive} if there is no element $A\in H_2(X;\bZ)$ in the image of the Hurewicz homomorphism $h_2\colon \pi_2(X)\to H_2(X;\bZ)$ such that $\Omega(A)>0$ and $3-n\leq c_1(A)<0$, where $c_1$ is the first Chern class associated to $\Omega$.

These assumptions should not come as a surprise for experts in the field. 
Indeed, they will be used to guarantee compactness of a certain moduli space of pseudo-holomorphic spheres in the proof of the following result.

\begin{thm}
    \label{thm:rank_char_fol_at_least_two}
    Let $(\eta,\CS_{\eta,\mu})$ be a geodesible tame c-symplectic confoliation on a closed connected $(2n+1)$-dimensional manifold $M$.
    Assume that the following properties hold.
    \begin{enumerate}
        \item There is a symplectic manifold $(W,\Omega)$ as in Hypothesis \hypref{hyp:weak_filling} or \hypref{hyp:cobordism_from_hypertight_shs}.
        
        \item There is a closed submanifold $L\subset M$ of dimension $2k$, with $1\leq k\leq n$, having trivial normal bundle in $M$ and such that:
        \begin{enumerate}
            \item\label{item:rank_2k_near_L__thm_rank_at_least_two}  $TL\subset \eta\vert_L$ (and hence $TL\subset \cK_\eta\vert_L$), and near $L$ the characteristic foliation $\cK_\eta$ has trivial holonomy and rank constant equal to $2k$;
            \item\label{item:split_and_sympl_factors__thm_rank_at_least_two} $L$ is diffeomorphic to a split product $\bS^2\times X$ of the $2$-sphere and a closed connected $(2k-2)$-manifold $X$, in such a way that 
            $\mu\vert_L = \omega_{\bS^2} + \omega_X$ for an area form $\omega_{\bS^2}$ on $\bS^2$ and a symplectically aspherical symplectic form $\omega_X$ on $X$, and that the $\bS^2$-factor induces a primitive homology class in $W$.
        \end{enumerate}
    \end{enumerate} 
    Then, $\eta$ is nowhere contact, i.e.\ $\cK_\eta$ has rank at least $2$ everywhere.
\end{thm}

Note that the split factors of $L$ are symplectic with respect to $\Omega|_L$.
Indeed, by the definition of weak filling,
\[
\Omega|_{\cK_\eta}=f\,\mu|_{\cK_\eta}
\]
for some positive function $f$, while
\[
\cK_\eta|_L=TL.
\]
The explicit requirement $\mu\vert_L=\omega_{\bS^2}+\omega_X$ forces $f$ to be constant on $L$, and hence the restriction of $\Omega$ to $L$ is of the same split form.

\smallskip

The proof of \Cref{thm:rank_char_fol_at_least_two} is based on the study of a moduli space of pseudo-holomorphic spheres stemming from the $\bS^2$-factor in the leaf $L$.
More precisely, we select $1$-dimensional sub-families of such spheres evaluating along a curve starting at the model spheres in $L$ and reaching any other given point of the boundary component $M$.
The weak maximum principle and the Hypotheses \hypref{hyp:weak_filling} and \hypref{hyp:cobordism_from_hypertight_shs} allow us to argue that such a $1$-dimensional moduli space can be arranged to avoid nodal degenerations in the interior of the filling.
We then conclude that there must be a (possibly broken) pseudo-holomorphic sphere passing through the endpoint of the chosen curve, and by the maximum principle (and some further genericity argument) we conclude that such endpoint of the curve must be a point where $\cK_\eta$ has rank at least $2$.

\begin{rmk}
    \label{rmk:alternative_assumptions}
    The main purpose of \Cref{thm:rank_char_fol_at_least_two} is to illustrate that restrictions on the nature of $\cK_\eta$ can be derived using pseudo-holomorphic curve techniques.
    We make no claim that the technical assumptions in \Cref{thm:rank_char_fol_at_least_two} above are optimal in any sense, as there are other situations in which a similar conclusion can be drawn.
    For instance, the assumption on symplectic asphericity of $\omega_X$ can be substituted with the fact that the area of the factor $\bS^2$ is minimal among all spheres with positive symplectic area in $L$.
    On the other hand, hypothesis \hypref{hyp:weak_filling} or \hypref{hyp:cobordism_from_hypertight_shs} are of a more fundamental nature for the methods of proof we employ.
\end{rmk}

\subsection{A Reeb-stability result for high-dimensional confoliations}
\label{sec:intro_reeb_stability}

Before stating the aforementioned Reeb stability type result, we give the main technical statement behind it, which is a refinement of \Cref{thm:rank_char_fol_at_least_two} that is of independent interest:

\begin{thm}
    \label{thm:rank_char_fol_at_least_four}
    Let $(\eta,\CS_{\eta,\mu})$ be a geodesible tame c-symplectic confoliation on a closed connected $(2n+1)$-dimensional manifold $M$.
    Assume that the following properties hold.
    \begin{enumerate}
        \item There is a symplectic manifold $(W,\Omega)$ as in Hypothesis \hypref{hyp:weak_filling} or \hypref{hyp:cobordism_from_hypertight_shs}.
        
        \item There is a closed submanifold $L\subset M$ of dimension $2k$, with $2\leq k\leq n$,  having trivial normal bundle in $M$ and such that
        \begin{enumerate}
        
            \item\label{item:rank_2k_near_L__thm_rank_at_least_four}  $TL\subset \eta\vert_L$ (and hence $TL\subset \cK_\eta\vert_L$), and near $L$ the characteristic foliation $\cK_\eta$ has trivial holonomy and rank constant equal to $2k$;
            
            \item\label{item:split_and_sympl_factors_and_min_energy__thm_rank_at_least_four} $L$ is diffeomorphic to a split product $\bS^2\times \bS^2\times X$, where $\bS^2$ is the $2$-sphere and $X$ is a closed connected $(2k-4)$-manifold, in such a way that 
            \begin{enumerate}
                \item\label{item:splitting_mu__thm_rank_at_least_four} 
                $\mu\vert_L = \omega_{\bS^2} + \omega_{\bS^2} + \omega_X$ for an area form $\omega_{\bS^2}$ on $\bS^2$ and a symplectically aspherical symplectic form $\omega_X$ on $X$ and $\omega_{\bS^2}(\bS^2)$ is the minimal $\Omega$-energy of a spherical class in $M$,

                \item\label{item:homology_classes__thm_rank_at_least_four} and the two $\bS^2$ factors define different and primitive homology classes in $W$.
            \end{enumerate} 

        \end{enumerate}
    \end{enumerate} 
    Then, $\cK_\eta$ has rank at least $4$ everywhere.
\end{thm}

\begin{rmk}
    \label{rmk:cannot_prove_rank_at_least_6}
    While it is natural to expect that the rank of $\cK_\eta$ should be at least $2k$ if the leaf $L$ has $k$ factors given by symplectic $2$-spheres, the arguments in the proof given in later sections unfortunately are not enough to conclude this.
    Roughly, the way we force the rank to be at least $4$ at most points is to prove that through (almost) every point pass two pseudo-holomorphic spheres, tangent to $\cK_\eta$ at all points by the weak maximum principle, and with $\bC$-linearly independent tangent space at the intersection point.
\end{rmk}

We now specialize to the case where the dimension of the ambient manifold $M$ is $5$.
Let us make some comments about the content of \Cref{thm:rank_char_fol_at_least_four} in this case.

First, in this case necessarily $k=2$ and $L$ is asked to be diffeomorphic to a product of two copies of the $2$-sphere $\bS^2$.
Moreover, in the same spirit as \Cref{rmk:alternative_assumptions}, the Hypothesis 
\ref{item:splitting_mu__thm_rank_at_least_four}
can be replaced by
\begin{enumerate}[
  label=2(b)i-bis.,
  labelwidth=!,
  labelsep=0.5em,
  leftmargin=!,
  align=left,
  labelindent=25pt
]
\item
\label{item:splitting_mu_bis__thm_rank_at_least_four} 
$\mu\vert_L = \omega_{\bS^2} + \omega_{\bS^2}$ for an area form $\omega_{\bS^2}$ on $\bS^2$.
\end{enumerate}
Note that $X=\{pt\}$ is always symplectically aspherical in this ambient dimension.
Hence, the only difference with respect to \ref{item:splitting_mu__thm_rank_at_least_four} is the fact that we allow here for homology classes of lower $\Omega$-energy than $\omega_{\bS^2}(\bS^2)$ in $M\setminus L$.
 
The proof in this case can be achieved (as explained in \Cref{sec:proof_rank_at_least_four_alternative_assumption}) by better leveraging Hypothesis \ref{item:homology_classes__thm_rank_at_least_four} and using the following general fact, which is of independent interest and generalizes the $3$-dimensional statement in \cite[Proposition 1.2.12]{EliashbergThurston_Confoliations}:

\begin{prop}
\label{prop:fully_foliated_part_closed}
    The following subset of $M$, called the \emph{fully foliated part} of $\eta$, is closed:
    \[
    \Fol(\eta) := \{\, p\in M \mid \text{ $\nexists$ arc $\gamma$ tangent to $\eta$, from $p$ to a point where $\rk\cK_\eta$ is not maximal} \, \} \, . 
    \]
\end{prop}

The second comment to make concerning \Cref{thm:rank_char_fol_at_least_four} in the ambient $5$-dimensional case is that Hypothesis \ref{item:rank_2k_near_L__thm_rank_at_least_four} readily implies, by the local Reeb stability theorem for foliations (see e.g.\ \cite[Theorem 2.4.1]{CandelConlon_FoliationsIBook}) that, near $L$, the $4$-dimensional leaves of $\eta$ are all diffeomorphic to $L$.
Likewise, the following is an immediate consequence of the global version of the Reeb stability theorem for foliations (see e.g.\ \cite[Theorem 6.1.5]{CandelConlon_FoliationsIBook}) and of the conclusion of \Cref{thm:rank_char_fol_at_least_four}:
\begin{cor}
\label{cor:reeb_stab_confoliations}
    Under the same assumptions as \Cref{thm:rank_char_fol_at_least_four}, with moreover $M$ of dimension $5$, one has that
    $M$ fibers over $\bS^1$ and $\eta=\cK_\eta$ is a strong symplectic foliation with leaves given by the fibers.
\end{cor}

    It is important to point out that, while the same argument used in the proof of \Cref{cor:reeb_stab_confoliations} would give an analogous statement also in ambient dimension $3$, such a result would be considerably weaker than the Reeb stability result for $3$-dimensional confoliations proven by Eliashberg and Thurston in \cite[Proposition 1.3.9]{EliashbergThurston_Confoliations}.
    This is not an accident, and is due to the fact that, contrary to \cite{EliashbergThurston_Confoliations}, we use pseudo-holomorphic curve techniques, which unfortunately require some more strict assumptions, namely Hypothesis \hypref{hyp:weak_filling} or \hypref{hyp:cobordism_from_hypertight_shs}, on the c-symplectic confoliation.
    It is unclear to the authors whether a more direct adaptation of the low-dimensional approach in the style of \cite{EliashbergThurston_Confoliations} is possible in our higher dimensional setting; the main underlying complication is the fact that, as is typical in symplectic topology, manipulating symplectic forms in ambient dimension at least $4$ is harder than on surfaces.

We also mention that there is a Reeb stability type result for symplectic foliations in the literature, due to Crainic and Mărcuț \cite{CrainicMarcut_ReebThurstonStabilitySymplFol}, which is of a more geometric nature and aimed instead at describing when a compact leaf admits a tubular neighborhood where the leafwise symplectic form also admits a normal form.

\subsection{Contact approximations and obstructions thereof}
\label{sec:intro_contact_approximations_and_obstructions}

The notion of contact smooth approximation studied in \cite{AltschulerWu_DeformingConfoliations,DatheRukimbira_FoliationsAndContactStructures,DatheKhoule_DeformationsFoliationsToContact} obviously doesn't give any control on symplectic properties, such as weak symplectic fillability, for the approximating contact structures.
The motivation for this part of the paper is exactly to give a notion of contact approximation of confoliations which allows to recover the stability of weak fillability under contact approximations of confoliations in dimension $3$ from \cite[Corollary 3.2.5]{EliashbergThurston_Confoliations} mentioned above.

We should first of all mention that in \cite[Chapter 2]{Toussaint_PhDThesis} there is an in-depth study of a notion of contact approximations $(\xi_m)_{m\in \bN}$ of conformal symplectic \emph{foliations} $(\cF,\CS_{\mu})$ (and not general confoliations), where $\mu$ is a non-degenerate $2$-form on $\cF$.
There, the conformal symplectic structures $\CS_{\xi_m}$ on an approximating $\xi_m$ are asked to converge to $\CS_\mu$ on $\cF$.
The advantage of this notion is that the obstruction to contact approximations in ambient dimension $3$ from \cite{EliashbergThurston_Confoliations} can be ``stabilized'' to higher dimensions: namely, if $(M,\cF,\CS_\mu)$ has as $3$-dimensional submanifold $(\bS^1\times \bS^2,\CS_{\omega_{\bS^2}})$, then it cannot be contact approximated in this sense.
In \cite[Section 2.5]{Toussaint_PhDThesis} there is also an obstruction, due to Presas--Toussaint
and for this same notion of contact approximation, of a different nature than the $3$-dimensional one.

Furthermore, in \cite{CardonaGironella_ConfoliationsHighDim} the first author and Cardona already introduced a notion of contact approximations for c-symplectic confoliations.
As discussed in \cite[Section 7]{CardonaGironella_ConfoliationsHighDim}, besides the fact that this notion is also defined for confoliations and not just foliations, it is distinct from the one in \cite{Toussaint_PhDThesis} as it asks for a certain type of convergence only for the ``dominating term'' (as $m\to\infty$) of $\CS_{\xi_m}$.
On the other hand, this notion of contact approximation is unfortunately not strong enough to guarantee the desired stability of weak fillability (c.f.\ \Cref{exa:old_notion_approx_not_enough_for_weak_fill}).

In this paper we introduce two new notions of \textbf{contact approximation (and deformations)} of c-symplectic confoliations, namely \textbf{of strict and tame type}, both of which recover stability of weak fillability, as we state below.

The strict version is in spirit similar to the notion in \cite{CardonaGironella_ConfoliationsHighDim}. 
In particular, its definition is somewhat involved and requires checking the behavior of certain wedge products of differential forms involving contact forms defining $\xi_m$ and $\eta$, as well as (the restriction to $\cK_\eta$ of) the form $\mu$ directing the c-symplectic structure $\CS_{\eta,\mu}$ on $\eta$.
On the other hand, as we state below, all the interesting examples of contact smooth approximations of smooth confoliations from the literature analyzed in \cite{CardonaGironella_ConfoliationsHighDim} also satisfy this notion of contact strict approximation.

The notion of contact tame approximation, on the other hand, has a much more concise definition\footnote{This definition was already discussed by the first author with Cardona in the context of the work \cite{CardonaGironella_ConfoliationsHighDim}, where it hasn't been included due to the lack of examples and, previously, also of applications such as \Cref{thm:weak_fillability_stable_under_tame_contact_approx}. The authors are grateful to Robert Cardona for letting them include this definition in the present paper.}, based on convergence of a sequence of tamed complex structures on $\xi_m$ to one on $\eta$ tamed by $\CS_{\eta,\mu}$ (c.f.\ \Cref{def:contact_tame_approx_and_deform}), but unfortunately it is much harder to check in actual concrete and explicit cases, and it seems to be satisfied by less examples in the literature.

\medskip

We resume some notable examples of these two notions of contact approximations in the following proposition:
\begin{prop}
    \label{prop:examples_contact_approx_and_deform}
    \begin{enumerate}
    
    \item \label{item:examples_contact_approx_and_deform__fiber_sum_branched_cover} 
    Contact branched covers and contact fibered sums defined in \cite{Geiges_BranchedCoversFiberSums} are contact strict  $C^\infty$-deformations of a (strict) c-symplectic confoliation on the underlying smooth branched covers and fibered sums.
    
    \item \label{item:examples_contact_approx_and_deform__bourgeois}
    Let $\xi$ be a contact structure on $M^{2n-1}$.
    Then, Bourgeois contact structures \cite{Bourgeois_ToriAreContact} on $M\times \bT^2$, associated to any choice of supporting open book for $\xi$, are:
    \begin{itemize}
        \item contact strict $C^\infty$-deformations of the (strict) c-symplectic confoliation $(\eta=\xi \oplus T\bT^2,\CS_{\eta,\omega})$ on $M\times \bT^2$, where $\omega$ is any area form\footnote{To be precise, we defined above $\CS_{\eta,\mu}$ for any $\mu$ non-degenerate on $\eta$. Here, $\omega$ is non-degenerate only on $T\bT^2$, but adding any arbitrary multiple of $\d\beta$, with $\beta$ any one-form defining $\eta$, indeed makes it non-degenerate on all of $\eta$; hence, $\CS_{\eta,\omega}=\{\omega + f \d\beta \mid f\colon M \to \bR_{>0} \}$ is still a set of two-forms all of which are non-degenerate on $\eta$.} on $\bT^2$,
        \item and contact $C^\infty$-deformations, of both strict and tame type, of a (strict and tame) c-symplectic confoliation ``adapted'' to the natural ``spinal open book'' of the Bourgeois manifold.
    \end{itemize}

    \item \label{item:examples_contact_approx_and_deform__mori}  
    Mori's $(0,1]$-family of contact structures \cite{Mori_NoteMitsumatsuConstruction} on $\bS^5$ is a contact strict $C^\infty$-deformation of Mitsumatsu's symplectic foliation \cite{Mitsumatsu_SymplecticLawsonFoliation}, seen as a (strict) c-symplectic confoliation.

    \item \label{item:examples_contact_approx_and_deform__bertelson_meigniez} 
    The linear deformations (in the space of hyperplane fields) to contact structures of locally conformal symplectic foliations $(\cF,\omega)$ considered in \cite{BertelsonMeigniez_ConformalSymplecticFoliations}
    are contact $C^\infty$-deformations, of both strict and tame types, of $(\cF,\omega)$, seen as (strict and tame) c-symplectic confoliations.

    \item    \label{item:examples_contact_approx_and_deform__giroux_torsion}
    There is a (strict) c-symplectic confoliation of the form $(\eta,\CS_{\eta,\mu})$, where $\eta=\ker(\d x)$ with $(x,y)$ coordinates on $\bT^2$, such that each of the contact structures $\xi_k$ on the manifold of the form $M\times\bT^2$ constructed as in \cite[Theorem A]{MassotNiederkrugerWendl_WeakFillings} is a contact strict $C^\infty$-deformation of $(\eta,\CS_{\eta,\mu})$.
    \end{enumerate}
\end{prop}

For the newly introduced notions of contact strict and tame approximations we prove the following results:
\begin{thm}
\label{thm:weak_fillability_stable_under_tame_contact_approx}
    Let $M$ be a closed $(2n+1)$-dimensional manifold. 
    On it, consider a tame c-symplectic confoliation $(\eta,\CS_{\eta,\mu})$, as well as a sequence $(\xi_m)_{m\in \bN}$ of contact structures on $M$ that $C^1$-approximates the former in the tame sense. 
    Then, any weak symplectic filling $(W,\Omega)$ of $(M,\eta,\CS_{\eta,\mu})$ is, after attaching a long cylindrical end, also a weak symplectic filling of $(M,\xi_m)$, for any $m$ large enough.
\end{thm}

\begin{thm}
\label{thm:weak_fillability_stable_under_strict_contact_approx}
    Let $M$ be a closed $(2n+1)$-dimensional manifold. 
    On it, consider a c-symplectic confoliation $(\eta,\CS_{\eta,\mu})$, as well as a sequence $(\xi_m)_{m\in \bN}$ of contact structures on $M$ that $C^1$-approximates the former in the strict sense. 
    Then, any quasi-strong symplectic filling $(W,\Omega)$ of $(M,\eta,\CS_{\eta,\mu})$ is also a weak symplectic filling of $(M,\xi_m)$, for any $m$ large enough.

    \noindent
    Moreover, the same holds for any \emph{strict} weak symplectic filling $(W,\Omega)$ of $(\eta,\CS_{\eta,\mu})$, after attaching a long cylindrical end.
\end{thm}

The above statements are honest generalizations of the $3$-dimensional case, in the sense that they recover completely \cite[Corollary 3.2.5]{EliashbergThurston_Confoliations}.
While their proof is not technically hard, the difficulty behind these generalization lied mainly in pinning down the correct notions of contact approximation for c-symplectic confoliations that allowed to recover such a result.

\medskip

Lastly, the following obstruction to contact approximations then follows directly from \Cref{thm:rank_char_fol_at_least_two,thm:weak_fillability_stable_under_tame_contact_approx,thm:weak_fillability_stable_under_strict_contact_approx} above, using the fact that a strong symplectic foliation always admits a (smoothly trivial) strong symplectic semi-filling.
(After the $3$-dimensional \cite[Proposition 3.2.2]{EliashbergThurston_Confoliations}, this last mentioned fact is folklore, and proved in exactly the same way, in all ambient odd dimensions. This was already proven explicitly in \cite[Lemma 2.7.6]{Toussaint_PhDThesis}, and we include in this paper as well a short proof as \Cref{lem:strong_sympl_fol_are_quasi-strong_sympl_semi-fillable} for completeness.)

\begin{cor}
    \label{cor:no_contact_approximations_product_foliation}
    Let $(X,\omega_X)$ be a closed, symplectically-aspherical symplectic manifold.
    Consider $M=\bS^1_\theta\times \bS^2\times X$ equipped with the strong symplectic foliation $(\cF = \ker(\d\theta), \omega = \omega_{\bS^2}+\omega_X) $, where $\omega_{\bS^2}$ is any area form on $\bS^2$.
    Then, $(\cF,\omega)$ cannot be $C^1$-approximated by contact structures, in either the tame or the strict sense.
\end{cor}

\begin{rmk}
    \label{rmk:C1_regularity_for_contact_approximations}
    In the above statement, and more generally in all the discussions about strict and tame approximations by contact structures in this paper, we stick to the case where the underlying smooth hyperplane fields converge in norm at least $C^1$.
    While the $C^0$ case can be treated analogously to \cite[Item 2 in Definition 7.5]{CardonaGironella_ConfoliationsHighDim} in order to obtain stability of weak fillability as in \Cref{thm:weak_fillability_stable_under_strict_contact_approx,thm:weak_fillability_stable_under_tame_contact_approx} and a resulting obstruction to $C^0$-approximations in \Cref{cor:no_contact_approximations_product_foliation}, we stick to norm $C^k$ for $k\geq 1$ as there is currently no natural class of examples of contact $C^0$-approximations (that are not $C^1$-approximations) of confoliations in dimension at least $5$.
\end{rmk}

\begin{rmk}
    \label{rmk:no_contact_approx_prod_fol_holds_possibly_more_generally}
    As it will be clear from its proof, \Cref{cor:no_contact_approximations_product_foliation} holds more generally for any notion of contact approximation such that an analogue of \Cref{thm:weak_fillability_stable_under_tame_contact_approx,thm:weak_fillability_stable_under_strict_contact_approx} holds, i.e.\ such that weak fillability of the c-symplectic confoliation is inherited by the approximating contact structures.
\end{rmk}

\paragraph*{Outline.}
\Cref{sec:preliminaries} starts with recalls on basic properties of c-symplectic confoliations from \cite{CardonaGironella_ConfoliationsHighDim}, and continues with the proof of \Cref{prop:fully_foliated_part_closed} about the fully foliated subset, then introduces the notion of ``strict'' c-symplectic confoliation, and lastly describes the natural symplectic weak semi-filling for strong symplectic foliations.
\Cref{sec:single_sphere_family} introduces the moduli space of pseudo-holomorphic spheres, together with its properties, necessary to the proof of \Cref{thm:rank_char_fol_at_least_two}, which is detailed in 
\Cref{sec:proof_rank_char_fol_at_least_two} afterwards.
Then, \Cref{sec:pseudo-hol_spheres_semi-filling} and \Cref{sec:proof} contain the description of the moduli spaces of pseudo-holomorphic spheres necessary for \Cref{thm:rank_char_fol_at_least_four} as well as the proof of the latter (including under the alternative assumption \Cref{item:splitting_mu_bis__thm_rank_at_least_four}).
Lastly, \Cref{sec:contact_approximations_weak_fillability} contains the new notions of strict and tame contact approximations, as well as the proofs of \Cref{thm:weak_fillability_stable_under_strict_contact_approx,thm:weak_fillability_stable_under_tame_contact_approx}, \Cref{prop:examples_contact_approx_and_deform} and \Cref{cor:no_contact_approximations_product_foliation}.

\paragraph*{Acknowledgements.}
The authors are grateful to Robert Cardona for letting them include the definition of contact tame approximation, originated from discussions in the context of \cite{CardonaGironella_ConfoliationsHighDim}, in the present paper.
During the preparation of this work, the authors benefited from the support of: the region Pays de la Loire, via
the project Etoile Montante 2023 SymFol; the French government’s ``Investissements
d’Avenir'' program integrated to France 2030, bearing the following reference ANR-11-LABX0020-01; the ANR Grant COSY ``New challenges in symplectic and contact topology'', bearing the reference ANR-21-CE40-0002; the ANR Grant SiSyFo ``Singular Symplectic Foliations in Contact, Symplectic and Poisson topology'', bearing the reference ANR-25-ERCS-0010; and the Basic Science Research Institute Fund, whose NRF grant number is RS-2021-NR060139.

\paragraph*{AI disclosure.}
As asked in the \href{https://info.arxiv.org/help/moderation/index.html#policy-for-authors-use-of-generative-ai-language-tools}{arXiv guidelines}, we explicitly discuss AI use in the preparation of this article. 
No AI software was used for the mathematical part of the paper, but a minor use has 
been done for a final spell-checking and to generate ad-hoc latex commands to
solve a compilation issue (involving the package cleveref) with the arXiv pdf compiler.


\section{Preliminaries}
\label{sec:preliminaries}


\subsection{High-dimensional c-symplectic confoliations}
\label{sec:high_dim_confoliations}

We have already recalled many of the important notions in \Cref{sec:intro}.
We continue here by recalling some additional details about the content of \cite{CardonaGironella_ConfoliationsHighDim} that will be needed in later sections.

\medskip

The first is a normal form near a regular leaf of $\cK_\eta$ without holonomy:
\begin{lem}
    \label{lem:trivial_holonomy}
    Let $L\subset M$ be a closed submanifold of dimension $2k$, with $TL\subset \cK_\eta$ and assume that $\cK_\eta$ is of constant rank $2k$ near $L$ and without holonomy.
    Then there is a normal neighborhood $U_0=(-\delta,\delta)_s\times B_\eta^{2(n-k)}\times L$ of $L$ in $M$ such that $\eta\vert_{U_0} = \ker(\d s + \lambda_\eta)$, where $\lambda_\eta$ is the standard Liouville form on the ball $B_\eta^{2(n-k)}$.
\end{lem}
The proof, here omitted, is an easy combination of local Reeb stability for $\cK_\eta$ (see e.g.\ \cite[Theorem 2.4.1]{CandelConlon_FoliationsIBook}) and \cite[Proposition 2.8]{CardonaGironella_ConfoliationsHighDim}.

\medskip

The second is a property of the rank of $\cK_\eta$:
\begin{lem}[{\cite[Lemma 2.10]{CardonaGironella_ConfoliationsHighDim}}]
    \label{lem:rank_function_upper_semi-continuous}
    The rank $\rk\cK_\eta$ of $\cK_\eta$ is an upper semi-continuous integer-valued function on $M$.
\end{lem}

\noindent
In other words, if $p_m\rightarrow p $ in $M$ as $m\to \infty$, then $(\rk\cK_\eta)_{p_m}\leq (\rk\cK_\eta)_p $ for $m$ big enough.

For further use, we now describe a special partition of $M$, to which we will refer as the ``rank partition'' and that was already described in \cite[Section 7]{CardonaGironella_ConfoliationsHighDim}.

Note that, in the setting of c-symplectic confoliations $(\eta,\CS_{\eta,\mu})$, where $\mu\vert_{\cK_\eta}$ is non-degenerate on $\cK_\eta$ by definition, we moreover have that $\cK_{\eta}$ must have even rank at all points.
As this will be our setup later on, we restrict to this case here as well.

For $0\leq i \leq n$, we define $U_i \subset M$ to be the locus of points $p\in M$ such that $\rk(\cK_{\eta})(p)< 2(i+1)$; or, said otherwise, $\rk(\cK_\eta)(p)\leq 2i$ by the fact that the rank is integer-valued.
Note that, due to \Cref{lem:rank_function_upper_semi-continuous}, such $U_i$'s are open subsets of $M$.
We also let $U_{-1}=\emptyset$ for notational convenience.
Note that $U_{n}=M$ and $U_{i}\supset U_{i-1}$.
We then let, for $i=0,\ldots,n$, $C_i := U_{i}\setminus U_{i-1}$; this is nothing else than the set of points where $\rk(\cK_{\eta})=2i$.
On $M$, we then consider the partition $\cP = \{C_i\mid 0\leq i \leq n\}$, which we will call \textbf{rank partition}.

\bigskip

The next property that we will need in later sections is the behavior of pseudo-holomorphic curves in the presence of a \textbf{tame weakly pluri-subharmonic function} $\psi$.
By this we mean a function $\psi\colon X \to \bR$ on an almost complex manifold $(X,J)$ such that 
\[
    -\d\d^\bC\psi(v,Jv)\geq 0 \, , 
    \text{ with equality if and only if } v \in \cK_\eta \, ,
\]
where $\eta = \ker(-\d^\bC\psi)\cap \ker (\d\psi)$, with $\d^\bC\psi:= \d\psi\circ J$.
Recall that this implies in particular that $\eta $ is $J$-invariant.
The desired weak maximum principle then states the following:

\begin{lem}[{\cite[Lemma 3.4]{CardonaGironella_ConfoliationsHighDim}}]
    \label{lem:weak_maximum_principle}
    Let $(X,J)$ be an almost complex manifold, $\psi\colon X\to\bR$ a tame weakly pluri-subharmonic function, $(\Sigma,j)$ a Riemann surface and $u\colon (\Sigma,j)\to (X,J)$ a pseudo-holomorphic map so that $\psi\circ u$ attains its maximum in an interior point of $\Sigma$.
    Then, $u$ has image contained in a level set $M$ of $\psi$.
    More precisely, for all $p\in \Sigma$, the image of $\d_p u$ is included in $\cK_{\eta}$, where $\eta:=\ker((-\d^\bC\psi)\vert_{TM})\subset TM$. 
\end{lem}

The above statement clarifies the motivation behind the notion of tame c-symplectic confoliation: namely, pseudo-holomorphic curves tangent to a tame c-symplectic confoliated boundary $(M,\eta,\CS_{\eta,\mu})$ in an interior point can then be controlled simply with topological assumptions on $\cK_\eta$, for instance the fact that there are tangent spheres at all, or no tangent sphere on which $\mu$ restricts positively.

We also recall that, while not every c-symplectic confoliation is automatically tame, as \cite[Example 2.29]{CardonaGironella_ConfoliationsHighDim} shows, there are some sufficient criteria for tameness, which roughly speaking work well if the behavior of $\cK_\eta$ near the points where it changes rank is not ``too wild''.
As this will not be needed in this paper, we simply refer the reader to \cite[Lemma 2.32]{CardonaGironella_ConfoliationsHighDim} and the discussion surrounding it.

\bigskip

The last ingredient from \cite{CardonaGironella_ConfoliationsHighDim} we will need in later sections is the following technical proposition, which allows us to interpolate between two distinct tame complex structures on a c-symplectic confoliation with interpolation supported in a region of the manifold where the characteristic foliation has constant rank.
We extrapolate here the statement of interest for us from the more abstract \cite[Proposition A.1]{CardonaGironella_ConfoliationsHighDim}:

\begin{prop}
    \label{prop:interpolation_tame_complex_structures}
    Let $(\eta,\CS_{\eta,\mu})$ be a tame c-symplectic confoliation on $M$, and let $J$ denote the tame complex structure on it.
    Assume that $U\subset M$ is an open domain such that there is a neighborhood $V$ of $\partial U$ with $\rk\cK_\eta\vert_V$ constant.
    Assume moreover that there is a second complex structure $J'$ on $\eta\vert_U$ also tamed by $\CS_{\eta,\mu}$.
    Then, there are a smaller neighborhood $V'\subset V$ of $\partial U$ and a complex structure $J''$ on $\eta$ tamed by $\CS_{\eta,\mu}$, coinciding with $J$ on the complement of $U\cup V'$ and with $J'$ on $U\setminus V'$.
\end{prop}


\subsection{The foliated subset}
\label{sec:foliated_subset}

We use the notation $C_i$ and $U_i$ for the rank partition introduced in
the previous subsection.

As already defined in the introduction, we will call
\emph{fully foliated part} of $\eta$, and denote it by
$\Fol(\eta)\subset M$, the subset of points $p\in M$ for which there is no
path $\gamma\colon [0,1]\to M$ with $\gamma(0)=p$,
$\gamma(1)\in U_{n-1}$, and
$\gamma'(t)\in \eta_{\gamma(t)}$ for all $t\in[0,1]$.

\begin{rmk}
    \label{rmk:fully-foliated_part_ET}
    The definition of fully foliated part of a $3$-dimensional confoliation given in \cite{EliashbergThurston_Confoliations} is different from the above, namely it is defined as the subset of those points through which there is a connected surface that is at all points tangent to $\eta$ and that is \emph{complete}, in the sense that there is a closed ball in $M$ centered at $p$ that intersects said surface in a set each of whose connected components is compact.
    However, it is not hard to see that the definition just given describes the same set in this ambient dimension, and in fact this alternative interpretation as set of those points which cannot be joined to the contact region via a tangent arc has already been used in the literature, see e.g.\ \cite{Vogel_RigidityVersusFlexibilityTightConfol,Vogel_UniquenessApproximatingContact}.
\end{rmk}

With this notation set, we are ready to prove \Cref{prop:fully_foliated_part_closed} from the Introduction.

\begin{proof}[Proof of \Cref{prop:fully_foliated_part_closed}]
Let $p\in M\setminus \Fol(\eta)$.
By definition, there exists a smooth curve
\[
\gamma:[0,1]\to M, \qquad \dot\gamma(t)\in \eta, \quad 
\gamma(0)=p,\ \gamma(1)=q\in U_{n-1}.
\]

Since $\gamma([0,1])$ is compact, it can be covered by finitely many coordinate charts
$\{V_j\}_{j=1}^k$ with the following property: on each $V_j$ there are local coordinates $(x,z)\in \mathbb{R}^{2n}\times \mathbb{R}$ such that 
\[
\eta=\ker\!\Bigl(dz+\sum_{i=1}^{2n} a_i(x,z)\,dx_i\Bigr),
\]
so that $\eta$ is a (local) connection, i.e.\ a horizontal subbundle, for the projection 
$\pi_j:V_j\to \mathbb{R}^{2n}$, $\pi_j(x,z)=x$.
In particular, a curve $\Gamma(t)=(x(t),z(t))$ is tangent to $\eta$
if and only if it satisfies the first–order ODE
\begin{equation}\label{eq:loc_ode}
\dot z(t) = -\sum_{i=1}^{2n} a_i\bigl(x(t),z(t)\bigr)\,\dot x_i(t).
\end{equation}

\smallskip
\emph{Step 1: One chart.}
Assume first that $\gamma([0,1])\subset V_1$.
For a point $p'$ sufficiently close to $p$ in $V_1$,
we construct an $\eta$–tangent curve $\gamma_{p'}$ starting at $p'$ and remaining close to $\gamma$.
Write $\gamma(t)=(x(t),z(t))$; by what was said above, its coordinates satisfy \Cref{eq:loc_ode}.
Similarly, we write $\gamma_{p'}(t)=(x'(t),z'(t))$ for the desired curve $\gamma_{p'}$, whose coordinates also will need to satisfy \Cref{eq:loc_ode}.

If $\pi_1(p')=\pi_1(p)$ (i.e.\ $x'(0)=x(0)$), keep the same base path $x(t)=\pi_1(\gamma(t))$
and determine $z'(t)$ uniquely by solving \eqref{eq:loc_ode} with the initial condition $z'(0)$.
By standard ODE theory (existence–uniqueness and continuous dependence; e.g.\ Picard--Lindel\"of),
a unique smooth solution exists on $[0,1]$, since $M$ is compact and the coefficients $a_i$ are smooth.

If instead $\pi_1(p')\neq \pi_1(p)$, slightly modify the base curve:
in the Euclidean chart $V_1$, set
\[
x'(t) := x(t) + \bigl(\pi_1(p') - \pi_1(p)\bigr),
\]
so that $x'(0)=\pi_1(p')$ and $x'(t)$ remains uniformly close to $x(t)$ for $p'$ near $p$.
Then solve \eqref{eq:loc_ode} for $z'(t)$ with this base path and the initial condition $z'(0)$.
Again, by standard ODE theory, a unique smooth solution exists and depends smoothly on $p'$.
In either case, $\gamma_{p'}(t)=(x'(t),z'(t))$ is $\eta$–tangent, $\gamma_{p'}(0)=p'$, and the endpoint
\[
E_1(p') := \gamma_{p'}(1)
\]
depends continuously (in fact, even smoothly) on $p'$.

\smallskip
\emph{Step 2: Several charts and gluing.}
In the general case $\gamma([0,1])\subset \bigcup_{j=1}^k V_j$,
choose a partition $0=t_0<t_1<\cdots<t_k=1$ with
$\gamma([t_{j-1},t_j])\subset V_j$
and $r_j:=\gamma(t_j)\in V_j\cap V_{j+1}$.
Applying Step~1 on $V_1$ gives, for $p'$ close to $p$, a first segment
$\gamma_{p'}|_{[0,t_1]}$ ending at $E_1(p')$ close to $r_1$.
This endpoint lies in $V_1\cap V_2$ for $p'$ near $p$ and serves as the initial condition
for the construction in $V_2$, yielding a second continuous endpoint map $E_2$.
Iterating the procedure produces a composition
\[
E := E_k\circ \cdots \circ E_2\circ E_1,
\]
which is continuous near $p$.
Therefore $E(p')$ is close to $q=\gamma(1)$ for all $p'$ in a neighborhood of $p$.

The resulting curve is a priori only piecewise smooth at the subdivision points. Using the local connection charts above, we may smooth its base path near these points and take its horizontal lift, obtaining a smooth $\eta$-tangent curve with the same initial point and endpoint arbitrarily close to $E(p')$.

\smallskip
Since $U_{n-1}$ is open, choose an open neighborhood $V_q\Subset U_{n-1}$ of $q$,
and by shrinking the neighborhood of $p$ and choosing the smoothing sufficiently small ensure that the endpoint of the resulting smooth curve lies in $V_q$.
Thus every $p'$ sufficiently close to $p$ belongs to $M\setminus \Fol(\eta)$,
showing that $\Fol(\eta)$ is closed, as claimed.
\end{proof}


\subsection{Strict c-symplectic confoliations}
\label{sec:strict_c-symplectic_confoliations}

The aim of this subsection is to introduce a special subclass of c-symplectic confoliations that will guarantee stability of weak fillability under the notion of contact approximations.

The motivation for this comes from the following observation of Ding and Geiges \cite[Proposition 3.5]{DingGeiges_ContactCircleBundles}: up to attaching a large cylindrical end to the weak filling $(W,\Omega)$ of a contact structure $\xi=\ker\alpha$ as defined in \cite{MassotNiederkrugerWendl_WeakFillings}, one can achieve a (seemingly much more strict) compatibility condition between $\CS_{\xi}$ and $\omega:=\Omega\vert_\xi$.
Namely, up to replacing $\omega$ by $\omega + T \d\alpha$, which can be achieved by attaching a long cylindrical end, the polynomial in $t$ given by 
\[
    \alpha \wedge (\omega + t\d\alpha)^n
\]
can be guaranteed to have every coefficient of the form $f \Vol$, where $\Vol$ is a volume form on $M$ and $f> 0$ everywhere.

For simplicity, a maximum-degree differential form bounded below by $f\Vol$ with $f>0$ will be called \textbf{positive}; similarly, it will be called \textbf{non-negative} if $f\geq 0$.

As recalled in the previous sections, the notion of c-symplectic confoliation from \cite{CardonaGironella_ConfoliationsHighDim} is based on the same symplectic compatibility condition  $\beta\wedge (\mu + t\d\beta)^n>0$ for all $t\geq 0$, directly inspired from \cite{MassotNiederkrugerWendl_WeakFillings}.
However, as the following example shows, the analogue of \cite[Proposition 3.5]{DingGeiges_ContactCircleBundles} in our setup is false. Namely, one cannot in general arrange, up to adding a large multiple of $\d\beta$ to $\mu$, that each coefficient of the polynomial in $t$ given by $\beta\wedge (\mu + t\d\beta)^n$ is of the form $f\Vol$ for some function $f\geq 0$ and $\Vol$ volume form on $M$.
(Note that one cannot hope for $f>0$, as there will in general be points where $\rk\cK_\eta$ is bigger than $0$, and there the polynomial will have degree strictly less than $n$.)
\begin{exa}
    \label{exa:prop_DingGeiges_false_for_confoliations}
    On $\bR^5$ with coordinates $(z,x_i,y_i)$ consider $\beta=\d z + z^2 \gamma$, with $\gamma = -x_1 \d y_1 - x_2 \d y_2 + x_1 \d x_2 - y_1 \d y_2$, and $\omega = \sum_{i=1}^2 \d x_i \wedge \d y_i$.
    Then, we have $\beta\wedge (\omega+ t\d\beta )^ 2 = dz \wedge (\omega + t z^2 \d\gamma)^2 = P(t) \d z \wedge \omega^2$, where $P(t) = 1 - 2z^2 t + 2z^4 t^2$.
    In order for the shift $P_T(t):=P(t+T)$ to have positive coefficients as a polynomial of $t$, we need in particular $P'(T)>0$. 
    However, we have $P'(T)= 2z^2(2z^2T-1)$, hence this is positive iff $T>\frac{1}{2z^2}$, which explodes to $+\infty$ as one approaches the origin of $\bR^5$.
\end{exa}

Note that it is true that, for each fixed point $p\in M$ one can find $T_p>0$ such that the shift $\omega+T\d\beta$ makes all the coefficients of the polynomial \emph{at the point $p$} strictly non-degenerate. 
However, the issue that the above example highlights is that in general $T_p$ might need to diverge to $+\infty$ as $p$ goes from a subset where $\rk\cK_\eta=2i$ asymptotically to a point where $\rk\cK_\eta=2i+2$.

Motivated by this issue, we introduce the following definition:
\begin{defn}
    \label{defn:strict_c-symplectic_confoliation}
    A c-symplectic confoliation $(\eta,\CS_{\eta,\mu})$ on $M$ is called \textbf{strict} if, for some (and hence any) choice of $\beta$ defining $\eta$, the polynomial 
    \begin{equation}
    \label{eqn:polynomial_form}
        P(t):=\beta\wedge (\mu + t \d\beta )^n 
    \end{equation}
    with coefficients in the space of $(2n+1)$-differential forms
    (which is known to evaluate to a volume form for all $t\geq 0$ by definition of c-symplectic confoliation) satisfies the following:
    there is $T_0>0$ such that, for all $t>T_0$, the derivatives
    $P'(t)$, $\ldots$, $P^{(n)}(t)$ are all non-negative (as maximum-degree differential forms, depending on $t$). 
\end{defn}

Due to the technical nature of the above definition, one might be tempted to think that strict c-symplectic confoliations are scarce. However, the opposite seems to be rather the case.
As we will see in \Cref{sec:examples_contact_strict_and_tame_approx_and_deform}, many natural examples of c-symplectic confoliations are in fact strict.
What's more, $3$-dimensional confoliations, c-symplectic foliations and also c-symplectic confoliations naturally arising from weakly filled contact structures are also always strict, according to the following remark.

\begin{rmk}
\label{rmk:3Dconf_and_fol_and_weakly-filled-contact_are_strict}
    Let $(\eta,\CS_{\eta,\mu})$ be a c-symplectic confoliation.
    If all coefficients of $P(t)$ from \eqref{eqn:polynomial_form} are non-negative, then it is not hard to see that the c-symplectic confoliation is strict.

    It follows that any $3$-dimensional confoliation in the sense of Eliashberg--Thurston \cite{EliashbergThurston_Confoliations} and any c-symplectic foliation in any ambient dimension are in particular strict as c-symplectic confoliations.
    Similarly, thanks to \cite[Proposition 3.5]{DingGeiges_ContactCircleBundles}, a weak filling $(W,\Omega)$ of a contact manifold $(M,\xi)$ naturally gives a strict c-symplectic confoliation $(\xi,\CS_{\xi,\mu})$, where $\mu=\Omega\vert_\xi$.
\end{rmk}

In fact, \Cref{rmk:3Dconf_and_fol_and_weakly-filled-contact_are_strict} is a special case of the following general principle:
\begin{lem}
\label{lem:sometimes_can_arrange_c-sympl_confol_to_be_strict}
Let $(\eta,\CS_{\eta,\mu})$ be a c-symplectic confoliation. If $\cK_\eta$ ``increases rank by at most $2$ at a time'',
then, up to shifting $\mu$ to $\mu+T\d\beta$ for a $T>0$ big enough, the polynomial $P(t)$ from \eqref{eqn:polynomial_form} can be arranged to have non-negative coefficients.
In particular, $(\eta,\CS_{\eta,\mu})$ is strict.
\end{lem}

Here, $\cK_\eta$ is said to \textbf{increase rank by at most $2$ at a time} if for any sequence of points $p_n$ of $M$ such that $p_n\to p$ and with $\rk(\cK_\eta)(p_n)=2i$, then $\rk(\cK_\eta)(p)\leq 2i+2$.
(Recall that, by upper semi-continuity, one also has that $\rk(\cK_\eta)(p)\geq 2i$.)
Note that this includes in particular the case where $\rk\cK_\eta$ is constant on all of $M$.

\begin{proof}
We prove, by reverse induction\footnote{In order to avoid confusion, we rather argue by reverse induction instead of using $n-i$ as index.} on the index $i=i_{\max},\ldots, i_{\min}$ with $i_{\max}\leq n$ and $i_{\min}\geq 0$, that there is $T_{i}>0$ such that for all $T>T_i$, the coefficients of 
\begin{equation}
\label{eqn:polynomial_form_shifted}
P_T(t):=\beta\wedge (\mu + T\d\beta+t\d\beta)^n
\end{equation}
are all non-negative on $C_i:=\{\rk\cK_\eta=2i\}$.
Then, the desired global shift making all the coefficients of \Cref{eqn:polynomial_form_shifted} non-negative will just be $T_{\mathrm{tot}}:=\max_i T_i$.

\medskip

Let $i_{\max}$ be the maximum $i$ such that there is a point $p\in M$ with $\rk\cK_\eta(p)=2i$, and $C_{i_{\max}}\subset M$ the subset of points where $\rk\cK_\eta=2i_{\max}$.
By upper semi-continuity of the rank, $C_{i_{\max}}$ is a closed subset of $M$, hence a compact.

One can compute, at all points of $C_{i_{\max}}$, that
\begin{equation}
\label{eqn:polynomial_form_shifted_computed}
\begin{split}
    \beta \wedge (\mu + T \d\beta + t \d\beta)^n & \, = \sum_{k=0}^{n-i_{\max}} \binom{n}{k} t^k \beta\wedge (\mu + T \d\beta )^{n-k}\wedge \d\beta^k \\
    & \, = \sum_{k=0}^{n-i_{\max}} \binom{n}{k} t^k \beta\wedge \left(\sum_{j=0}^{n-i_{\max}-k} \binom{n-k}{j} T^j \mu^{n-k-j} \wedge \d\beta^j\right)\wedge \d\beta^k \, .
\end{split}
\end{equation}
Hence, at points of $C_{i_{\max}}$, the coefficient of $t^k$ has a leading term in $T$ given by 
\[
\binom{n}{k}\binom{n-k}{n-i_{\max}-k}T^{n-i_{\max}-k}\beta\wedge \d\beta^{n-i_{\max}}\wedge \mu^{i_{\max}} \, ,
\]  
which is strictly positive as $\mu$ is non-degenerate on $\cK_\eta$.
Moreover, for all points in $C_{i_{\max}}$, the same $T$-degree $n-i_{\max}-k$ is non-degenerate and dominates the coefficient of the term with $t$-degree $k$.
Hence, by compactness of $C_{i_{\max}}$, there is $T_{i_{\max}}>0$ big enough so that for all $T>T_{i_{\max}}$ the coefficients of the shifted polynomial in \eqref{eqn:polynomial_form_shifted_computed} are positive.\footnote{Note that we only look at the coefficients that \emph{can} be strictly positive, i.e.\ those containing a term $\beta\wedge \d\beta^k$ that is allowed to be non-zero. Obviously, the $t^{k}$-degree coefficients, with $k>n-i_{\max}$ will be zero at points of $C_{i_{\max}}$. Hence, at the best one can achieve non-negative coefficients in \Cref{defn:strict_c-symplectic_confoliation}.}

\medskip

For the inductive step, we assume now that the property holds for $i+1$, and we need to prove that it holds for $i$.

Let us start by pointing out that the additional difficulty with respect to the base case of $i_{\max}$ is that $C_i$ is \emph{not} compact. 
In other words, despite the fact that for each given $p\in C_i$ one can definitely find a shift $T_p>0$ making all coefficients positive at that point, such $T_p$ might in general explode to $+\infty$ as $p$ approached $C_{i+1}$; this is exactly the issue pointed out in \Cref{exa:prop_DingGeiges_false_for_confoliations}.
However, the hypothesis that $\cK_\eta$ increases rank by at most $2$ at a time allows us to conclude nonetheless as follows.

We claim that there is a small neighborhood $U_{i+1}$ of $C_{i+1}\cap \overline{C_i}$ such that there is a shift $T_{U_{i+1}}>0$ which, for $T>0$, makes the coefficients of \eqref{eqn:polynomial_form_shifted} non-negative for all $p\in U_{i+1}$.
Note that each $p\in U_{i+1}\setminus C_{i+1}$ actually lies in $C_i$, by the assumption that $\cK_\eta$ increases rank by at most $2$ at a time.
Hence, proving the claim is enough to conclude the inductive step, as $C_{i}\setminus U_{i+1}$ is then closed in $M$ and hence compact, and over $C_{i}\setminus U_{i+1}$ we can hence find a shift $T_{i}'>0$ making all the coefficients non-negative, and considering $T_i:=\max(T_i',T_{U_{i+1}})$ will then work for the conclusion of the inductive step.
We hence prove the claim.

Consider a point $p\in C_i$ very near to $C_{i+1}$.
Then, at that point we can develop as follows:
\begin{equation}
\label{eqn:polynomial_form_shifted_computed_inductive_step}
\begin{split}
    \beta \wedge &(\mu + T \d\beta + t \d\beta)^n  = \sum_{k=0}^{n-i} \binom{n}{k} t^k \beta\wedge (\mu + T \d\beta )^{n-k}\wedge \d\beta^k \\
    & \, = \sum_{k=0}^{n-i} \binom{n}{k} t^k \beta\wedge \left[\sum_{j=0}^{n-i-k} \binom{n-k}{j} T^j \mu^{n-k-j} \wedge \d\beta^j\right]\wedge \d\beta^k \\
    & \, \overset{(*)}{=} \sum_{k=0}^{n-i} \binom{n}{k} t^k \beta\wedge \left[ \underbrace{\binom{n-k}{n-i-k} T^{n-i-k}\mu^{i}\wedge \d\beta^{n-i-k}}_{\Gamma_{i}} \right. \\
    & \qquad\qquad\qquad\qquad\qquad \left. + \underbrace{\sum_{j=0}^{n-i-1-k} \binom{n-k}{j} T^j \mu^{n-k-j} \wedge \d\beta^j}_{=:\Gamma_{i+1}}\right]\wedge \d\beta^k \, ,
\end{split}
\end{equation}
where $(*)$ simply comes from isolating the term with $j=n-i-k$ in the sum.
Now, we know that the term $\beta\wedge \Gamma_{i}\wedge \d\beta^k$ in the above sum is strictly positive at all points of $C_{i}$ (even though it shrinks to zero as one approaches $C_{i+1}$).
Moreover, by the inductive hypothesis, the term $\beta\wedge \Gamma_{i+1}\wedge \d\beta^k$ is non-negative for all $T>T_{i+1}$ and for all points of a sufficiently small neighborhood, that we denote indeed $U_{i+1}$ as in the claim, of $C_{i+1}\cap \overline{C_i}$.
In particular, there is $T_{U_{i+1}}$ such that for all $T>T_{U_{i+1}}$, the $T$-shifted polynomial \eqref{eqn:polynomial_form_shifted_computed_inductive_step} has all coefficients (as polynomial in $t$) that are non-negative, as desired in the claim.

As already explained, this concludes the inductive step, and hence the whole lemma.
\end{proof}

In the opposite direction to \Cref{lem:sometimes_can_arrange_c-sympl_confol_to_be_strict}, we have the following (almost tautological) analogue of \cite[Proposition 3.5]{DingGeiges_ContactCircleBundles}:
\begin{lem}
    \label{lem:strict_c-symplectic_confoliation_positive_coefficients_polynomial}
    Let $(\eta,\CS_{\eta,\mu})$ be a strict c-symplectic confoliation.
    Then, up to shifting $\mu$ to $\mu+T\d\beta$, the coefficients of the polynomial $P(t)$ from \eqref{eqn:polynomial_form} can be made non-negative.
\end{lem}
\begin{proof}
    This simply follows from the identity
    \[
        \beta \wedge (\mu + T \d\beta + t \d\beta)^n = P(t+T)= \sum_{i=0}^n \frac{P^{(i)}(T)}{i!}t^i
    \]
    and the definition of strict c-symplectic confoliation.
\end{proof}

\bigskip

A similar issue to the one above is present in the context of weak fillings of c-symplectic confoliations.
Namely, even though it is true according to \cite[Proposition 3.5]{DingGeiges_ContactCircleBundles} that for weak fillings $(W,\Omega)$ of contact manifolds $(M,\xi)$ one can always arrange, for any defining form $\alpha$ for $\xi$ and for $T>0$ big enough that 
\[
    \alpha\wedge (\Omega\vert_{\xi} + T\d\alpha + t \d\alpha)^n>0 \text{ for all }t\geq 0 \, ,
\]
this is not always true for weak fillings $(W,\Omega)$ of a c-symplectic confoliation $(\eta,\CS_{\eta,\mu})$, due to the same phenomenon discussed above.
(Recall that, according to the definition of weak fillability, we only know that $\Omega\vert_{\cK_\eta}=f\mu\vert_{\cK_\eta}$ for some positive function $f$ on $M$.)

This motivates the following definition, analogous to \Cref{defn:strict_c-symplectic_confoliation}:
\begin{defn}
    \label{def:strict_weak_filling}
    Let $(\eta,\CS_{\eta,\mu})$ be a c-symplectic confoliation.
    A weak filling $(W,\Omega)$ of $(\eta,\CS_{\eta,\mu})$ is said to be \textbf{strict} if, for an (and hence every) auxiliary defining $1$-form $\beta$ for $\eta$,
    \begin{equation}
    \label{eqn:polynomial_form_weak_filling}
        P_\Omega(t):= \beta\wedge (\Omega\vert_\eta + t\d\beta)^n 
    \end{equation}
    has derivatives $P_\Omega'(T)$, $\ldots$, $P^{(n)}_\Omega(T)$ strictly positive for $T>0$ large enough.
\end{defn}
Note that any weak filling of a $3$-dimensional confoliation, in the sense of \cite{EliashbergThurston_Confoliations} is also a strict weak filling.

Analogously to \Cref{lem:strict_c-symplectic_confoliation_positive_coefficients_polynomial} we have the following:
\begin{lem}
    \label{lem:strict_weak_filling_c-symplectic_confoliation_positive_coefficients_polynomial}
    Let $(\eta,\CS_{\eta,\mu})$ be a c-symplectic confoliation and $(W,\Omega)$ a strict weak filling.
    Then, up to shifting $\Omega\vert_\eta$ to $\Omega\vert_\eta+T\d\beta$, which amounts to attaching a long cylindrical end in terms of the filling, the coefficients of the polynomial $P_\Omega(t)$ from \eqref{eqn:polynomial_form_weak_filling} can be made non-negative.
\end{lem}


\subsection{Strong symplectic foliations are symplectically semi-fillable}
\label{sec:strong_sympl_fol_are_sympl_semi-fillable}

Recall that a \textbf{symplectic foliation} $(\cF,\omega)$ is  a smooth regular foliation $\cF$ equipped with a leafwise $2$-form $\omega\in \Gamma(\Lambda^2\cF^*)$ whose restriction to each leaf is symplectic.
A symplectic foliation is moreover called \textbf{strong} if $\omega$ admits an extension to a $2$-form $\omega_M$ on $M$ such that $\d\omega_M(X,Y,Z) = 0$ for all triples of (local) vector fields on $M$ so that at least two among them are tangent to $\cF$.

If the codimension of $\cF$ is one, a symplectic foliation $(\cF,\omega)$ is in particular a c-symplectic confoliation.
Moreover, in this case strongness amounts to asking that the extension $\omega_M$ is closed as a $2$-form on $M$.

\medskip 

With these notions recalled, we can pass to the main aim of this section, which is to state and prove the following easy lemma, that directly generalizes the $3$-dimensional case described in \cite[Corollary 3.2.5]{EliashbergThurston_Confoliations}.

\begin{lem}
    \label{lem:strong_sympl_fol_are_quasi-strong_sympl_semi-fillable}
    Any codimension $1$ strong symplectic foliation $(\cF,\omega)$ on a closed manifold $M$ is strong symplectic semi-fillable.
    More precisely, there is a smoothly trivial symplectic manifold $(W,\Omega)$ with $W=[1,2]\times M$, that is a strong symplectic filling of 
    \[
        (\overline{M} , \overline{\cF},\omega) \sqcup( M , \cF,\omega) \, .
    \]
\end{lem}
Here, $\overline{M}$ denotes $M$ with opposite orientation, and $\overline{\cF}$ denotes the smooth foliation $\cF$ but intended as equipped with the opposite coorientation.
\begin{proof}
    Let $\beta$ be any defining $1$-form for $\cF$, and consider the piece of symplectization 
    \[
        (W:=[1,2]\times M, \Omega := \d(t\beta) + \omega_M ) \, .
    \]
    Then, $(W,\Omega)$ naturally has a strong confoliated type boundary at $\{2\}\times M$, with respect to $(\cF,\omega)$ intended as a c-symplectic confoliation.
    For the other boundary component, note that one has a natural orientation preserving isomorphism as follows:
    \begin{align*}
    ((-1-\epsilon,-1]_s\times \overline{M}, \overline{\cF}, \Omega':=\d(s(-\beta)) + \omega_M)  &\to ([1,1+\epsilon)_t\times \overline{M}, \cF, \Omega':=\d(t\beta) + \omega_M) \, , \\
        (s,p)& \mapsto (t=-s,p) \, .
    \end{align*}
    In other words, $(\overline{M},\overline{\cF},\omega)$ is the other strong confoliated type boundary component of $(W,\Omega)$, as desired.
\end{proof}


\section{A family of pseudo-holomorphic spheres}
\label{sec:single_sphere_family}

In this section, we construct the pseudo-holomorphic sphere family needed
for the proof of \Cref{thm:rank_char_fol_at_least_two}. The construction is
based on the distinguished \(\bS^2\)-factor in
\[
L\cong \bS^2\times X.
\]

Throughout the section, \((W,\Omega)\) denotes either the weak symplectic
semi-filling or the symplectic cobordism appearing in the assumptions 
\hypref{hyp:weak_filling} or \hypref{hyp:cobordism_from_hypertight_shs} of
\Cref{thm:rank_char_fol_at_least_two}.

In \Cref{sec:model_single_family}, we introduce a product local model near
\(L\), define the corresponding family of model spheres, and establish their
regularity and local uniqueness properties. In
\Cref{sec:moduli_single_family}, we extend the local almost complex structure
to \(W\), introduce the associated marked moduli space and its evaluation
map, and study its compactness properties and limit set.


\subsection{The local model and the model spheres near \(L\)}
\label{sec:model_single_family}

Recall from \Cref{thm:rank_char_fol_at_least_two} that $M$ is $(2n+1)$-dimensional and 
\(\cK_\eta\) admits a closed $2k$-dimensional leaf
\[
L = \bS^2\times X
\]
where \(X\) is a closed connected \((2k-2)\)-dimensional manifold. 
In particular, as explained after the statement of \Cref{thm:rank_char_fol_at_least_two},
the restriction of the symplectic form $\Omega$ of the filling or cobordism as in \hypref{hyp:weak_filling} or \hypref{hyp:cobordism_from_hypertight_shs} to \(L\) is
non-degenerate.

We denote by
\[
E:=\eta|_L/TL
\]
the quotient normal bundle of \(L\) inside \(\eta\). Since
\(TL=\cK_\eta|_L\) and $\mu$ is non-degenerate on $\cK_\eta$, 
$\mu$ on \(\eta\) descends to a
fiberwise symplectic form on \(E\). 
Thus \(E\) is a symplectic vector bundle
of rank \(2(n-k)\).  

According to the assumptions of \Cref{thm:rank_char_fol_at_least_two}, namely \Cref{item:rank_2k_near_L__thm_rank_at_least_two} and the fact that $L$ has trivial normal bundle in $M$, 
$E$ is trivial and using \Cref{lem:trivial_holonomy} (and the standard normal neighborhood for confoliated hypersurfaces in symplectic manifolds from \cite[Lemma 2.46]{CardonaGironella_ConfoliationsHighDim}) we can find
a neighborhood \(U_0\) of \(L\) in \(W\) of the form
\[
U_0=
(-\epsilon,0]_t
\times
(-\delta,\delta)_s
\times
B^{2(n-k)}_v
\times
\bS^2
\times
X,
\]
whose closure $\overline{U_0}$ in $W$ is diffeomorphic to $[-\epsilon,0]_t\times[-\delta,\delta]_s\times \overline{B}_v^{2(n-k)}\times \bS^2\times X$, and
so that
\[
M\cap \overline{U_0}=\{t=0\}, \quad \eta\vert_{M \cap \overline{U_0}}=\ker(\d s +\lambda_v) \,
\]
where $\lambda_v$ is the standard Liouville form on $B_v^{2(n-k)}$ and
\[
L\simeq
\{t=0,\ s=0,\ v=0\}
\subset U_0 \, .
\]
Moreover, along \(L\), the splitting
\[
TW|_L
=
\mathbb R\langle\partial_t\rangle
\oplus
\mathbb R\langle\partial_s\rangle
\oplus
\mathbb R^{2(n-k)}_v
\oplus
T\bS^2
\oplus
TX
\]
is compatible with \(\Omega\), the \(v\)-directions identify with the
symplectic normal bundle
\[
E=\eta|_L/TL,
\]
and, with respect to this splitting,
\[
\Omega|_{TW|_L}
=
dt\wedge ds
+
\omega_v
+
\omega_{\bS^2}
+
\omega_X.
\]
Here \(\omega_v\) is the standard symplectic form on
\(\bR^{2(n-k)}_v\), while \(\omega_{\bS^2}\) and \(\omega_X\) are
symplectic forms on \(\bS^2\) and \(X\), respectively.

\smallskip

We now choose \(J_0\) on \(\overline{U_0}\) to be the natural product almost complex structure
determined by the above split form of $\overline{U_0}$. 
More explicitly, we require
\[
J_0\partial_t=\partial_s,
\]
take the standard complex structure on the \(v\)-directions, and choose on the \(L\)-factor a product almost complex structure
\[
J_0|_{TL}=i_{\bS^2}\oplus J_X.
\]
Here \(i_{\bS^2}\) is the standard complex structure on
\(\bS^2\), and \(J_X\) is an \(\omega_X\)-tame almost complex
structure on \(X\).  In particular,
\[
J_0(T\bS^2)=T\bS^2,
\qquad
J_0(TX)=TX.
\]
Note that, up to shrinking $\delta$, $J_0$ is tamed by $\CS_{\eta,\mu}$. 
Moreover, up to shrinking $\epsilon$ (and after attaching a long cylindrical end as in \cite[Lemma 2.47]{CardonaGironella_ConfoliationsHighDim}), $J_0$ is also tamed by $\Omega$.

For each
\[
(t_0,s_0,v_0,x_0)
\in
[-\epsilon,0]
\times
[-\delta,\delta]
\times
\overline{B}^{2(n-k)}_v
\times
X\, ,
\]
we define
\[
u_{t_0,s_0,v_0,x_0}
\colon
\bS^2\to \overline{U_0}
\]
by
\[
u_{t_0,s_0,v_0,x_0}(p)
=
(t_0,s_0,v_0,p,x_0).
\]
We call these curves \emph{model spheres}.  By construction of \(J_0\), they
are \(J_0\)-holomorphic.

We denote by
\[
A\in H_2(W;\mathbb Z)
\]
the homology class represented by the model spheres.

The parameter space of the model family has dimension
\[
1+1+2(n-k)+(2k-2)=2n.
\]
Thus the unparametrized model family has dimension \(2n\).  After adding one
marked point, the corresponding marked model family has dimension
\[
2n+2=\dim W.
\]

\begin{lem}
\label{lem:model_spheres_regular_single_family}
Each model sphere \(u_{t_0,s_0,v_0,x_0}\) is Fredholm regular.  
\end{lem}

\begin{proof}
Let
\[
u=u_{t_0,s_0,v_0,x_0}
\]
be a model sphere.  By the product form of the local model and by the choice
of \(J_0\), the pullback bundle splits as a complex vector bundle as
\[
u^*TW
\cong
T\bS^2
\oplus
\bC
\oplus
\bC^{\,n-k}
\oplus
\bC^{\,k-1}.
\]
Here \(T\bS^2\) is the tangent bundle to the image of the model sphere,
\(\bC\) corresponds to the \((t,s)\)-directions,
\(\bC^{\,n-k}\) corresponds to the normal directions of \(L\)
inside \(\eta\), and \(\bC^{\,k-1}\) corresponds to the
\(X\)-directions.

With respect to this splitting, the linearized Cauchy--Riemann operator
splits as
\[
D_u
=
D_{T\bS^2}
\oplus
\bar\partial_{\bC}
\oplus
\bar\partial_{\bC^{\,n-k}}
\oplus
\bar\partial_{\bC^{\,k-1}}.
\]
The first summand is the standard linearized Cauchy--Riemann operator for the
identity map of \(\bS^2\), and the remaining summands are the standard
Cauchy--Riemann operators on trivial complex vector bundles over \(\bS^2\).
All these operators are surjective.  Hence \(D_u\) is surjective, and \(u\) is Fredholm regular.
\end{proof}

We next have the following local uniqueness property for pseudo-holomorphic spheres that are entirely contained in the product neighborhood.

\begin{lem}
\label{lem:local_uniqueness_single_family}
Let
\[
u\colon\bS^2\longrightarrow \overline{U_0}
\]
be a \(J_0\)-holomorphic sphere whose image is entirely contained in \(\overline{U_0}\) and whose 
energy is at most $\Omega(A)$.
Then \(u\) is, up to
reparametrization, one of the model spheres.
\end{lem}

\begin{proof}
Using the product coordinates, write
\[
u=(u_{\bC},u_v,u_{\bS^2},u_X).
\]

Since \(J_0\) is a product almost complex structure, the projection
\[
\pi_{\bC}\colon \overline{U_0}\longrightarrow
[-\epsilon,0]\times[-\delta,\delta]\subset\bC
\]
is holomorphic. Therefore
\[
u_{\bC}:=\pi_{\bC}\circ u
\]
is holomorphic and hence constant.

Likewise,
\[
u_v:=\pi_v\circ u
\colon
\bS^2\longrightarrow \overline{B}^{2(n-k)}_v\subset\bC^{\,n-k}
\]
is holomorphic and hence constant.

By the assumption in \Cref{thm:rank_char_fol_at_least_two} that $\omega_X$ is symplectically aspherical, 
the \(J_X\)-holomorphic map
\[
u_X\colon \bS^2\longrightarrow X
\]
is also constant. It follows that \(u\) is contained in a slice
\[
\{t=t_0,\ s=s_0,\ v=v_0\}
\times
\bS^2
\times \{x_0\}.
\]

The remaining component
\[
u_{\bS^2}\colon \bS^2\longrightarrow\bS^2
\]
is holomorphic and therefore has some non-negative degree \(d\). Since all other components of \(u\) are constant, the $\Omega$-energy of  \(u\) is \(d\cdot \Omega(A)\). On the other hand, \(u\) has at most energy \(\Omega(A)\) by assumption. Thus
\[
d\cdot \Omega(A)\leq \Omega(A).
\]
Using \(\Omega(A)>0\), gives \(d=1\). Hence \(u_{\bS^2}\) is a biholomorphism.

After reparametrizing the domain, we may assume
\[
u_{\bS^2}=\operatorname{id}_{\bS^2}.
\]
Therefore
\[
u(p)=(t_0,s_0,v_0,p,x_0)
\]
for some fixed \(x_0\in X\). Hence \(u\) is a model sphere.
\end{proof}


\subsection{The global moduli space and the evaluation map}
\label{sec:moduli_single_family}

We now extend the model almost complex structure \(J_0\) to all of \(W\)
and introduce the moduli space associated to the model spheres.

Consider an extension $J_1$ of \(J_0\), defined on $\overline{U_0}$,to $[-\epsilon,0]_t\times M$, namely
\[
J_1\in\End\bigl(TW|_{[-\epsilon,0]_t\times M}\bigr)
\]
such that
\[
J_1|_{\overline{U_0}}=J_0,
\]
on each slice $\{t\}\times M$ the restriction \(J_1|_\eta\) is cotamed by \(\eta\) 
and \(\Omega|_\eta\), and finally
\[
J_1((t+1)\partial_t)=R,
\]
where \(R\) is the geodesible vector field of a $1$-form  $\alpha$ defining $\eta$ that extends $\d s$ on $\overline{U_0}$ and $\omega:=\Omega\vert_M$.

Here, such a geodesible vector field exists for the following reason.
First, by assumption our c-symplectic confoliation is geodesible, so there is such a vector field $R_0$ for \emph{some} defining $1$-form $\alpha_0$.
Now, over $\overline{U_0}$ we have $\eta=\ker(\d s)$ by assumption, hence $\alpha_0 = f \d s$.
As $R_0\pitchfork \eta$, we can write it as $R_0=a \partial_s + X$ with $a\neq 0$ and $X\in \eta$.
An easy computation shows that the condition $\iota_{R_0}\d\alpha_0=0$ implies that $f$ can only depend on the $s$-coordinate.
Hence, we can perform a cutoff \emph{in the $s$-coordinate} to the model $\partial_s$ inside a smaller copy of $\overline{U_0}$, still containing the model curves near $L$ have been constructed; let $\alpha$ be the defining form for $\eta$ found in this way.
Now, an explicit computation shows that $R_0$ is in fact still geodesible for $\alpha$.
Hence, by the affine nature of the geodesibility condition (with respect to the vector field), we can interpolate from $R_0$ near $\partial\overline{U_0}$ to the standard $\partial_s$ on a smaller copy of $U_0$ still containing the model spheres near $L$; we call $R$ such global geodesible vector field for $\alpha$.
After this has been arranged, the extension of $J$ exists according to \Cref{prop:interpolation_tame_complex_structures} and \cite[Lemma 2.47]{CardonaGironella_ConfoliationsHighDim}.

The following is then standard (recall $\overline{U_0}\subset [-\epsilon,0]\times M$); see, for example,
\cite{McDuffSalamon_Jholomorphic} or
\cite{Wendl_LecturesHolomorphic}.

\begin{thm}
\label{thm:relative_transversality_single_family}
The set of almost complex structures $J$ on $W$ extending $J_1$ and satisfying the following condition is a residual subset of the set of all those extending $J_1$:
every
\(J\)-holomorphic sphere having an injective point in
\[
W\setminus\bigl([-\epsilon,0]\times M\bigr)
\]
is Fredholm regular.
\end{thm}

From now on, we use a $J$ that is generic in the sense of \Cref{thm:relative_transversality_single_family}.

Consider now the (unparametrized) marked moduli space \(\cM_A^*\) consisting of equivalence
classes \([u,z]\) of pairs \((u,z)\), where
\[
u\colon(\bS^2,i)\longrightarrow ( W , J)
\]
is a simple \(J\)-holomorphic sphere homotopic to a reparametrization of an interior
model sphere
\[
u_{t_0,s_0,v_0,x_0},
\qquad
t_0<0,
\]
and \(z\in\bS^2\) is a marked point. Here, the equivalence relation is given
by
\[
(u,z)\sim(u',z')
\]
if there exists a biholomorphism
\[
\varphi\colon(\bS^2,i)\longrightarrow(\bS^2,i)
\]
such that
\[
z'=\varphi(z)
\qquad\text{and}\qquad
u'=u\circ\varphi^{-1}.
\]
Note that there is a natural evaluation map
\begin{equation}
\label{eqn:evaluation_map}
\ev_A\colon\cM_A^*\longrightarrow  W,
\qquad
[u,z]\longmapsto u(z).
\end{equation}

By \Cref{lem:model_spheres_regular_single_family} and the implicit function
theorem, near each marked model sphere, the marked moduli space
\(\cM_A^*\) is locally identified with
\[
\overline{U_0}=[-\epsilon,0]
\times
[-\delta,\delta]
\times
\overline{B}^{2(n-k)}_v
\times
X
\times
\bS^2 \, .
\]
Here, the first four factors parametrize the model
spheres, while the last factor corresponds to the marked point in the domain $\bS^2$.

With respect to this identification, the evaluation map is given by
\[
(t_0,s_0,v_0,x_0,p)
\longmapsto
(t_0,s_0,v_0,p,x_0).
\]
Thus, \(\ev_A\) is a local diffeomorphism near the interior model family.
In particular, if $\cM_A^*$ is smooth at each of its points (which will be argued in the next section), then its dimension would be
\[
\dim\cM_A^*=2n+2=\dim W
\]
near the interior model family.

We next strengthen the local uniqueness property established in
\Cref{lem:local_uniqueness_single_family}.

\begin{lem}
\label{lem:uniqueness_lemma}
There is an open neighborhood
\[
U_0'\Subset U_0
\]
of \(L\) such that every \(J\)-holomorphic sphere intersecting $\overline{U_0'}$ and having $\Omega$-energy at most $\Omega(A)$ must coincide, up to reparametrization, with one of the model spheres.
\end{lem}

\begin{proof}
    Thanks to \Cref{lem:local_uniqueness_single_family}, it is enough to prove that there is an open subset $U_0'\Subset U_0$ such that every $J$-holomorphic sphere $u$ whose image intersects $\overline{U_0'}$ must in fact be entirely contained in $\overline U_0'$ (and hence in $U_0$).
    This can be proven with an argument completely analogous to the one in \cite[Pages 334 and 335]{MassotNiederkrugerWendl_WeakFillings}.

    First, we claim that the conclusion holds for all those $J$-holomorphic spheres intersecting 
    \[
    \{t=0\}\times \{s=0\}\times \{v=0\} \times L  
    \subset \overline{U_0}=[-\epsilon,0]_t\times [-\delta,\delta]_s\times B_v^{2(n-k)} \times L \, .
    \]
    Indeed, for any such $J$-holomorphic sphere $u\colon \bS^2\to W$, denoting $\Sigma:=u^{-1}(\overline{U_0})$, we know that the function $t\circ u\vert_{\Sigma}$ must have an (interior) maximum point $z_0\in \Sigma\subset \bS^2$.
    Hence by \Cref{lem:weak_maximum_principle} (applied to a slight cylindrical extension of $J$ on $(-\epsilon,\epsilon')\times M$ for a very small $\epsilon'>0$), the $t$-coordinate of $u$ must be constant equal to $0$ near $z_0$, and more precisely $u\vert_\Sigma$, and hence $u$ by unique continuation, must be tangent to the factor $L=\bS^2\times X$ of $\{t=0\}\times (-\delta,\delta)_s\times B_v^{2(n-k)} \times L$.
    Now, the fact that $X$ is symplectically aspherical tells us that the projection of $u$ to the $X$ factor of $L$ is constant.
    Lastly, an energy argument using the fact that $u$ is in the homology class $A$ tells that the former must in fact be a reparametrization of a model sphere, as claimed.

    We now prove the conclusion of the lemma by contradiction. 
    Namely, assume that, for each $m\in \bN$, there is a $J$-holomorphic sphere $u_m$, which is \emph{not} a model sphere but is homologous to one, and intersects $\{-t_m\}\times (-\delta_m,\delta_m)\times \{\vert v\vert < \rho_m\}\times L\subset \overline{U_0}$, where $t_m$, $\delta_m$ and $\rho_m$ are a sequence of positive real numbers converging to $0$ as $m$ goes to $+\infty$.
    Then, as the homology class of each of the $u_m$'s is the same as that of a model sphere, their energy is uniformly bounded, and so by Gromov's compactness theorem the $u_m$ converge to a nodal curve $u_\infty$, which by continuity must have a nodal component $u_\infty^0$ intersecting $\{t=0\}\times \{s=0\}\times \{v=0\}\times L$.
    
    According to the previous claim, $u_\infty^0$ must then be a reparametrization of a model sphere of the form $u_{t=0,s=0,v_0=0,x_0}$, that is tangent to $\{t=0,s=0,v=0\}\times L\subset \overline{U_0}$.
    But this means that also every other nodal component of $u_\infty$ that shares a node with $u_\infty^0$ must intersect $\{t=0,s=0,v=0\}\times L$, and hence must be a model sphere of the type $u_{t=0,s=0,v_0=0,x}$ for some $x\in X$.
    By induction on the nodal component, we hence deduce that each nodal component of $u_\infty$ must be a model sphere of such type.
    However, as the total energy is $\Omega(A)$, that is also the same as the energy of each nodal sphere, we must have that there is in fact a single nodal component, i.e.\ that $u_\infty$ is in fact a smooth curve in $\cM_A^*$.
    
    Now, as the virtual dimension of the moduli space at model spheres has dimension $2n$, and because $u_m$ are \emph{not} model spheres but converge to $u_\infty$ which is a (non-nodal) model sphere, this contradicts \Cref{lem:model_spheres_regular_single_family} which implies that an open neighborhood of $u_\infty$ in the moduli space of all $J$-holomorphic spheres in the homology class $A$ is given by model spheres by regularity of the latter and the implicit function theorem.
    This concludes the proof.
\end{proof}


\section{Obstruction to the existence of contact points}
\label{sec:proof_rank_char_fol_at_least_two}

We now prove \Cref{thm:rank_char_fol_at_least_two} using the marked moduli
space \(\cM_A^*\) and the evaluation map \(\ev_A\) introduced in
\Cref{sec:single_sphere_family}.
To this end, we first need to understand better in \Cref{sec:contact_point_at_boundary} the behavior of holomorphic curves near a potential contact point.
Afterwards, in \Cref{sec:evaluation_map_pseudocycle_near_arc} we prove that the evaluation map from \eqref{eqn:evaluation_map} is a pseudocycle near the image of a special arc in $W$.
Lastly, in \Cref{sec:obstruction_presence_contact_regions} we prove \Cref{thm:rank_char_fol_at_least_two} from the Introduction.


\subsection{Absence of $J$-holomorphic curves near a contact point}
\label{sec:contact_point_at_boundary}

\begin{lem}
\label{lem:no_spheres_near_contact_point}
Suppose that \(p\in M\) has a neighborhood \(O\subset M\) on which
\[
\rk\cK_\eta=0.
\]
Then there is a neighborhood \(V\subset W\) of \(p\), with
\[
V\cap M\subset O,
\]
such that no curve in \(\cM_A^*\) has image intersecting \(V\).
\end{lem}

Note that, up to shrinking the $\epsilon$ in the definition of $U_0$ in the previous sections (and changing the almost complex $J$ as a result), we can assume without loss of generality that $V$ is of the form $(-\epsilon,0]\times Op(p)$, where $Op(p)$ denotes an open neighborhood of $p$ in $M$, that $U_0=(-\epsilon,0]_t\times (-\delta,\delta)_s\times B_v^{2(n-k)}\times L$, and that $U_0$ is subset of a neighborhood of $M$ in $W$ of the form $(-\epsilon,0]_t\times M$ in which the function $t$ is weakly pluri-subharmonic with respect to the chosen $J$.
We will assume that this property holds in \Cref{sec:evaluation_map_pseudocycle_near_arc,sec:obstruction_presence_contact_regions} below.

\begin{proof}[Proof of \Cref{lem:no_spheres_near_contact_point}]
Suppose otherwise. Then there are sequences $[u_j,z_j]\in\cM_A^*$
and \(w_j\in\bS^2\) such that
\[
u_j(w_j)\longrightarrow p.
\]
By Gromov compactness, after passing to a subsequence, the curves converge
to a nodal curve representing \(A\) whose image contains \(p\).

Since
\[
\Omega(A)>0,
\]
the limiting nodal curve has a nonconstant component whose image contains
\(p\). By the weak maximum principle\footnote{Here, the usual maximum principle in the contact case can also be applied.}, this component is contained in \(M\)
and, over \(O\), tangent to \(\cK_\eta\). Since
\[
\cK_\eta|_O=0,
\]
the component is constant on an open neighborhood of a preimage of \(p\).
By unique continuation, it must be constant everywhere, which is a
contradiction. This concludes.
\end{proof}


\subsection{Pseudocycle property in the presence of a contact point}
\label{sec:evaluation_map_pseudocycle_near_arc}

Throughout this section, we assume that there is a point \(p\in M\) such that
\[
\rk(\cK_\eta)_p=0.
\]
By upper semi-continuity of the rank, there is then a neighborhood \(O\subset M\)
of \(p\) such that
\[
\rk\cK_\eta|_O=0.
\]
Let \(V\subset W\) be the neighborhood given by
\Cref{lem:no_spheres_near_contact_point}.

Fix a point \(p_0\in L\), and choose a smooth path
\[
\gamma\colon[0,1]\longrightarrow W
\]
such that
\[
\gamma(0)=p_0,
\qquad
\gamma(1)=p,
\qquad
\gamma((0,1))\subset W\setminus M.
\]
We may moreover arrange that \(\gamma\) agrees near \(p_0\) with a path tangent to the $t$-coordinate in
the neighborhood \(U_0\), and that
\[
[1-\epsilon,1]= \gamma^{-1}(\overline V)
\]
for some sufficiently small \(\epsilon>0\).

\bigskip

The first key property that we will need in the next section is that $\cM_A^*$ is a manifold near marked curves evaluating to $\gamma$.

\begin{lem}
    \label{lem:regularity_curves_intersecting_gamma}
    Every marked $J$-holomorphic map in $\cM_A^*$ that evaluates to $\gamma$ is regular.
\end{lem}

\begin{proof}
    The conclusion holds by \Cref{thm:relative_transversality_single_family} for any marked curve $[u,z]\in \cM_A^*$ such that $u(z)\in \gamma$ belongs to a region where $J$ is generic, as all curves in $\cM_A^*$ are simple, due to their homology class being that of a model sphere which is known to be primitive in the filling by assumption.
    The only potentially problematic marked curves are then those passing through a point of $\gamma$ that belongs to a region where $J$ is not generic.
    By the paragraph after \Cref{lem:no_spheres_near_contact_point}, these points either belong to $\{-\epsilon\}\times Op(p) \subset [-\epsilon,0]\times M$ in the neighborhood of $M$ in $W$, or to $\overline{U_0}$.
    
    In the first situation, as we know that no $J$-curve intersects the neighborhood $V$ of $p$, the curve is either entirely contained in $\{-\epsilon\}\times Op(p)$ or it touches the latter tangentially and then enters further inside $W$.
    Both situations are in fact forbidden by the maximum principle in the contact setting, as $\{-\epsilon\}\times Op(p)$ is of strict $J$-convex type (up to shrinking $\epsilon$ in all the previous constructions).

    In the situation where $u(z)$ belongs to $\overline{U_0}$, we distinguish two situations. 
    First, in the case where $u(z)\in U_0$, one can restrict to $\Sigma:=u^{-1}(U_0)$ and get that the $t$-coordinate of $(-\epsilon,0]\times M$ (that contains $U_0$) composed with $u\vert_\Sigma$ has an interior maximum point, and hence by the weak maximum principle we conclude that $u$ is tangent to a $t$-slice; the fact that $u(z)\in \gamma$ in fact forces $u$ to be entirely contained in $U_0$ and hence to be a model sphere by \Cref{lem:local_uniqueness_single_family}, which is then regular by \Cref{lem:model_spheres_regular_single_family}.
    Second, in the case where $u(z)\in \gamma$ is in the ``lower boundary'' of $\overline{U_0}$, i.e.\ $(\overline{U_0}\setminus U_0)\cap\{-\epsilon\} \times M$, we again have several possibilities.

    First, $u$ cannot go in the interior $U_0$; indeed, if so the weak maximum principle would say that $u$ is tangent to the level set $\{t=-\epsilon\}$ by the same argument as above, which would then be a contradiction.
    The second possibility is that $u$ intersects $W\setminus [-\epsilon,0]\times M$; in this case we know that it will be regular as in this region $J$ is generic.
    The last possibility is that $u$ is entirely tangent to $\{-\epsilon\}\times M$; by the same computations as in the proof of \Cref{lem:weak_maximum_principle} from \cite{CardonaGironella_ConfoliationsHighDim}, we then get that $u$ must in fact be tangent to the characteristic distribution in that level set, i.e.\ to the factors $\bS^2\times X$ passing through $u(z)$.
    But this means that $u$ is entirely contained in $\overline{U_0}$, and is hence a model sphere by \Cref{lem:local_uniqueness_single_family}, and then regular by \Cref{lem:model_spheres_regular_single_family}.
    
    This concludes the proof.
\end{proof}

\bigskip

In order to prove \Cref{thm:rank_char_fol_at_least_two} in \Cref{sec:obstruction_presence_contact_regions} below, we will need an understanding of the ``dimension'' of the intersection of a small neighborhood of $\gamma$ with the
\textbf{limit set} \(\Omega_A\) of \(\ev_A\).
The latter is defined as follows:
\[
\Omega_A
:=
\bigcap_{K\subset\cM_A^*\text{ compact}}
\overline{\ev_A(\cM_A^*\setminus K)}.
\]
In other words, \(\Omega_A\) is the set of all limit points of sequences of
the form
\[
\ev_A([u_j,z_j]),
\]
where \([u_j,z_j]\) has no convergent subsequence in \(\cM_A^*\).

Note that, by Gromov's compactness theorem, \(\Omega_A\) is made of
evaluations at the marked points of all possible marked nodal curves arising
as Gromov limits of sequences in \(\cM_A^*\).

We set
\[
\Omega_A^\circ:=\Omega_A\cap(W\setminus M).
\]
As far as the ``dimension'' is concerned, for simplicity, we say that
\(\Omega_A^\circ\) has \emph{dimension at most \(d\)} if there is a
possibly disconnected smooth oriented manifold \(Y\) of dimension at most
\(d\) and a smooth map
\[
f\colon Y\longrightarrow W\setminus M
\]
such that
\[
\Omega_A^\circ\subset f(Y).
\]

We now prove that $\ev$ satisfies the ``pseudo-cycle property'' near the arc $\gamma$.

\begin{lem}
\label{lem:limit_set_single_family}
For \(J\) as in \Cref{thm:relative_transversality_single_family} (c.f.\ the paragraph after \Cref{lem:no_spheres_near_contact_point}), the evaluation map \(\ev_A\) has image with compact
closure, and there is a neighborhood $V_\gamma$ of $\gamma$ in $W$ such that the intersection $\Omega_A^\circ \cap V_\gamma$ 
has dimension at most
\[
\dim\cM_A^*-2.
\]
\end{lem}

\begin{rmk}
\label{rmk:not_full_pseudocycle_single_family}
Technically speaking, \Cref{lem:limit_set_single_family} does not prove that
\(\ev_A\) is a \emph{pseudocycle} in the sense of
\cite[Page 10]{McDuffSalamon_Jholomorphic}. There, the whole limit set
\(\Omega_A\) is required to have dimension at most
\(\dim\cM_A^*-2\).

The issue in our setup is that nodal curves contained in the boundary \(M\)
are not necessarily regular. Hence, even under the semi-positivity
assumption, there is no guarantee that their evaluation images have
dimension at most \(\dim\cM_A^*-2\). 
However, for our purposes, it is
sufficient to control the subset \(\Omega_A^\circ\cap V_\gamma\).
\end{rmk}

\begin{proof}
The fact that the image has compact closure follows from the compactness of
\(W\), Gromov's compactness theorem using the global energy bound, and the
fact that the maximum principle in \Cref{lem:weak_maximum_principle} implies that a
\(J\)-holomorphic sphere touching \(M=\partial W\) must be tangent to \(M\).

\medskip

For the property about the dimension, we consider a small neighborhood $V_\gamma$ of $\gamma$ in $W$, that is allowed to be shrunk along the proof.
We first note that every point of \(\Omega_A^\circ\cap V_\gamma\) can only be obtained
as a limit of evaluations
\[
\ev_A([u_j,z_j])
\]
of marked curves
\[
[u_j,z_j]\in\cM_A^*
\]
whose images stay uniformly away from \(M\). 
This simply follows from \Cref{lem:no_spheres_near_contact_point}, provided that $V_\gamma$ is small enough near the endpoint $p$ of $\gamma$.

By this fact, the weak maximum principle from \Cref{lem:weak_maximum_principle} and Gromov's compactness theorem, \(\Omega_A^\circ\cap V_\gamma\) can then be
covered by evaluation images of marked nodal curves that are contained in $W$ (but might touch $M$).

A priori, the moduli spaces of such nodal curves might fail to have
codimension at least \(2\), for instance if some of their components are
multiply covered. However, as we are only concerned with images under the
evaluation map, by
\cite[Proposition 6.1.2]{McDuffSalamon_Jholomorphic}, we may replace the
evaluation images of non-simple nodal curves with those of simple nodal
curves. Thus, \(\Omega_A^\circ\cap V_\gamma\) is covered by evaluation images of \emph{simple}
bubbled spheres.
Note that while we know that at least one nodal component of the involved nodal spheres intersects $V_\gamma$, in general not all of them will do so.

The relevant simple components are regular. Indeed, components contained
in \(U_0\) are controlled by
\Cref{lem:local_uniqueness_single_family} and the regularity of the model
spheres. 

Every other simple component is regular by a similar argument to that in \Cref{lem:regularity_curves_intersecting_gamma}.
Namely, each (simple underlying curve of each) nodal component is either contained in $W\setminus [-\epsilon,0]\times M$, in which case it is regular by genericity of $J$, or intersects $[-\epsilon,0]\times M$.
In this second case, either it also intersects $(-\epsilon,0]\times M$ and then the weak maximum principle implies that such nodal component is tangent to a $t$-slice, and hence also all the nodal curve by induction on the nodal component; but as they are required to pass through a point in $V_\gamma$, by \Cref{lem:no_spheres_near_contact_point} this means that they should in fact all be contained in $\overline{U_0}$, in which case they are a (non-nodal) model sphere by \Cref{lem:local_uniqueness_single_family} and regular by \Cref{lem:model_spheres_regular_single_family}.
The last possible situation is the one where the nodal component of interest (or better its underlying simple curve) is entirely contained in $\{-\epsilon\}\times M$; in this case, an induction on the nodal components and the same computation as in the proof of \Cref{lem:weak_maximum_principle} from \cite{CardonaGironella_ConfoliationsHighDim} allows to conclude that the whole nodal curve must be contained in $V_\gamma\cap \{-\epsilon\}\times M$ and hence in $\overline{U_0}$ by \Cref{lem:no_spheres_near_contact_point}, in which case the curve must again be a (non-nodal) model sphere by \Cref{lem:local_uniqueness_single_family} and regular by \Cref{lem:model_spheres_regular_single_family}.

At this point, the standard dimension count for genus-zero nodal curves
applies. More precisely, by the semi-positivity of \(W\), the genericity of
\(J\) away from \(M\cup\overline{U_0}\), and the above regularity
statements, the moduli spaces of once-marked simple nodal curves whose
evaluation images cover \(\Omega_A^\circ\) have connected components of
dimension at most
\[
\dim\cM_A^*-2.
\]
See \cite[Chapter 6]{McDuffSalamon_Jholomorphic}; more specifically, a
slight adaptation of the argument in
\cite[Theorem 6.6.1]{McDuffSalamon_Jholomorphic} gives the required
dimension estimate.

Since
\[
\dim\cM_A^*=\dim W,
\]
this proves the claim.
\end{proof}


\subsection{Proof of the obstruction}
\label{sec:obstruction_presence_contact_regions}

We assume by contradiction that
\Cref{thm:rank_char_fol_at_least_two} is false, i.e.\ that there is a point $p\in M$ such that $\cK_\eta$ has rank $0$ at $p$.
We are then in the assumption of the previous subsection, and can utilize its notations and conclusions.

By choice of $V$, the uniqueness property in \Cref{lem:uniqueness_lemma} and the fact that
\[
\dim(\Omega_A^\circ\cap V_\gamma)\leq\dim W-2,
\]
we may perturb \(\gamma\), relative to neighborhoods of its endpoints and inside $V_\gamma$, so
that
\[
\gamma((0,1))\cap\Omega_A^\circ=\emptyset
\]
and \(\ev_A\) is transverse to \(\gamma((0,1))\).
By \Cref{lem:regularity_curves_intersecting_gamma} and the fact that the resulting evaluation map is then smooth, it follows that
\[
\cM_{A,\gamma}^*
:=
\ev_A^{-1}\bigl(\gamma((0,1))\bigr)
\]
is a smooth \(1\)-dimensional manifold.

Because \(\gamma((0,1))\) avoids \(\Omega_A^\circ\), no sequence in
\(\cM_{A,\gamma}^*\) can degenerate to a nodal curve whose marked point lies
in \(W\setminus M\). Moreover, by
\Cref{lem:no_spheres_near_contact_point}, no curve in \(\cM_A^*\) meets
\(\gamma([1-\epsilon,1])\).

Near \(p_0\), again the uniqueness property in \Cref{lem:uniqueness_lemma} shows that
\(\cM_{A,\gamma}^*\) has exactly one end. Its limit is the unique marked
model sphere through \(p_0\). 

Consequently, \(\cM_{A,\gamma}^*\) can be compactified by adding exactly
one boundary point, corresponding to the marked model sphere through
\(p_0\). The resulting space is a compact \(1\)-dimensional manifold with
exactly one boundary point.

This is impossible, since every compact \(1\)-dimensional manifold has an
even number of boundary points. We have therefore obtained a contradiction.

It follows that
\[
\rk(\cK_\eta)_p\geq2
\]
for every \(p\in M\), as desired.
\hfill
\qedsymbol


\section{Two families of pseudo-holomorphic spheres}
\label{sec:pseudo-hol_spheres_semi-filling}

We now construct the two families of pseudo-holomorphic spheres needed for
the proof of \Cref{thm:rank_char_fol_at_least_four}. The construction is
the two-family analogue of \Cref{sec:single_sphere_family}, so we only
describe the points where the presence of two distinguished sphere factors
requires a modification.

Throughout this section, \((W,\Omega)\) denotes the symplectic manifold
appearing in one of Hypotheses
\hypref{hyp:weak_filling} and
\hypref{hyp:cobordism_from_hypertight_shs}.


\subsection{The two model families}
\label{sec:model_near_leaf}

Let \(L\cong\bS^2\times\bS^2\times X\) be the submanifold appearing in
\Cref{thm:rank_char_fol_at_least_four}, where \(\dim L=2k\) and
\(\dim X=2k-4\).

For the argument below, we use the analogue of
the model neighborhood of $L$ in the previous section, with the difference that here
\(L=\bS^2\times\bS^2\times X\). Thus, we have a product neighborhood
\[
U_0=
(-\epsilon,0]_t\times(-\delta,\delta)_s
\times B^{2(n-k)}_v\times\bS^2\times\bS^2\times X
\]
with
\[
\overline{U_0}=
[-\epsilon,0]_t\times[-\delta,\delta]_s
\times \overline{B}^{2(n-k)}_v\times\bS^2\times\bS^2\times X
\]
and a product almost complex structure \(J_0\) on $\overline{U_0}$ satisfying
\(J_0\partial_t=\partial_s\) and
\[
J_0|_{TL}=i_{\bS^2}\oplus i_{\bS^2}\oplus J_X.
\]
We also know that there is no non-constant \(J_X\)-holomorphic sphere in $X$, by the assumption that $\omega_X$ is symplectically aspherical.

Let \(A_i\in H_2(W;\mathbb Z)\) denote the class represented by the
\(i\)-th sphere factor, for \(i=1,2\). The two families of model spheres
are
\begin{align*}
u^1_{t_0,s_0,v_0,q_0,x_0}(p)
&=(t_0,s_0,v_0,p,q_0,x_0),\\
u^2_{t_0,s_0,v_0,q_0,x_0}(p)
&=(t_0,s_0,v_0,q_0,p,x_0).
\end{align*}
Here \(t_0\in[-\epsilon,0]\), \(s_0\in[-\delta,\delta]\),
\(v_0\in \overline B^{2(n-k)}_v\), \(q_0\in\bS^2\), and \(x_0\in X\).

\begin{lem}
\label{lem:regularity_near_L}
Every sphere in either model family is Fredholm regular. 
\end{lem}

\begin{proof}
The proof is the same as that of
\Cref{lem:model_spheres_regular_single_family}. For a sphere in either
family, its pullback tangent bundle splits as the tangent bundle of the
nonconstant sphere factor together with trivial complex summands
corresponding to the \((t,s)\)-directions, the \(v\)-directions, the
other sphere factor, and the \(X\)-directions. The linearized
Cauchy--Riemann operator splits accordingly, and each summand is
surjective. Hence every model sphere is Fredholm regular.
\end{proof}

\begin{lem}
\label{lem:uniqueness_near_L}
Let \(u\colon\bS^2\to U_0\) be a \(J_0\)-holomorphic sphere representing
\(A_i\). Then \(u\) is, up to reparametrization, a sphere in the \(i\)-th model family.
\end{lem}

\begin{proof}
As in the proof of
\Cref{lem:local_uniqueness_single_family}, the projections of \(u\) to the
\((t,s)\)- and \(v\)-factors are constant. Its projection to \(X\) is
constant by the fact that $\omega_X$ is symplectically aspherical.

Let \(d_1,d_2\geq0\) be the degrees of the projections of \(u\) to the two
sphere factors. 
By
\Cref{item:splitting_mu__thm_rank_at_least_four}, the two sphere factors have the
same minimal positive energy, which we denote by \(E_0\). Since \(u\)
represents \(A_i\), its energy is \(E_0\), whereas the product form of
\(\Omega\) gives
\[
\Omega([u])=(d_1+d_2)E_0.
\]
It follows that \(d_1+d_2=1\). Thus one sphere projection is a
biholomorphism and the other is constant.

If \(u\) represents \(A_1\), the nonconstant projection must be the first
one. Otherwise \(u\) would represent \(A_2\), contrary to
\Cref{item:homology_classes__thm_rank_at_least_four}. Therefore \(u\)
belongs, up to reparametrization, to the first model family. The case of
\(A_2\) is identical.
\end{proof}


\subsection{The two moduli spaces}
\label{sec:mod_space_spheres}

Following
\Cref{sec:moduli_single_family}, we extend \(J_0\) on $\overline{U_0}=[-\epsilon,0]\times[-\delta,\delta]\times \overline{B}^{2(n-k)}_v \times L$ to an almost complex
structure \(J\) on \(W\) which has the required cotaming properties on each $t$-slice in the collar neighborhood $[-\epsilon,0]_t\times M$, agrees with \(J_0\) on \(\overline{U_0}\), and is generic away from
$[-\epsilon,0]_t\times M$ (note that $\overline{U_0}=[-\epsilon,0]\times[-\delta,\delta]\times \overline{B}^{2(n-k)}_v\times L\subset [-\epsilon,0]_t\times M$ by construction).

Since only the two classes \(A_1\) and \(A_2\) are involved, \(J\) may be
chosen so that the relative generic regularity statement as in \Cref{thm:relative_transversality_single_family} holds simultaneously
for both classes.
Namely, we have the following.

\begin{thm}
\label{thm:regularity_away_from_L}
For \(i=1,2\), every simple \(J\)-holomorphic sphere representing \(A_i\)
and having an injective point in
\(W\setminus([-\epsilon,0]\times M)\) is Fredholm regular.
\end{thm}

For \(i=1,2\), let \(\cM_*^i\) denote the marked moduli space of
simple \(J\)-holomorphic spheres in the class \(A_i\), defined in the same way as
\(\cM_A^*\) in \Cref{sec:moduli_single_family}. We denote its evaluation
map by
\[
\ev_i\colon\cM_*^i\longrightarrow W,
\qquad
\ev_i([u,z])=u(z).
\]

We now describe the two marked moduli spaces near their model families.
For either family, the kernel of the linearized Cauchy--Riemann operator,
after quotienting by infinitesimal reparametrizations, consists precisely
of variations of \(t_0,s_0,v_0,q_0\), and \(x_0\). These parameters form
a space of dimension
\[
2+2(n-k)+2+(2k-4)=2n.
\]
Thus, by \Cref{lem:regularity_near_L} and the implicit function theorem,
near each marked interior model sphere the corresponding marked moduli
space is locally identified with
\[
(-\epsilon,0)\times(-\delta,\delta)\times
B^{2(n-k)}_v\times\bS^2\times X\times\bS^2,
\]
after restricting the parameters so that the model spheres are contained
in \(U_0\). The last factor corresponds to the marked point.

Under this identification, \(\ev_i\) evaluates the marked model sphere
at its marked point. Consequently, for \(i=1,2\),
\[
\dim\cM_*^i=2n+2=\dim W,
\]
and \(\ev_i\) is a local diffeomorphism near the corresponding interior
model family.

As in \Cref{sec:moduli_single_family}, the model spheres with \(t_0=0\)
arise as limits of the interior model families. The local description
extends to them with boundary parameter \(t_0=0\).

\medskip

We next strengthen \Cref{lem:uniqueness_near_L} in the same way as
\Cref{lem:uniqueness_lemma}.

\begin{lem}
    \label{lem:uniqueness_lemma_bis}
    There is a compactly contained open neighborhood $U_0'\Subset U_0$ of $L$ such that, for $i=1,2$, every $J$-holomorphic sphere intersecting $\overline{U_0'}$ and in the homology class $A_i$ must coincide, up to reparametrization, with a sphere in the $i$-th model family.
\end{lem}

\begin{proof}
This can be proved analogously to \Cref{lem:uniqueness_lemma}.
The only difference is the fact that there are here two symplectically spherical homology classes; however, the assumption in \Cref{thm:rank_char_fol_at_least_four} that the two $\bS^2$ factors in $L$ define distinct homology classes in the fillings allows to argue in an analogous way to \Cref{lem:uniqueness_lemma}.
We hence omit the details.
\end{proof}


\section{Obstructing rank $2$ characteristic distributions}
\label{sec:proof}

The overall aim of this section is to prove \Cref{thm:rank_char_fol_at_least_four} from the Introduction.
As this proof has very similar ingredients to that of \Cref{thm:rank_char_fol_at_least_two} in \Cref{sec:proof_rank_char_fol_at_least_two}, we will sometimes just describe how the proofs from the previous sections need to be modified.

\medskip

First, according to \Cref{thm:rank_char_fol_at_least_two}, we know that $\cK_\eta$ has rank at least $2$ at each point of $M$.
As in \Cref{sec:proof_rank_char_fol_at_least_two}, the proof is again by contradiction.
In other words, we assume that there is $p\in M$ where the rank of $\cK_\eta$ is $2$, and deduce a contradiction.

First, note that, as $\rk\cK_\eta\geq 2$ everywhere on $M$, there is a neighborhood $V_M$ of $p$ in $M$ where the rank of $\cK_\eta$ will be exactly equal to $2$.
We hence work in this situation throughout this section.

\medskip

In \Cref{sec:curves_near_point_rank_2} we will describe how our $J$-holomorphic curves behave near $V\subset M\subset \partial W$.
In \Cref{sec:evaluation_map_pseudocycle_near_arc_bis} we then prove a pseudo-cycle property analogous to the one in the previous section.
Lastly, in \Cref{sec:proof_thm_rank_at_least_four} we will prove \Cref{thm:rank_char_fol_at_least_four}.


\subsection{$J$-curves near a boundary point of characteristic rank $2$}
\label{sec:curves_near_point_rank_2}

We first point out a special property of $J$-curves in $W$ that are sufficiently near the point $p\in M$ where the rank of $\cK_\eta$ is exactly $2$, that will play the role of the previous \Cref{lem:no_spheres_near_contact_point} in this new situation:
\begin{lem}
    \label{lem:controlled_spheres_near_contact_point}
    There is a neighborhood $V$ of $p$ in $(-\epsilon,0]_t\times M$ of the form $(-\epsilon,0]_t\times Op(p)$, with $Op(p)\subset M $ a neighborhood of $p$ in $M$, such that each $J$-holomorphic sphere $u$ intersecting $V$ has image that coincides, locally near $p$, with $\{t\}\times L$, where $t\in(-\epsilon,0]$ and $L$ is a leaf of $\cK_\eta\vert_{Op(p)}$. 
\end{lem}
    The proof is completely analogous to that of \Cref{lem:no_spheres_near_contact_point}, and is hence omitted.

\bigskip

Now, as done in \Cref{sec:proof_rank_char_fol_at_least_two}, fix a point \(p_0\in L\), and choose a smooth path
\[
\gamma\colon[0,1]\longrightarrow W
\]
such that
\[
\gamma(0)=p_0,
\qquad
\gamma(1)=p,
\qquad
\gamma((0,1))\subset W\setminus M.
\]
We may moreover arrange that \(\gamma\) agrees near \(p_0\) with a path tangent to the $t$-coordinate in
the neighborhood \(U_0\), and that
\[
[1-\epsilon,1]= \gamma^{-1}(\overline V)
\]
for some sufficiently small \(\epsilon>0\).


\subsection{Pseudo-cycle property near a special arc}
\label{sec:evaluation_map_pseudocycle_near_arc_bis}

We now study how the limit set of the relative evaluation maps behave in this new setting.
More precisely, let \(\Omega_i\) be the limit set of \(\ev_i\), defined as in
\Cref{sec:moduli_single_family}, and set
\[\Omega_i^\circ:=\Omega_i\cap(W\setminus M) \, .\]

Using \Cref{lem:controlled_spheres_near_contact_point} instead of \Cref{lem:no_spheres_near_contact_point}, one can argue similarly to \Cref{lem:limit_set_single_family} in order to prove the following:

\begin{lem}
\label{lem:evaluation_pseudocycles}
For generic \(J\), each evaluation map \(\ev_i\) has image with compact
closure, and there is a sufficiently small neighborhood $V_\gamma$ of $\gamma$ in $W$, coinciding with $V$ near the endpoint $p\in M$, such that
\[
\dim(\Omega_i^\circ\cap V_\gamma)\leq\dim\cM_*^i-2=\dim W-2.
\]
\end{lem}

\begin{proof}
The proof of \Cref{lem:limit_set_single_family} applies separately to the
two classes. The local regularity and uniqueness needed in that argument
are provided by
\Cref{lem:regularity_near_L,lem:uniqueness_near_L}, while relative
transversality away from \([-\epsilon,0]\times M\) follows from
\Cref{thm:regularity_away_from_L}.
Moreover, \Cref{lem:no_spheres_near_contact_point} should be substituted in this situation by \Cref{lem:controlled_spheres_near_contact_point}.
Lastly, the same compactness,
semi-positivity, and dimension-counting arguments therefore give the
claimed estimate for both interior limit sets.
\end{proof}

As in \Cref{rmk:not_full_pseudocycle_single_family}, the preceding lemma
does not control the part of \(\Omega_i\) contained in \(M\). This causes no
difficulty because the paths used below have interiors contained in
\(W\setminus M\).

\bigskip

Afterwards, we record a consequence of the $1$-dimensional moduli-space argument
used in the proof of \Cref{thm:rank_char_fol_at_least_two}.

\begin{lem}
\label{lem:stable_spheres_through_boundary_points}
For every \(q\in Op(p)=V\cap M\) and \(i\in\{1,2\}\), there exists a marked nodal
\(J\)-holomorphic sphere representing \(A_i\) whose image contains \(q\).
\end{lem}

\begin{proof}
Fix \(i\in\{1,2\}\). We repeat the path argument from the proof of
\Cref{thm:rank_char_fol_at_least_two}, with
\[
(\cM_A^*,\ev_A,\Omega_A^\circ)
\quad\text{replaced by}\quad
(\cM_*^i,\ev_i,\Omega_i^\circ).
\]

More precisely, fix \(p_0\in L\) and choose a path
\(\gamma'\colon[0,1]\to V_\gamma\) from \(p_0\) to \(q\), with interior contained
in \(W\setminus M\). Near \(p_0\), we choose \(\gamma'\) in the product
neighborhood and transverse to \(M\).

By \Cref{lem:uniqueness_lemma_bis,lem:evaluation_pseudocycles}, after perturbing \(\gamma'\) relative
to its endpoints, we may assume that \(\gamma'((0,1))\) avoids
\(\Omega_i^\circ\cap V_\gamma\) and that \(\ev_i\) is transverse to
\(\gamma'((0,1))\). Hence
\[
(\ev_i)^{-1}\bigl(\gamma'((0,1))\bigr)
\]
is a smooth one-dimensional manifold having one end converging to the
unique marked model sphere through \(p_0\).

If no marked nodal sphere representing \(A_i\) passed through \(q\), then
Gromov compactness would give a neighborhood of \(p\) disjoint from the
images of all curves in \(\cM_*^i\). 
The compactification argument in the
proof of \Cref{thm:rank_char_fol_at_least_two} would then produce a compact
$1$-dimensional manifold with exactly one boundary point, corresponding
to the model sphere through \(p_0\). This is impossible. Therefore a
marked nodal sphere representing \(A_i\) passes through \(q\).
\end{proof}


\subsection{Proof of maximality of the characteristic rank}
\label{sec:proof_thm_rank_at_least_four}

We work in the situation arranged in the previous part of \Cref{sec:proof} (arguing by contradiction).

For \(i=1,2\), let \(\mathbf u_i\) be a nodal sphere representing \(A_i\)
and passing through \(p\), as provided by
\Cref{lem:stable_spheres_through_boundary_points}. By the weak maximum
principle, every nonconstant component of \(\mathbf u_i\) is contained in
\(M\) and, over \(Op(p)=V\cap M\), is tangent to \(\cK_\eta\).

Let
\[
E_0:=\Omega(A_1)=\Omega(A_2)=\omega_{\bS^2}([\bS^2]).
\]
By 
\Cref{item:splitting_mu__thm_rank_at_least_four},
\(E_0\) is the minimal
positive energy of a spherical class in \(M\). Every nonconstant component
of \(\mathbf u_i\) therefore has energy at most \(E_0\). Since the total
energy of \(\mathbf u_i\) is \(E_0\), it follows that \(\mathbf u_i\) has
exactly one nonconstant component. Denote it by
\[
u_i\colon\bS^2\longrightarrow M.
\]
It represents \(A_i\), and its image contains \(p\).

Moreover, \(u_i\) is simple. Otherwise it would be a nontrivial multiple
cover of a sphere with positive energy strictly smaller than \(E_0\),
contrary to the minimality of \(E_0\).

Since \(\cK_\eta\) has rank \(2\) on \(Op(p)\), the images of \(u_1\) and \(u_2\)
agree near \(p\) with an open subset of the characteristic leaf through
\(p\). It follows from
\cite[Proposition 2.4.4]{McDuffSalamon_Jholomorphic}
that
\[
u_1(\bS^2)=u_2(\bS^2).
\]
As \(u_1\) and \(u_2\) are simple, they therefore represent the same
homology class in \(W\). This contradicts
\Cref{item:homology_classes__thm_rank_at_least_four}, which states that
\(A_1\neq A_2\).

Consequently, rank \(2\) cannot occur. Together with the rank-two lower
bound established above, this gives
\[
\rk(\cK_\eta)_p\geq4
\]
for every \(p\in M\), as desired.
\hfill
\qedsymbol


\subsection{Alternative proof in ambient dimension $5$}
\label{sec:proof_rank_at_least_four_alternative_assumption}

We now prove \Cref{thm:rank_char_fol_at_least_four} under the alternative Assumption \ref{item:splitting_mu_bis__thm_rank_at_least_four} (instead of \ref{item:splitting_mu__thm_rank_at_least_four}).
Note that in this ambient dimension $X=\{pt\}$, which is hence symplectically aspherical.
The only difference with respect to \ref{item:splitting_mu__thm_rank_at_least_four} is hence that we don't know whether the $\bS^2$-factors have minimal $\Omega$-energy.

This change of hypothesis doesn't affect the proof of the fact that $\rk\cK_\eta\geq 2$ at all points of $M$ from the previous section. 
In order to prove that in fact such rank is equal to $4$ (recall we are in ambient dimension $\dim M = 5$), we can then argue as follows leveraging \Cref{prop:fully_foliated_part_closed} stating closedness of the fully foliated part.

\bigskip

Consider then any potential point $p\in M$ such that $\rk(\cK_\eta)_p=2$.

\begin{claim}
\label{claim:honestly-nodal}
For at least one $i\in\{1,2\}$, there is a marked \emph{honestly-nodal} $J$-holomorphic curve $[u_p^i,z_p^i]$ passing through $p$, in the homology class of the model spheres $u_{t_0,s_0,v_0,q_0,x_0}^i$.
\end{claim}

Note that each nodal component of such a honestly-nodal curve has energy at most $E_0-\hbar$, where $E_0=\omega_{\bS^2}(\bS^2)$ and $\hbar$ is the minimal $\Omega$-energy that a non-constant $J$-holomorphic sphere in $W$ can have by the energy quantization property (c.f.\ \cite[Proposition 4.1.4]{McDuffSalamon_Jholomorphic}).

\begin{proof}
The fact that, for $i=1,2$, there is a (possibly-not-honestly-)nodal curve $[u_p^i,z_p^i]$ passing through $p$ 
is exactly the content of \Cref{lem:stable_spheres_through_boundary_points}.
We then should argue that at least one such $[u_p^i,z_p^i]$ must in fact be \emph{honestly-nodal}. This can be argued as follows.

Assume by contradiction that both are smooth.
As $\rk (\cK_\eta)=2$ near $p$ by upper semi-continuity of the rank, the image of both $u_p^i$ must then coincide with the leaf of $\cK_\eta$ through $p$ near $z_p^i$.
As both these curves are non-constant, according to \cite[Proposition 2.4.4]{McDuffSalamon_Jholomorphic} their images must hence coincide; but this contradicts the fact that $u_p^1$ is in the homology class of 
$u_{t_0,s_0,v_0,q_0,x_0}^1$,
and that $u_p^2$ is in that of $u_{t_0,s_0,v_0,q_0,x_0}^2$, which is distinct.
This concludes.
\end{proof}

We then consider the maximal (under inclusion) path-connected subset $\cS$ of $M$ containing $L$ and made of those points through which passes an embedded $\bS^2\times \bS^2$ that is everywhere tangent to $\eta$ (and hence to $\cK_\eta$ as we are in ambient dimension $5$).
Note that, by upper semi-continuity of $\rk \cK_\eta$, the fact that $\Fol(\eta)$ is closed and Reeb stability in the foliated setup, the set $\cS$ is closed in $M$.

Assume now by contradiction that $\cS$ is not open.
(Note that, as $\cS$ is closed, proving that it is open would give that it coincides with an entire connected component of $M$.)
We claim that there is then a sequence $p_n$ so that $\rk (\cK_\eta)_{p_n}=2$, converging to $p_\infty\in \cS$.

Indeed, if $\cS$ is not open then there is a sequence of points $p_n$ in its complement converging to $p_\infty \in \cS$.
Now, if $p_n$ all satisfy that $\rk (\cK_\eta)_{p_n}=4$ (for $n$ big enough), then again by Reeb stability in the foliated setup and the fact that $\Fol(\eta)$ is closed we would deduce that $\cS$ can be enlarged to a strictly bigger subset of $M$ that is foliated by leaves diffeomorphic to $\bS^2\times \bS^2$, which contradicts the maximality of $\cS$ under inclusion.
Hence we deduce the claim.

According to \Cref{claim:honestly-nodal}, we then have an associated sequence $[u_n,z_n]$ of \emph{honestly-nodal} curves given by the above claim, so that $u_n(z_n)=p_n$.
Let moreover $[v_n,z_n]$ be the nodal connected component containing the marked point $z_n$, which then must have energy less than $E_0-\hbar$ (c.f.\ the paragraph after the statement of \Cref{claim:honestly-nodal}).

By Gromov's compactness theorem we know that the $[v_n,z_n]$ Gromov-converge to a limit nodal curve $[v_{\infty},z_\infty]$ so that $v_\infty(z_\infty)=p_\infty$.
Moreover, either by the maximum principle or by the fact that all $u_n$'s have image in $M$, $v_\infty$ also has image entirely contained in $M$. 
In fact, as $p_\infty\in \cS$, $v_\infty$ has image in a closed leaf $L_\infty$ diffeomorphic to $S^2\times S^2$ in the fully foliated part $\Fol(\eta)$ of $\eta$.

But since $\Omega([v_\infty(S^2)])\leq E_0-\hbar<E_0$ as well because this holds for all $v_n$'s, we get a contradiction.
Indeed, by the global Reeb stability theorem for foliations, each closed leaf in $\cS$ is homotopic to $L$, 
and we hence know that each sphere in $L_\infty$ is homotopic to one in $L$ and spheres in the latter have minimal $\Omega$-energy equal to $E_0$.

This contradiction shows that $\cS$ must in fact be open.
Hence, it coincides with the whole connected component $M$ of $\partial W$.
This concludes the proof of \Cref{thm:rank_char_fol_at_least_four} in ambient dimension $5$, under the alternative assumption \ref{item:splitting_mu_bis__thm_rank_at_least_four}.
\hfill
\qedsymbol


\section{Contact approximations and weak fillability}
\label{sec:contact_approximations_weak_fillability}

This section is devoted to the discussion of two new notions of contact approximations of c-symplectic confoliations in high dimensions. 
Namely, we introduce and discuss some properties of tame contact approximations in \Cref{sec:contact_tame_approximations}, and strict ones in \Cref{sec:contact_strict_approximations}.
Afterwards, we describe natural examples of both in \Cref{sec:examples_contact_strict_and_tame_approx_and_deform}, thus proving \Cref{prop:examples_contact_approx_and_deform} from the Introduction.
Then, we prove \Cref{thm:weak_fillability_stable_under_tame_contact_approx,thm:weak_fillability_stable_under_strict_contact_approx} on the stability of weak fillability under these two notions of contact approximations in \Cref{sec:stability_weak_fillability}.
Lastly, in \Cref{sec:sympl_fol_without_contact_approx} we prove \Cref{cor:no_contact_approximations_product_foliation} from the Introduction.


\subsection{Contact approximations via complex structures}
\label{sec:contact_tame_approximations}

We start by recalling a natural notion of convergence of sequences of endomorphisms defined on converging vector sub-bundles of $TM$.

Given a hyperplane field $\eta$ on a Riemannian manifold $(M,g)$ and an endomorphism $J\in\End(\eta)$, we denote by $\widetilde J$ its extension to $TM=\eta\oplus \eta^{\perp}$ obtained by posing $\widetilde J\vert_{\eta^{\perp}}=0$.
Whenever $M$ does not come equipped with any special Riemannian metric, as in our setting, we will simply use any auxiliary one.

Let now $\eta$ be a hyperplane field, and $(\xi_m)_{m\in \bN}$ be a sequence of hyperplane fields $C^k$-converging to $\eta$ for some $k\geq 1$.
(Recall \Cref{rmk:C1_regularity_for_contact_approximations}.)
More explicitly, by this we mean that there is a sequence of one-forms $\alpha_m$ defining $\xi_m$ so that $\alpha_m \xrightarrow{C^k}\beta$ with $\beta$ a defining one-form for $\eta$.

Consider now an endomorphism $J\in \End(\eta)$ of $\eta$, and a sequence of endomorphisms $J_m\in \End(\xi_m)$ of $\xi_m$.
Given $k\geq 1$, we will say that $J_m$ \textbf{$C^k$-converges} to $J$ if for some (and hence any) auxiliary Riemannian metric $g$ on $M$, the extension $\widetilde J_m$ to $TM=\xi_m\oplus \xi_m^\perp$ converges, in the space $\End(TM)$ equipped with its norm $C^k$, to the extension $\widetilde J $ on $TM=\eta\oplus \eta^\perp$.
(The fact that this notion is independent of the auxiliary Riemannian metric $g$ is a straightforward calculation.)

\medskip

With these preliminary notions of smooth convergence recalled, we can then give the following.
\begin{defn}
    \label{def:contact_tame_approx_and_deform}
    Let $k\geq 1$ be an integer, $(\eta,\CS_{\eta,\mu})$ a tame c-symplectic confoliation, and $(\xi_m)_{m\in \bN}$ a sequence of contact structures.
    Then, $(\xi_m)_{m\in \bN}$ is a \textbf{contact tame $C^k$-approximation} of $(\eta,\CS_{\eta,\mu})$ if there are a complex structure $J$ on $\eta$ tamed by $\mu$ and $\CpS_\eta$, and a sequence of complex structures $J_m$ on $\xi_m$, each tamed by $\CS_{\xi_m}$ on $\xi_m$, and $C^k$-converging to $J$ on $\eta$ as $m\to \infty$.
\end{defn}

\begin{rmk}
    \label{rmk:contact_tame_approx_3D}
    By the fact that an area form of a $2$-dimensional vector space induces a tamed complex structure on it, unique up to homotopy among tamed complex structures, in ambient dimension $3$, the above notion extends that of contact approximations of confoliations by Eliashberg and Thurston \cite{EliashbergThurston_Confoliations}.
\end{rmk}


\subsection{Contact approximations via differential forms}
\label{sec:contact_strict_approximations}

As outlined in the introduction, the overall idea is to define contact approximation of a c-symplectic confoliation $(\eta,\CS_{\eta,\mu})$ a sequence of contact structures $\xi_m$ which smoothly (i.e.\ as hyperplane fields) converge to the underlying smooth confoliation as $m\to \infty$, in such a way that, for some (and hence any) choice of defining one-forms $\alpha_m$ for $\xi_m$, up to conformal rescaling (depending on $m$) the wedge powers of the form $\alpha_m\wedge \d\alpha_m^k\wedge \omega^{n-k}$ converge to positive volume forms as $m\to \infty$.

There are mainly two sources of cumbersome technicalities in formalizing this simple idea.
The first is the fact that a priori, the conformal rescaling might depend on the point, in a discontinuous way (c.f.\ \Cref{exa:subpartition_necessary} below). 
Hence, one cannot simply rescaling by a smooth function on the ambient manifold that additionally depends on $m$.
The second source consists in the fact that this convergence (up to rescaling) to positive volume forms is necessary only for $k$ strictly bigger than $n$ minus the rank of $\cK_\eta$.
This rank function is not continuous, hence this is a second obstruction to choosing rescalings that are smooth functions on the ambient manifold.
(In fact, this second issue can be solved up to a slight sacrifice in terms of generality of the definition; c.f.\ \Cref{rmk:strict_contact_approximation}.)

For this reason, we now need to formalize the notion of ``limits of differential forms up to rescaling depending discontinuously on the points''.
We follow here the approach of \emph{partitioned differential forms} from \cite{CardonaGironella_ConfoliationsHighDim}, which we recall briefly below.


\subsubsection{Recalls on partitioned differential forms}
\label{sec:partitioned_forms}

We recall here briefly the notion of \emph{partitioned differential forms} and their \emph{conformal convergence} from \cite{CardonaGironella_ConfoliationsHighDim}, and refer the reader to the latter for more details.

\bigskip

Let $M$ be a smooth manifold and $X\subset M$ be any subset.
A \textbf{$C^k$-differential $l$-form $\mu$ on $X$} is a section of the bundle $\Lambda^l(T^* M\vert_X)\to X$ such that there are an open neighborhood $U\subset M$ of $X$ and a (honest) $C^k$-differential $l$-form $\widetilde{\mu}$ on $U$ such that $\widetilde{\mu}\vert_X = \mu$ as sections of $\Lambda^l(T^* M\vert_X)\to X$.

Given a finite partition $M= \sqcup_{i\in I} X_i$ of $M$, we call \textbf{$C^k$-differential partitioned-form $(\mu_i)_{i\in I}$} a collection of $C^k$-differential forms $\mu_i$ on $X_i$, of degree $\deg(\mu_i)=d_i$ possibly depending on $i$.
If $d_i=l$ for all $i\in I$, we will talk about \textbf{$l$-forms}. 
A $C^k$-differential partitioned $0$-form will also be called \textbf{$C^k$-differential partitioned-function}.
Moreover, whenever we want to specify the partition $\cP=\{X_i\mid i\in I\}$ of $M=\sqcup_{i\in I} X_i$, we will also say \textbf{$\cP$-partitioned form}.

\medskip

Let now $\cP$ be a finite partition $M=\bigsqcup_{i\in I} X_i$.
Consider also $(\mu_i)_{i\in I}$ a $C^k$-differential $\cP$-partitioned-form such that, for all $ i \in I$, $\mu_i$ is everywhere non-zero on $X_i$.
Let moreover $(\zeta_{m,i} \mid i\in I)_{m\in \bN}$ be a sequence of $C^k$-differential $\cP$-partitioned-forms, such that $\deg(\zeta_{m,i})=\deg(\mu_i)$ for all $m$ and all $i$.
We then say that the sequence $(\zeta_ {m,i} \mid i\in I)_m$ 
    \textbf{conformally $C^k$-converges} to $(\mu_i \mid i \in I)$, and denote as
    \[
    [(\zeta_{m,i})]\xrightarrow{C^k}[(\mu_i)] \; ,
    \]
    if there are open neighborhoods $O_i$ of $X_i$ for all $i\in I$ and a sequence of smooth partitioned-functions $(f_{m,i}\colon O_i\to \bR_{>0} \mid i \in I)_m$, such that
    \[
    f_{m,i} \,\widetilde{\zeta}_{m,i}  \xrightarrow{C^k} \widetilde{\mu}_i \quad \text{on } O_i \, ,
    \quad \text{as } m\to \infty \, ,
    \]
    where $\widetilde{\zeta}_{m,i}$ and $\widetilde{\mu_i}$ are $C^k$-extensions over $O_i$ of respectively $\zeta_{m,i}$ and $\mu_i$ (as in the definition of partitioned forms).

\medskip

Note that, if a partition $\cP'$ of $M$ is a \textbf{sub-partition} of another partition $\cP$, i.e.\ every subset in $\cP$ is union of subsets in $\cP'$, any $\cP$-partitioned-form can naturally be seen as a $\cP'$-partitioned-form as well.
In this case, we will hence allow ourselves a little abuse of notation, by talking about sequences of $\cP'$-partitioned-forms converging to a $\cP$-partitioned-form.


\subsubsection{Contact strict approximations}
\label{sec:contact_strict_approximations_subsection}

We now describe the notion of contact approximation which will allow us to prove stability of weak fillability, as described in the Introduction.
In order to do so, we make use of the rank partition $\cP=\{C_i \mid i=0,\ldots n\}$ of $M$ as recalled in \Cref{sec:high_dim_confoliations}. 

\begin{defn}
    \label{def:contact_strict_approx_and_deform}
    Let $k\in\{1,\ldots,\infty\}$.
    Given a c-symplectic confoliation $(\eta,\CS_{\eta,\mu})$, we call \textbf{contact strict $C^k$-approximation}, or simply \textbf{contact strict approximation} if $k=\infty$, of the latter any sequence of contact structures $\xi_m=\ker(\alpha_m)$, indexed by $m\in \bN$, such that the following properties are satisfied.
    \begin{enumerate}
        \item \label{item:contact_approx_and_deform__hyperplanes_convergence} As hyperplane fields, $\xi_m$ $C^k$-converges to $\eta$. More explicitly, this means that the sequence $(\alpha_m)_m$, intended as partitioned $1$-forms for the trivial partition, conformally converges to $\beta$.

        \item \label{item:contact_approx_and_deform__other_mixed_terms} 
        For all $h=0,\ldots, n$, there are 
        \begin{itemize}
        \item a sub-partition $\cP_h= \{C_{i,j}^h\mid 0\leq i \leq n \, , \; j\in J_i^h\}$ of the rank partition $\cP=\{C_i \mid 0\leq i \leq n\}$, where $J_i^h=\{*\}$ if $i\leq n-h-1$,
        
        \item a $C^k$-differential $\cP_h$-partitioned form $(\mu_{h,i,j} \mid 0\leq i \leq n, j\in J_i^h)$, where $\deg\mu_{h,i,j}=2(h-n+i)$ if $n-h\leq i\leq n$ and $\mu_{h,i,j}$ is the constant function $1$ if $0\leq i\leq n-h-1$,
        \end{itemize} 
        satisfying the following properties: 
        \begin{enumerate}
            \item \label{item:contact_approx_and_deform__conformal_convergence_other_mixed_terms}         the $C^k$-differential $\cP_h$-partitioned form 
                \[
                    (\alpha_m\wedge \d\alpha_m^h \text{ on $C_{i,j}^h$ , for } 1\leq i \leq n, j\in J_i^h)
                \]
                conformally $C^k$-converges to the $\cP_h$-partitioned form
                \[
                (\beta\wedge \d\beta^h  \text{ for } 0\leq i\leq n-h-1, j\in J_i^h \; ; \;\;
                \beta\wedge \d\beta^{n-i}\wedge \mu_{h,i,j}\text{ for } n-h \leq i \leq n \, , j\in J_i^h) \, ,
                \]
            \item \label{item:contact_approx_and_deform__positivity_limit_volume_forms}
            for a (hence every) extension $\widetilde\omega\in\Omega^2(M)$ of (any) $\omega\in \CS_{\eta,\mu}$, there is $\delta>0$ and a volume form $\Vol$ on $M$ such that the maximum-degree $C^k$-differential $\cP_h$-partitioned form 
            \[
             (\Vol \text{ for } 0\leq i\leq n-h-1, j\in J_i^h \; ; \;\;
              \beta\wedge \d\beta^{n-i}\wedge \mu_{h,i,j} \wedge \widetilde\omega^{n-h}  \text{, for }n-h\leq i \leq n, j\in J_i^h) \, ,
            \]
                is bounded below by $\delta \Vol$.
        \end{enumerate}
    \end{enumerate}

    \noindent
    A \textbf{contact strict $C^k$-deformation} is defined analogously to a contact strict $C^k$-approximation, but with a parameter  $s\in[0,1]$ instead of $m$, with the family being $C^k$ in the parameter $s$ too, and $s\to 0$ instead of $m\to \infty$.
\end{defn}

Several remarks about the above definition are in order.

\begin{rmk}
\label{rmk:contact_approx_and_deform_only_depends_on_omega_on_Keta}
\Cref{item:contact_approx_and_deform__positivity_limit_volume_forms} above only depends on $\widetilde\omega\vert_{\cK_\eta}=\mu\vert_{\cK_\eta}$.
In particular, it is indeed independent on the choice of $\omega\in \CS_{\eta,\mu}$ and of the extension $\widetilde\omega$.
(Note that this is fundamental for the statement of \Cref{thm:weak_fillability_stable_under_strict_contact_approx} to even make sense, due to the notion of weak filling for c-symplectic confoliations recalled in \Cref{sec:high_dim_confoliations}.)
\end{rmk}

\begin{rmk}
\label{rmk:3D_contact_approx}
Contact approximations and deformations of confoliations in ambient dimension $3$, as defined by Eliashberg--Thurston \cite{EliashbergThurston_Confoliations}, naturally satisfy \Cref{def:contact_strict_approx_and_deform} above.
This is simply due to the fact that $\alpha_m\wedge \d\alpha_m = f_m \Vol$, with $f_m\to 0$, by the contact condition.
\end{rmk}

\begin{rmk}
    \label{rmk:Keta_rank2_case_contact_approx}
    In the case where $\cK_\eta$ has rank at most $2$ at each point of $M$, \Cref{def:contact_strict_approx_and_deform} simplifies considerably and becomes equivalent to \cite[Definition 7.5]{CardonaGironella_ConfoliationsHighDim}.
    In the general case, the additional information guaranteed by \Cref{def:contact_strict_approx_and_deform} with respect to the notion in \cite{CardonaGironella_ConfoliationsHighDim} is the fact that some control is asked to the wedge products $\alpha_m\wedge \d\alpha_m^h$ with $h$ between $n-\frac 12\rk\cK_\eta+1$ and $n$.
    This is key for \Cref{thm:weak_fillability_stable_under_strict_contact_approx}, i.e.\ for the stability of weak fillability under (strict) contact approximations; c.f.\ \Cref{exa:old_notion_approx_not_enough_for_weak_fill}. 
\end{rmk}

\begin{rmk}
\label{rmk:some_mu_hij_are_determined_by_beta}
    Thanks to the requirement in \Cref{item:contact_approx_and_deform__hyperplanes_convergence}, $\alpha_m\wedge \d\alpha_m^h \to \beta\wedge \d\beta^h$ on $C_{i}$ with $n-h\leq i \leq n$.
    Thus the real requirement of \Cref{item:contact_approx_and_deform__conformal_convergence_other_mixed_terms,item:contact_approx_and_deform__positivity_limit_volume_forms} is the one for $n-h \leq i \leq n$.
\end{rmk}

\begin{rmk}
    \label{rmk:subpartition_necessary}
    Using partitioned differential forms is necessary for two reasons. 
    First, the maximal non-zero contribution of $(\eta,\CpS_{\eta})$ to the above wedge products is $\beta\wedge\d\beta^{n-i}$, which depends on $i$ in a ``discontinuous way''; as a result, the conformal factors needed to isolate the first converging error term in $\alpha_m\wedge \d\alpha_m^h$ will in general depend discontinuously on the point of $M$, with discontinuities at the interfaces between the different $C_i$'s.
    Second, even inside a single $C_i$, the dominating (as $m\to \infty$) factor of $\alpha_m\wedge \d\alpha_m^h$ might change in a discontinuous way depending on the point of $C_i\subset M$, as the following example shows.
\end{rmk}

\begin{rmk}
    \label{rmk:strict_contact_approximation}
    If $(\eta,\CS_{\eta,\mu})$ is a \emph{strict} c-symplectic confoliation, according to \Cref{lem:strict_c-symplectic_confoliation_positive_coefficients_polynomial}, up to using $\omega\in \CS_{\eta,\mu}$ of the form $\omega = \mu + T \d\beta$ for a very large $T>0$, we are allowed to replace  \Cref{item:contact_approx_and_deform__positivity_limit_volume_forms} with the following:
    \begin{itemize}
        \item[2.(b')] \label{item:contact_approx_and_deform__positivity_limit_volume_forms_bis} for a (hence every) extension $\widetilde\omega\in\Omega^2(M)$ of $\omega\in \CS_{\eta,\mu}$, there is $\delta>0$ and a volume form $\Vol$ on $M$ such that the maximum-degree $C^k$-differential $\cP_h$-partitioned form 
            \begin{multline*}
             \left( \, \beta\wedge\d\beta^h\wedge\widetilde\omega^{n-h} \text{ for } 0\leq i\leq n-h-1, j\in J_i^h \; ; \right. \\ \;\;
              \left. \beta\wedge \d\beta^{n-i}\wedge \mu_{h,i,j} \wedge \widetilde\omega^{n-h}  \text{, for }n-h\leq i \leq n, j\in J_i^h \, \right)
            \end{multline*}
                is bounded below by $\delta \Vol$.
    \end{itemize}
    In fact, as the statement of \Cref{thm:weak_fillability_stable_under_strict_contact_approx} clearly indicates, \Cref{def:contact_strict_approx_and_deform} is useful especially in the case of strict c-symplectic confoliations. 
    This is exactly what motivates the nomenclature of ``contact strict approximation''.
\end{rmk}

\medskip

We now describe some examples which show that some other, equally plausible, notions of contact approximations present some undesired properties, most notably the fact that they depend on the choice of extension $\widetilde \omega$ of $\omega\in \CS_{\eta,\mu}$.

\begin{exa}
Asking a lower bound by a strictly positive volume form, independent of $m$, in \Cref{item:contact_approx_and_deform__positivity_limit_volume_forms} will be important for the proof of \Cref{thm:weak_fillability_stable_under_strict_contact_approx} in \Cref{sec:proof_stability_weak_fill_under_contact_strict_approx}.
This is required as it is possible, for a family of polynomials depending on integers $m>0$, all of them positive for all $t\geq 0$, to have minima that converge to $0$ as $m\to \infty$.
(For instance, $P_m(t)=(t-m)^2+1/m$.)
\end{exa}

\begin{exa}
\label{exa:old_notion_approx_not_enough_for_weak_fill}
The notion of contact approximation from \cite[Definition 7.5]{CardonaGironella_ConfoliationsHighDim} is unfortunately not enough to guarantee stability of weak fillability analogously to \Cref{thm:weak_fillability_stable_under_strict_contact_approx,thm:weak_fillability_stable_under_tame_contact_approx}.
The issue is that \cite[Definition 7.5]{CardonaGironella_ConfoliationsHighDim} only requires the good compatibility conditions wit $\CS_{\eta,\mu}$ of \emph{the first order} ``correction'' given by the forms of the type $\alpha_m \wedge \d\alpha_m^{n-\frac 12\rk \cK_\eta+1}$, and does not guarantee anything about the higher order products $\alpha_m \wedge \d\alpha_m^{h}$ for $n-\frac 12\rk \cK_\eta+1<h\leq n$. 
(Note that the products $\alpha_m\wedge \d\alpha_m^h$ for $h\leq n-\frac 12\rk\cK_\eta$ are guaranteed to converge to $\beta\wedge\d\beta^{h}$ as we only look here at $C^k$-approximations for $k\geq 1$.)
An explicit (local) counterexample showing that \cite[Definition 7.5]{CardonaGironella_ConfoliationsHighDim} is not enough for stability of weak fillability is the following.

On $\bR^7$ with coordinates $(z,x_i,y_i)$, consider $\beta= \d z $, $\mu = \sum_{i=1}^3 \d x_i \wedge \d y_i $ and $\alpha_m = \d z + \frac{1}{m}x_1 \d y_1 - \frac{1}{m^2}x_2\d y_2 - \frac{2}{m^2}x_3\d y_3$.
Then an explicit computation gives
\[
    \beta\wedge (\mu + t \d\beta)^3=\beta \wedge \mu ^3 = 6\d z \wedge \d x_1 \wedge \d y_1 \wedge\d x_2 \wedge \d y_2 \wedge \d x_3 \wedge \d y_3 >0
\]
while 
\[
    \alpha_m \wedge (\mu + t\d\alpha_m)^3 = 6\left(1+\frac{t}{m}\right)\left(1-\frac{t}{m^2}\right)\left(1-\frac{2t}{m^2}\right) \d z \wedge \d x_1 \wedge \d y_1 \wedge\d x_2 \wedge \d y_2 \wedge \d x_3 \wedge \d y_3 \, ,
\]
which, for all $m\geq 1$, takes both positive and negative values for $t\in[0,\infty)$.
\end{exa}

\begin{exa}
    \label{exa:Mori_definition}
    In the spirit of the definition of $\epsilon\tau$-confoliation by Mori \cite{Mori_NoteMitsumatsuConstruction}, one could simply have asked that there is $m_0>0$ such that, for all $m\geq m_0$,
    \[
        \alpha_m \wedge (\widetilde \omega + t \d\alpha_m)^n >0 \text{ for all }t\geq 0 \, .
    \]
    However, this condition turns out to depend on the specific choice of extension $\widetilde \omega$ of $\omega\in \CS_{\eta,\mu}$. 
    (This is an undesirable dependence if one wants a statement such as \Cref{thm:weak_fillability_stable_under_tame_contact_approx,thm:weak_fillability_stable_under_strict_contact_approx}.) 
    Here is an explicit example.

    On $\bR^5$, consider $\beta=\d z$, $\widetilde\omega = \d x_1 \wedge \d y_1 + \d x_2 \wedge \d y_2$ and 
    \[
    \alpha_m = dz + \frac{x_1}{m}\d x_2 \, .
    \]
    Then, one has
    \[
        \alpha_m\wedge (\widetilde\omega + t\d\alpha_m)^2 = 2 \d z \wedge \d x_1 \wedge \d y_1 \wedge \d x_2 \wedge \d y_2 \, ,
    \]
    as the $\d x_1\wedge \d x_2$ term from $\d\alpha_m$ does not contribute any non-zero term to such maximum degree differential form.
    However, if we use instead $\widetilde \omega' = \widetilde \omega + \d z \wedge \d y_2$, which has same restriction as $\widetilde \omega$ to $\ker\beta=\ker\d z$, the fact that the $\d z$ contribution can now come also from $\widetilde\omega'$ instead of just from $\alpha_m$ leads to
    \[
        \alpha_m\wedge (\widetilde\omega' + t\d\alpha_m)^2 = 2\left(1-\frac{x_1}{m}\right) \d z \wedge \d x_1 \wedge \d y_1 \wedge \d x_2 \wedge \d y_2  \, ,
    \]
    which is instead negative at points with $x_1>m$.
\end{exa}


\subsection{Examples of contact, strict and tame, deformations}
\label{sec:examples_contact_strict_and_tame_approx_and_deform}

In this section we prove \Cref{prop:examples_contact_approx_and_deform} from the Introduction.
For clarity, we will describe the examples of contact strict and/or tame $C^\infty$-deformations separately.


\subsubsection{Branched covers and fiber sums}
\label{sec:branched_covers_fiber_sums}

We discuss here how contact branched covers and fiber sums as defined by Geiges in \cite{Geiges_BranchedCoversFiberSums} (c.f.\ also \cite{Gonzalo_ContactBranchedCovers} for the case of $3$-dimensional contact branched covers) can naturally be seen as contact strict and tame deformations of certain strict and tame c-symplectic confoliations, that have been introduced in \cite[Section 2]{Gironella_ConstructionsContactManifolds} (i.e.\ \cite[Section 5.2]{Gironella_PhDThesis}) and \cite[Section 5.3]{Gironella_PhDThesis}.
As in \cite{CardonaGironella_ConfoliationsHighDim}, we treat in detail the case of branched covers;
the case of fiber sums is completely analogous, hence we omit it.

Let $(M,\xi=\ker\alpha)$ be a contact $(2n+1)$-manifold, and consider $\pi\colon \widehat M \to M$ a smooth branched cover over a ``downstairs branching locus'' given by a codimension $2$ submanifold $B\subset M$.
Denote also $\widehat B = \pi^{-1}(B)$ the ``upstairs branching locus''.
Choose now any connection form $\gamma$ on the sphere normal bundle of $\widehat B$ in $\widehat M$.
Let also $\rho\colon \widehat M \to \bR_{>0}$ be a distance function from $\widehat B$, w.r.t.\ an auxiliary Riemannian metric. Up to rescaling the latter, we can arrange that all values in $(0,1]$ are regular for $\rho$, and their preimages are hence smooth copies of the sphere normal bundle of $\widehat B$.

Note that $\widehat\xi=\pi^*\xi$ is indeed a strict c-symplectic confoliation.
Indeed, the conformal partial-symplectic structure $\CpS_{\widehat \xi}=\pi^*(\CS_\xi)$ and the two-form $\rho\, \d \rho  \wedge \gamma$ are symplectically compatible $2$-forms on $\widehat\xi$.

For $\widehat\alpha:=\pi^*\alpha$, we then consider the family of hyperplane fields $\xi_s$ given by the kernel of the $1$-forms 
\[
    \alpha_s = \widehat\alpha + s \epsilon \chi(\rho) \rho^2 \gamma \, ,
\]
where $\chi\colon \bR \to [0,1]$ is a cut-off function that is $1$ near $0$ and $0$ near $1$.
According to the computations in \cite[Section 2.1]{Gironella_ConstructionsContactManifolds} (that are completely analogous to those in \cite{Geiges_BranchedCoversFiberSums}), if $\epsilon>0$ is small enough then $\xi_s$ is a contact structure for every $s\in(0,1]$.

As $\widehat\xi$ itself is contact on the complement of $\widehat B =\{\rho=0\}$, along which $\cK_{\widehat\xi}$ is of rank exactly $2$, 
the only thing to prove in order to check that \Cref{def:contact_strict_approx_and_deform} holds is that there is an extension $\Gamma_s$ of $(\alpha_s\wedge\d\alpha_s^{n})\vert_{\rho=0}$ on a neighborhood $U_\epsilon=\{\rho<\epsilon\}$ of $\{\rho=0\}$ such that, up to conformal rescaling, it converges to a differential form of the type $\widehat\alpha\wedge\d\widehat\alpha^{n-1}\wedge \mu$ that is positive at points of $\{\rho=0\}$.  

We then restrict to the open neighborhood $U$ of $\widehat B$ in $\widehat M$, and, up to shrinking $\epsilon$, we can assume that $\chi(\rho)=1$ over all of $U_\epsilon$.

As computed in \cite[Section 7.2]{CardonaGironella_ConfoliationsHighDim}, at each point of $U$, we then have 
\[
    \d\alpha_s = \d\widehat\alpha + s \epsilon( 2 \rho \, \d \rho \wedge \gamma + \rho^2 \d\gamma) \, ,
\]
and so
\begin{multline*}
\alpha_s\wedge(\d\alpha_s)^n
=
\widehat\alpha\wedge(\d\widehat\alpha)^n\\
+s\epsilon\Bigl[
2n\,\widehat\alpha\wedge(\d\widehat\alpha)^{n-1}
\wedge\rho\,\d\rho\wedge\gamma
+n\rho^2\,\widehat\alpha\wedge(\d\widehat\alpha)^{n-1}
\wedge\d\gamma
+\rho^2\gamma\wedge(\d\widehat\alpha)^n
\Bigr]
+O(s^2).
\end{multline*}

Set
\(\Gamma_s:=\alpha_s\wedge(\d\alpha_s)^n
- \widehat\alpha\wedge(\d\widehat\alpha)^n.\)
Since
\(\widehat\alpha\wedge(\d\widehat\alpha)^n\)
vanishes along \(\widehat B\), the form \(\Gamma_s\) is an extension of
\(
\left.
\alpha_s\wedge(\d\alpha_s)^n
\right|_{\widehat B}
\)
to \(U_\epsilon\).
By the above computation, we have
\[
\frac{1}{s}\Gamma_s
\xrightarrow{s\to0}
\epsilon\Bigl[
2n\,\widehat\alpha\wedge(\d\widehat\alpha)^{n-1}
\wedge\rho\,\d\rho\wedge\gamma
+n\rho^2\,\widehat\alpha\wedge(\d\widehat\alpha)^{n-1}
\wedge\d\gamma
+\rho^2\gamma\wedge(\d\widehat\alpha)^n
\Bigr].
\]

The first term on the right-hand side restricts to a positive volume
form along \(\widehat B\). Moreover, the remaining two terms are
multiplied by \(\rho^2\). Hence, after shrinking \(U_\epsilon\) if
necessary, the first term dominates the latter terms, and the above
limit is positive on \(U_\epsilon\). This proves the required
conformal convergence and concludes the proof.
\hfill\qedsymbol


\subsubsection{Bourgeois structures}
\label{sec:bourgeois_structures}

Bourgeois constructed in \cite{Bourgeois_ToriAreContact} an explicit contact structure $\xi$ on $M\times \bT^2$ starting from a contact structure $\eta$ on $M^{2n+1}$ equipped with a supporting open book decomposition.
More precisely, let the open book be presented as a pair $(B,\theta)$ of a codimension $2$ contact submanifold with a trivial normal bundle $B\times D^2$ and a fibration $\theta\colon M\setminus B \to \bS^1$ that restricts to the projection onto the angular coordinate of $D^2$ on $B\times (D^2\setminus \{0\})$. 
Let also $\beta$ a defining $1$-form for $\eta$ that is adapted to such open book, in the sense of Giroux \cite{Giroux_OpenBooks}.
Then, for a function $\rho$ that equals to the radial coordinate of $D^2$ on $B\times D^2$ near $B$, is constant away from $B\times \{0\}$ in $B\times D^2$, and satisfies $\rho'(r)\geq 0$ everywhere, the $1$-form 
\[
    \alpha_s = \beta + s \phi_1 \d x_1 - s \phi_2 \d x_2 
\]
is contact on $M\times \bT^2$ for all $s>0$, where $(x_1,x_2)$ are coordinates in $\bT^2$ and $(\phi_1,\phi_2)=\rho(r)\cdot (\cos\theta,\sin\theta)$.

\paragraph{Contact strict deformation of the split confoliation.} 
We start by proving that $\xi_s:=\ker\alpha_s$ is a contact strict $C^\infty$-deformation of the strict c-symplectic confoliation $(\xi_0:=\eta \oplus T\bT^2, \CS_{\xi_0,\omega_T})$, where $\omega_T$ is any area form on $\bT^2$.
We assume for simplicity that $\omega_T = \d x_1 \wedge \d x_2$.

First, said c-symplectic confoliation is indeed strict, because \eqref{eqn:polynomial_form} reduces here to simply 
\[
    \beta\wedge(\omega_T + t\d\beta)^{n+1} = (n+1)t^{n}\beta\wedge\d\beta^{n}\wedge \omega_T \, .
\]
Similarly, the contact $C^\infty$-deformation is indeed strict as we simply have
\begin{align*}
    \alpha_s\wedge(\omega_T + t\d\alpha_s)^{n+1} & \, = (n+1)t^{n}\alpha_s\wedge\d\alpha_s^{n}\wedge \omega_T + t^{n+1} \alpha_s \wedge \d\alpha_s^{n+1} \\
    & \, = (n+1)t^{n}\beta\wedge\d\beta^{n}\wedge \omega_T + t^{n+1} \alpha_s \wedge \d\alpha_s^{n+1} \, ,
\end{align*}
with the last term being equal, according to \cite{Bourgeois_ToriAreContact} (or an explicit computation), to
Similarly, we have
\begin{align*}
\alpha_s\wedge
(\omega_T+t\,d\alpha_s)^{n+1}
={}&
(n+1)t^n
\alpha_s\wedge(d\alpha_s)^n\wedge\omega_T
+t^{n+1}\alpha_s\wedge(d\alpha_s)^{n+1}\\
={}&
(n+1)t^n
\beta\wedge(d\beta)^n\wedge\omega_T
+t^{n+1}\alpha_s\wedge(d\alpha_s)^{n+1}.
\end{align*}
Moreover, a direct computation gives
\[
\begin{aligned}
\alpha_s\wedge(d\alpha_s)^{n+1}
={}&
s^2(n+1)
\Big[
(d\beta)^n\wedge
(\phi_1d\phi_2-\phi_2d\phi_1)\\
&\qquad\qquad
+n\,\beta\wedge(d\beta)^{n-1}
\wedge d\phi_1\wedge d\phi_2
\Big]
\wedge dx_1\wedge dx_2.
\end{aligned}
\]
The form in brackets is a positive volume form on \(M\), since the
open book supports the contact structure \(\eta\). Hence, after rescaling
by \(s^{-2}\), the top-degree term converges to a positive volume form
on \(M\times T^2\).
Rescaled by $1/s^2$, this is strictly positive by the fact that the open book on $M$ supports $\eta$, and hence \Cref{def:contact_strict_approx_and_deform} is satisfied, as desired.
\hfill\qedsymbol

\paragraph{Contact strict deformation of the spinal open book confoliation.} 
We now pass to showing that Bourgeois structures can also be seen as a contact $C^\infty$-deformations, first of strict type (and later also of tame type), of another special (strict) c-symplectic confoliation on $M\times \bT^2$.
Let us first start by recalling the geometric decomposition behind this second deformation.

As reported e.g.\ in \cite[Page 13 of arXiv:v1]{BowdenGironellaMoreno_TightFillBourgeois}, it was first observed by Lisi that the Bourgeois contact structures on $V=M\times \bT^2$ are supported by a ``spinal open book decomposition''.
Without too much abstraction (as the case at hand is very special and concrete), we can use as ``definition'' of \textbf{spinal open book decomposition} that of a geometric decomposition of a contact manifold $(V,\xi)$ that generalizes the notion of open book decomposition in such a way that the mapping torus is replaced by a fibration $V_P$, called \textbf{paper}, over a strict contact manifold $(Y,\alpha_Y)$ with Liouville fibers $(\Sigma,\lambda_\Sigma)$, and the neighborhood of the binding is replaced by a trivial fibration $V_S$, the \textbf{spine}, over a Liouville filling $(W,\lambda_W)$ of $(Y,\alpha_Y)$ and with fibers the contact boundary $(B,\alpha_B)$ of $(\Sigma,\lambda_\Sigma)$.

In the concrete case of the Bourgeois, the spine is simply $V_S=B\times DT^*_{1-\epsilon}\bT^2$, where $DT^*_{1-\epsilon}\bT^2:=\{\Vert p\Vert\leq 1-\epsilon\}\subset DT^*\bT^2$, equipped with the contact form $\alpha_B + \lambda_T$. 
The paper $V_P$ is a fibration over $(ST^*\bT^2=\bT^3,\alpha_{T}=\lambda_T\vert_{ST^*\bT^2})$ with fibers $(\Sigma^\epsilon, \lambda_\Sigma)$, given by the complement of an $\epsilon$-neighborhood of the boundary (in the standard normal neighborhood  given by the Liouville flow, of the form $((1-\delta,1]\times B , \d\lambda_\Sigma = \d(r\alpha_B))$ with $\alpha_B:=\lambda_\Sigma\vert_B$) of the pages of the open book decomposition on $M$.
Here, $V_P$ is considered equipped with a contact form $\alpha_P$ of the form $C\alpha_T + \lambda_{fib}$, for some big $C>0$, where the first $\alpha_T$ is intended as the pullback of the one on $ST^*\bT^2 $ under the projection $V_P \to ST^*\bT^2$, and $\lambda_{fib}$ is a fiberwise Liouville form restricting to $\lambda_{T}$ on each fiber, given by the starting contact open book.
Lastly, there is an interface region $V_I:=  (1-2\epsilon,1+2\epsilon)_r\times ST^*\bT^2\times B$ thanks to which the two  contact forms on $V_S$ and $V_P$ above are interpolated, with a contact form $\alpha_I$ of the type $f(r)\alpha_B + g(r) \alpha_T$, with $f$ and $g$ positive and $f'g-g'f<0$ for all $r\in (1-2\epsilon,1+2\epsilon)$ by the (positive) contact condition (here, that $V_I=(1-2\epsilon,1+2\epsilon)\times ST^*\bT^2\times B$ is considered with the product orientation induced by $\partial_r$, $\alpha_T$ and $\alpha_B$). 
Such $f$ and $g$ additionally need to satisfy that 
\[f(r)=1 \, , \; g(r) = e^{r-2+\epsilon} \text{ for } r< 1-\epsilon \, ,
\text{ and } 
f(r) = e^{-r+2-\epsilon} \, , \;  g(r)=C \text{ for } r> 1+\epsilon \, ,\]
in order to be able to glue the interface region to $V_S$ and $V_P$.
We also ask that $f'<0$ and $g'>0$ wherever they are not already required to be constant.
(Note that this is possible if $C>0$ is big enough, which we assume to be the case.)

Using this geometric decomposition and explicit description of the Bourgeois contact form (up to isotopy of the contact structure), the latter can naturally be seen as being the time $s=1$ of the following family of $1$-forms:
\[
\alpha_s' = \begin{cases}
    \alpha_B + s \lambda_T & \text{on }V_S=B \times DT^*_{1-\epsilon}\bT^2 \, , \\
    f_s \alpha_B + g_s \alpha_T & \text{on }V_I=(1-2\epsilon,1+2\epsilon)\times ST^*\bT^2 \times B  \, , \\
    s \lambda_{fib} + C\alpha_T & \text{on }V_P \, , \\
\end{cases}
\]
where $f_s$ and $g_s$ are $s$-parametric functions of $r\in (1-2\epsilon,1+2\epsilon)$ satisfying the following properties:
\begin{itemize}
    \item $f_1=f$, $g_1=g$;
    \item $f_s'< 0$ and $g_s'> 0$ for $s>0$ wherever they are not required to be constant;
    \item $f_0'<0$ and $g_0'>0$ for $1-\epsilon <r<1+\epsilon$;
    \item $f'_sg_s-g'_sf_s<0$ for all $r\in (1-2\epsilon,1+2\epsilon)$ and $s>0$, and
    \item $f_s(r)=1$ and $g_s(r) = se^{r-2+\epsilon}$  for $r< 1-\epsilon$;
    \item $f_s(r) = se^{-r+2-\epsilon}$, $g_s(r)=C$  for $r> 1+\epsilon$;
    \item for later purposes, we can also arrange that $f_s$ is of the form $sf_1 +(1-s)f_0$, and $g_s$ of the form $sg_1+(1-s)g_0$.
\end{itemize}

One can explicitly check that $\alpha_s'$ defines a contact structure $\xi_s'$ for all $s>0$, and that $\alpha_0'$ defines a smooth confoliation $\eta'$.
Note that $\cK_{\eta'}$ is of rank $4$ on $V_S$ and on the region $\{r\leq 1-\epsilon\}$ of $V_I$, of rank $2n-2$ on $V_P$ and on the region $\{r\geq 1+\epsilon\}$ of $V_I$, and of rank $0$ (i.e.\ $\eta'$ is contact) on $\{1-\epsilon < r < 1+ \epsilon\}$.
Moreover, $\eta'$ can be promoted to a c-symplectic confoliation via $\CS_{\eta',\mu'}$, where $\mu' $ is the following $2$-form\footnote{Strictly speaking, this $\mu'$ is degenerate on $\eta'$ as it has non-trivial kernel; however, up to adding a small multiple of $\d\alpha_0'$, it does become non-degenerate, and hence it defines indeed a well defined symplectic cone together with $\CpS_{\eta'}$, as required.\label{ftnt:muprime_degenerate}}:
\[
\mu' = \begin{cases}
    \d\lambda_T & \text{on }V_S=B \times DT^*\bT^2 \, , \\
    \rho(r)\,\d\!\left(e^{r-2+\epsilon}\alpha_T\right)
+
\bar\rho(r)\,\d\!\left(e^{-r+2-\epsilon}\alpha_B\right) & \text{on }V_I=(1-2\epsilon,1+2\epsilon)\times ST^*\bT^2 \times B \, , \\
    \d_{fib} \lambda_{fib} & \text{on }V_P \, , \\
\end{cases}
\]
The region \(V_P\) appears twice here, and the second occurrence should presumably be \(V_I\). 
Moreover, it seems that the interface term should be
\[
\rho(r)\,\d\!\left(e^{r-2+\epsilon}\alpha_T\right)
+
\bar\rho(r)\,\d\!\left(e^{-r+2-\epsilon}\alpha_B\right),
\]
so that it agrees with \(\d\lambda_T\) on the spine side and with
\(\d_{fib}\lambda_{fib}\) on the paper side. If so, the subsequent computation of
\(\alpha_s'\wedge(\mu'+t\,\d\alpha_s')^{n+1}\)
would also need to be modified accordingly. Nevertheless, it seems that the desired positivity and strict-deformation conclusion should still follow after these corrections.
where $\d_{fib}$ denotes the exterior derivative in the fibers direction, and $\rho,\bar\rho\colon (1-2\epsilon, 1+2\epsilon)\to \bR_{\ge0}$ with $\bar\rho(r)=\rho(2-r)$ 
and $\rho$ equals $1$ for $r\in (1-2\epsilon , 1-3\epsilon/2)$ and $0$ for $r\in[1-\epsilon,1+2\epsilon)$, and $\rho'<0$ on $(1-3\epsilon/2,1-\epsilon)$.
Note that $\rho'< 0$ where $f_0'=0$ and $\bar\rho'> 0$ where $g_0'=0$.

One can explicitly check that $\alpha_s'\wedge (\mu' + t\d\alpha_s')^{n+1}$ is a volume form on $V_S$ and $V_P$.
On the interface region \(V_I\), we have
\begin{equation}
\label{eqn:polynomial_bourgeois_spinal_open_book}
\begin{aligned}
&\alpha_s'\wedge
(\mu'+t\,\d\alpha_s')^{n+1}\\
={}&
n(n+1)
\Bigl[
g_s\left(tf_s'-e^{-r+2-\epsilon}\bar\rho\right)
-
f_s\left(tg_s'+e^{r-2+\epsilon}\rho\right)
\Bigr]\\
&\cdot
\left(tf_s+e^{-r+2-\epsilon}\bar\rho\right)^{n-1}
\left(tg_s+e^{r-2+\epsilon}\rho\right)
\d r\wedge\alpha_B\wedge(\d\alpha_B)^{n-1}
\wedge\alpha_T\wedge\d\alpha_T
>0.
\end{aligned}
\end{equation}
for every \(s\geq0\) and \(t\geq0\).

Now, thanks to the above choices, analyzing the sign of each of the involved functions one can see that each $t$-coefficient is strictly positive for $s>0$.

Moreover, for $s=0$, the terms in the above \eqref{eqn:polynomial_bourgeois_spinal_open_book} are all non-negative, and the ones coming from products of the form $\alpha_0'\wedge (\d\alpha_0')^h$ for $h$ at most the corank of $\cK_{\eta'}$ (that changes depending on the region of the manifold, as described above) are in fact also strictly positive.
In particular, this means that the above c-symplectic confoliation $(\eta',\CS_{\eta',\mu'})$ is strict.

Lastly, by what already said above, in order to check \Cref{def:contact_strict_approx_and_deform} is satisfied for $\xi_s'$ and $(\eta',\CS_{\eta',\mu'})$, it remains to check that the wedge products of the form $\alpha_s'\wedge (\d\alpha_s')^h$, for $h\geq n-\frac 12\rk\cK_{\eta'}+1$, converge, for $s\to 0$ and up to conformal rescaling, to differential forms of the type $\alpha_0'\wedge (\d\alpha_0')^{n-\frac 12\rk\cK_{\eta'}}\wedge \mu_h$, which after wedging with $(\mu')^{n-h}$ become positive volume forms.

In order to prove so, first notice that $\alpha_s\wedge \d\alpha_s^h$ are all polynomial in $s$, hence they do converge to some non-trivial differential form for $s\to 0$, up to rescaling by some power of $s$. 
More precisely, one can explicitly compute that, for the relevant cases of $h$ dictated by the ranks of $\cK_{\eta'}$ previously described, we have that
\[
    \alpha_s'\wedge (\d\alpha_s')^h =
    \begin{cases}
        K s^{h-n+1} \alpha_B\wedge \d\alpha_B^{n-1}\wedge \d\lambda_T^{h-n+1} + O(s^{h-n+2}) & \text{ on $V_S$ and for $n-1<h\leq n$}\\
        K's^{h-1} C^2 \alpha_T \wedge \d\alpha_T \wedge \d_{fib}\lambda_{fib}^{h-1} + O(s^h) & \text{on $V_P$ and for $1<h\leq n$,}
    \end{cases}
\]
where $K$ and $ K'$ are two positive constants (coming from binomial coefficients).
Note that no computation is needed on the region $V_I$, as on $V_I \cap \{1-\epsilon < r <1+\epsilon \}$ the $1$-form $\alpha_0'$ is itself contact, and $V_I \cap (\{r<1-\epsilon \} \cup \{ r >1+\epsilon \})$ is in fact the gluing region, and hence already accounted for in $V_S\cup V_P$.
From these formulas it is then straightforward to see that the last conditions required to verify \Cref{def:contact_strict_approx_and_deform} for $\alpha_s'$ are indeed verified, thus concluding the proof that $\xi_s'$ is a strict contact deformation of $(\eta',\CS_{\eta',\mu'})$.

\paragraph{Contact tame deformation of the spinal open book confoliation.} 
We use here the computations of the previous paragraph, and we now prove that $\xi_s'$ is also a contact \emph{tame} deformation of $(\eta',\CS_{\eta',\mu'})$.
We work on each piece of the manifold $V_S$, $V_P$ and $V_I$ separately; the candidate complex structures we give on each piece will then glue well to a global complex structure.

Let us start from the spine.
We can split 
\[
TV_S
=
\langle R_B\rangle\oplus\xi_B
\oplus
\langle\partial_{p_1},
\partial_{q_1}-sp_1R_B\rangle
\oplus
\langle\partial_{p_2},
\partial_{q_2}-sp_2R_B\rangle.
\]
where $R_B$ is the Reeb vector field of $\alpha_B$ on the $B$ factor, $\xi_B$ is the kernel of $\alpha_B$ interpreted again on the $B$ factor, and where we use coordinates $(p_i,q_i)$ on the factor $DT^*\bT^2$. Here, we use as complex structure $J^S_s$ on $\xi_s'$ defined as follows: it preserves $\xi_B$ and restricts to it as an $s$-independent complex structure $J^B$ tamed by $\d\alpha_B$, and it sends $\partial_{p_i}$ to $\partial_{q_i}-sp_iR_B$.
One can then explicitly check that $J_s^S$ is tamed by $\d\alpha_s'\vert_{V_S}$ and that it converges to some $J^S$ on $\eta'\vert_{V_S}$ that is tamed by $\CpS_{\eta'}$ and $\mu'$.

Let's now pass to the paper.
Denoting by $\VertBdle\subset TV_P$ the vertical bundle of the fibration structure of $V_P$ over $ST^*\bT^2$, we can consider its $\d\lambda_{fib}$-orthogonal $\HorBdle:=\VertBdle^{\perp_{\d\lambda_{fib}}}$ (which is independent of $s$).
We can then further decompose $\HorBdle$ via horizontal lifts $R_T^{\HorBdle}$ and $\xi_T^{\HorBdle}$ of, respectively, the Reeb vector field $R_T$ of $\alpha_T$ and the contact structure $\xi_T=\ker\alpha_T$ on $ST^*\bT^2$.
Let's also fix an auxiliary fiberwise $J_{fib}$ on $\VertBdle$ that is $\d_{fib}\lambda_{fib}$-tamed, and another auxiliary complex structure $J_T$ on $\xi_T$ tamed by $\d\alpha_T$.
We can then construct a complex structure $J_s^P$ on 
\[
\xi_s'\vert_{V_P}=\xi_T^{\HorBdle}\oplus \widetilde\VertBdle_s  \, ,
\text{ where } \widetilde\VertBdle_s :=  \langle CX - s\lambda_{fib}(X) R_T^{\HorBdle} \mid  X \in \VertBdle\rangle
\]
by declaring: $J_s^P \vert_{\xi_T^\HorBdle}$ is the pullback $J_T^{\HorBdle}$ under the projection map $\HorBdle\to \xi_T$ of $J_T$; on $\widetilde\VertBdle_s$ we define $J_s^P(\widetilde X)=\widetilde Y$ if $\widetilde X = CX - s\lambda_{fib}(X) R_T^{\HorBdle}$ and $\widetilde Y= CJ_{fib}X - s\lambda_{fib}(J_{fib}X) R_T^{\HorBdle}$.
Again, a direct computation shows that this is tamed by $\d\alpha_s'$, and that it converges to $J^P=J_{fib}\oplus J_T^{\HorBdle}$ on $\eta'\vert_{V_P}= \VertBdle\oplus \xi_T^{\HorBdle}$, that is tamed by $\CpS_{\eta'}$ and $\mu'$.

Lastly, it is only left to deal with the interface region $V_I$.
Here, we can split
\[
    \xi_s'\vert_{V_I} = \langle \partial_r , f_s R_T - g_s R_B \rangle \oplus  \xi_B \oplus \xi_T \, .
\]
Hence, for $J_B$ and $J_T$ as previously considered, we can define $J_s^I$ on $\xi_s'$ at all points of $V_I$ as follows:
$J_s^I(\partial_r) = f_s R_T - g_s R_B$, $J_s^I\vert_{\xi_B}=J_B$ and $J_s^I\vert_{\xi_T}=J_T$.
Note that, in order for this to make sense, we need that $f_s^2+g_s^2\neq 0$ for all $s\geq 0$; this is indeed the case by our choices above.
An explicit computation shows that this is indeed tamed by $\d\alpha_s'$ on $\xi_s'$ (if $C>0$ is big enough, which we assume).
Moreover, such $J_s^I$ converges to a certain $J^I$ on $\eta'\vert_{V_I}$, which is clearly tamed by $\CpS_{\eta'}$ and also by $\mu'+\delta \d\alpha_0$ (c.f.\ the footnote\footref{ftnt:muprime_degenerate}) by the choice of the cut-off functions $\rho$ and $\bar\rho$.
This concludes the proof.


\subsubsection{Mori's deformation}
\label{sec:mori_deformation}

All the needed computations are done in \cite[Section 7.2]{CardonaGironella_ConfoliationsHighDim}, following \cite[Section 3]{Mori_NoteMitsumatsuConstruction}. 
We report them here for completeness, using the notations in those two references, without explicitly reintroducing everything and just recalling the necessary formulas.

Recall that the main aim of \cite{Mori_NoteMitsumatsuConstruction} was the description of a certain type of contact deformation (called $\epsilon\tau$-deformation) of the codimension-$1$ symplectic foliation $(\cF,\tau_\cF)$ on $\bS^5$ found by Mitsumatsu \cite{Mitsumatsu_SymplecticLawsonFoliation}.

Firstly, note that, by \Cref{rmk:3Dconf_and_fol_and_weakly-filled-contact_are_strict}, $(\cF,\tau_\cF)$ is in particular a strict c-symplectic confoliation.

Then, recall that in \cite{Mori_NoteMitsumatsuConstruction}  Mori uses an open book decomposition of $\bS^5$ in order to describe his $\epsilon\tau$-deformation.
The used open book is a contact one, but it has the additional nice property that the binding $N^3$ also admits a ``cosymplectic structure'', i.e.\ a closed non-vanishing $1$-form $\nu$ and a closed $2$-form $\tau_N$ that is non-degenerate on $\ker\nu$; in this special case, we also know that $\tau_N$ is of the form $\eta\wedge \gamma$.
Moreover, as detailed e.g.\ in \cite[Section 3.1]{GironellaToussaint_ConformalSymplecticFoliations}, we also know that $\d\eta = \gamma \wedge \nu$, and hence $\eta$ is a contact form on $N$, inducing the contact structure at the binding of the contact open book.

Now, checking that our definition of contact strict $C^\infty$-deformation is satisfied over the mapping torus part of the open book is straightforward, thanks to the explicit forms of the family of contact forms $\alpha_s$ and of the ambient extension $\tau$ of $\tau_\cF$ (the leafwise symplectic structure on $\cF$).
Hence, we just deal more explicitly with the remaining part of $\bS^5$, i.e.\ with the neighborhood $N\times D^2$ of the binding.

Here, the family of $1$-forms giving the candidate strict contact deformation is given by
\[
    \alpha_s = (1-s)f_0\nu + s f_s \eta + g_s \d\theta + (1-s)h\d\rho \, .
\]
Then, most of the necessary computations have been done explicitly in \cite[Section 7.2]{CardonaGironella_ConfoliationsHighDim}.
Namely, we have that, where $f_0\neq 0$, i.e.\ for $\rho<1/2$, we have
\[
\alpha_s\wedge \d\alpha_s=s\alpha_0\wedge \zeta + O(s^2)\, ,
\]
where
\[
    \zeta = g_0 \d\theta \wedge \gamma + h \d\rho \wedge \gamma + \left( g_1' - \frac{f_0'}{f_0}g_1 \right) \d\rho \wedge \d\theta \, .
\]
Moreover, it is also computed in \cite{CardonaGironella_ConfoliationsHighDim} that, always for $\rho<1/2$,
\[
    \alpha_0\wedge \zeta \wedge \tau  = (f_0 g_1' - f_0' g_1) \delta \nu \wedge \omega \wedge \d\rho \wedge \d\theta \, .
\]
Using the signs listed in \cite[Page 9]{Mori_NoteMitsumatsuConstruction}, one can see that each of this is indeed strictly positive on $\{\rho<1/2\}$.

In order to finish checking \Cref{def:contact_strict_approx_and_deform}, all is left to do is hence to understand what is the dominating term of $\alpha_s\wedge \d\alpha_s$ on $\{\rho\geq 1/2\}$, and prove that it can also be written as $\alpha_0\wedge \mu'$ for some $\mu'$ satisfying that $\alpha_0\wedge \mu'\wedge \tau>0$ on $\{\rho\geq 1/2\}$.

According again to the computations in \cite[Section 7.2]{CardonaGironella_ConfoliationsHighDim}, in this region $\{\rho\geq 1/2\}$, where in particular $f_0=0$, we have 
\[
    \alpha_s\wedge \d\alpha_s=s^2\alpha_0\wedge \zeta' + O(s^3) \, ,
\]
with 
\[
\zeta' = f_1 \gamma \wedge \nu + \frac{f_1 g_0' - f_1' g_0}{g_0^2 + h^2} \Big( h \d\theta \wedge \eta - g_0 \d\rho \wedge \eta \Big) \, .
\]
Lastly, again according to \cite{CardonaGironella_ConfoliationsHighDim}, we have
\[
    \alpha_0\wedge\zeta' \wedge \tau = (f^2_1 g_0' - 2f_1 f'_1 g_0) \d\rho \wedge \d\theta \wedge \eta \wedge \gamma \wedge \nu + \delta (f_1 g_0' - f'_1 g_0) \d\rho \wedge \d\theta \wedge \eta \wedge \omega \, ,
\]
that is non-negative according to the signs of the involved functions as listed in \cite[Page 9]{Mori_NoteMitsumatsuConstruction}.
This concludes.
\hfill\qedsymbol


\subsubsection{Conformal symplectic foliations}
\label{sec:conformal_symplectic_foliations}

The (naive) linear contact deformations given in \cite{BertelsonMeigniez_ConformalSymplecticFoliations} are defined by the $s$-parametric family of $1$-forms $\alpha_s:= \beta + s\lambda$, where $\cF=\ker\beta$ is a smooth codimension $1$ foliation and $\lambda$ is a $1$-form on $M$ restricting to a leafwise (locally) conformal Liouville structure on $\cF$, with respect to a leafwise Lee form given by the restriction to $\cF$ of a certain $1$-form $\eta$ on $M$ which also coincides with the holonomy form of $\beta$, i.e.\ satisfies $\d\beta = \eta \wedge \beta$.
We will consider $\cF$ equipped with the leafwise conformal symplectic structure $\d_\eta\lambda$ as our limit c-symplectic confoliation.

\bigskip

Let us start by checking that this linear deformation can be realized as a contact tame $C^\infty$-deformation. For this, we need to find a family $J_s$ of complex structures on $\xi_{s}:=\ker\alpha_s$ which are tamed by $\CS_{\xi_s}$ and converge, as $s\to 0$, to some complex structure $J$ on $\cF$ that is tamed by $\d_\eta\lambda$.

For this, notice that 
\begin{equation}
\label{eqn:dalphas_berteleson-meigniez_example}
    \d\alpha_s = \d\beta + s\d\lambda 
    = \eta\wedge \beta + s\d\lambda
    =\eta\wedge \alpha_s + s\d_\eta\lambda \, ,
\end{equation}
hence the restriction $\d\alpha_s\vert_{\xi_s}$, that defines $\CS_{\xi_s}$, coincides just with the conformal class of $\d_\eta\lambda\vert_{\xi_s}$.
In particular, one can choose any $\d_\eta\lambda-$tamed complex structure $J$ on $\cF$, and then, by the openness of the taming condition and the fact that $\xi_s\xrightarrow{s\to 0}\cF$, any $C^\infty$-family $(J_s)_{s\in [0,\epsilon)}$ of complex structures on $\xi_s$ extending $J_0:=J$ on $\cF$ will be tamed by $\d_\eta\lambda\vert_{\xi_s}$ for $s$ small enough, as desired.

\bigskip

Let us now pass to proving that the linear deformation in \cite{BertelsonMeigniez_ConformalSymplecticFoliations} can also be seen as a contact strict deformation of $(\cF,\d_\eta\lambda)$.
First note that, according to \Cref{rmk:3Dconf_and_fol_and_weakly-filled-contact_are_strict}, $(\cF,\d_\eta\lambda)$ is indeed a strict c-symplectic confoliation.

We next show that the same family is a contact strict deformation of
\((\cF,\d_\eta\lambda)\). For every \(h=1,\ldots,n\), one has
\[
\alpha_s\wedge(\d\alpha_s)^h
=
s^h\alpha_s\wedge(\d_\eta\lambda)^h,
\]
because every term containing
\(\alpha_s\wedge\eta\wedge\alpha_s\) vanishes. Consequently,
\[
s^{-h}\alpha_s\wedge(\d\alpha_s)^h
\longrightarrow
\beta\wedge(\d_\eta\lambda)^h
\qquad\text{as }s\to0.
\]
After wedging with \((\d_\eta\lambda)^{n-h}\), we obtain
\[
s^{-h}\alpha_s\wedge(\d\alpha_s)^h
\wedge(\d_\eta\lambda)^{n-h}
\longrightarrow
\beta\wedge(\d_\eta\lambda)^n,
\]
which are strictly positive volume forms.
In other words, $\xi_s$ and $(\cF,\d_\eta\lambda)$ satisfy \Cref{def:contact_strict_approx_and_deform}, as desired.
\hfill\qedsymbol


\subsubsection{Giroux torsion domains}
\label{sec:giroux_torsion_domains}

We are now going to describe how the contact structures $\xi_k$ from \cite[Theorem A]{MassotNiederkrugerWendl_WeakFillings} on manifolds of the form $M\times \bT^2$, for some special $M$, can be seen as a contact strict $C^\infty$ deformation of a c-symplectic foliation on $M\times \bT^2$. 
We follow here the computations from \cite[Proposition 1.8, Item 5]{CardonaGironella_ConfoliationsHighDim} (adapting slightly the notation to better fit our conventions).

\medskip

Recall that $\xi_k$ on $M\times \bT^2$ from \cite{MassotNiederkrugerWendl_WeakFillings} is given by the kernel of
    \[
    \alpha_k = \frac{1+\cos (kx)}{2} \alpha_+ + \frac{1-\cos (kx)}{2}\alpha_- + \sin(kx) \d y \, ,
    \]
where $(x,y)$ are coordinates on $\bT^2$ and $(\alpha_+,\alpha_-)$ is a Liouville pair on $M$.
Moreover, $M$ is the quotient of $\bR^n\times \bR^{n+1}$ under a certain co-compact lattice for the action
    \begin{equation*}
    (\tau_1,\ldots,\tau_n)\cdot(t_1,\ldots,t_n,\theta_0,\ldots,\theta_n) 
    := (t_1+\tau_1 ,\ldots, t_n + \tau_n , 
    e^{-(\tau_1+\ldots + \tau_n)}\theta_0, e^{\tau_1}\theta_1,\ldots, e^{\tau_n}\theta_n)\, ,
    \end{equation*}
and the covering space $\bR^n\times \bR^{n+1}$, we have
    \[
    \alpha_\pm = \pm e^{t_1+\ldots + t_n}\d\theta_0 +{e^{-t_1}\d\theta_1 + \ldots + e^{-t_n} \d\theta_n} \, .
    \]

Following \cite{CardonaGironella_ConfoliationsHighDim}, one can first change the contact structures $\xi_k$ by a contact isotopy, given by Gray's stability, to contact structures $\eta_k$ defined by 
\[
    \beta_k = \d x + \alpha_k = \d x +  \frac{1+\cos (kx)}{2} \alpha_+ + \frac{1-\cos (kx)}{2}\alpha_- + \sin(kx) \d y \, .
\]
After this preliminary change, for $s\in[0,1]$, consider the family 
    \[
    \beta_{k,s} = \d x + s\alpha_k = \d x +  s\left[\frac{1+\cos (kx)}{2} \alpha_+ + \frac{1-\cos (kx)}{2}\alpha_- + \sin(kx) \d y\right] \, .
    \]
For $s=0$ we have $\beta_{k,0} = \d x$ defining the foliation $\eta_0=\ker(\d x)$.
This is considered as a c-symplectic confoliation, together with $\CS_{\eta_0,\omega}$ where $\omega$ is defined by (the quotient of) the following $2$-form on $\bR^n\times \bR^{n+1}$:
    \[
    \omega = e^{t_1+\ldots + t_n}\d\theta_0 \wedge \d y - \sum_{i=1}^n e^{-t_i} \d t_i \wedge \d\theta_i \, .
    \]
We then claim that $\xi_{k,s}=\ker(\beta_{k,s})$ gives, for all $k\geq 1$ and as $s\to 0$, a contact strict $C^\infty$-deformation of the strict c-symplectic confoliation $(\eta_0,\CS_{\eta_0,\omega})$.

First, notice that $(\eta_0,\CS_{\eta_0,\omega})$, as a c-symplectic confoliation, is indeed strict according to \Cref{rmk:3Dconf_and_fol_and_weakly-filled-contact_are_strict}.
Then, we need to check \Cref{def:contact_strict_approx_and_deform}.

For this, we start by computing
    \begin{align*}
    \d\alpha_k 
    = & \,  \frac{k}{2}\sin(kx) \d x \wedge (-\alpha_++\alpha_-) + \frac{1+\cos (kx)}{2} \d\alpha_+ \\
    & + \frac{1-\cos (kx)}{2}\d\alpha_- {+ k\cos(kx)\d x \wedge \d y} \\ 
    = & \,  -  k\sin(kx)e^{t_1+\ldots + t_n }\d x \wedge\d\theta_0 -\sum_{i=1}^n e^{-t_i}\d t_i \wedge\d\theta_i  \\
    & + \cos(kx) e^{\sum_i t_i} \sum_{i=1}^n \d t_i \wedge \d\theta_0 {+  k\cos(kx)\d x \wedge \d y} \, .
    \end{align*}

Hence, for all $h=1,\ldots, n+1$, we get
    \begin{align*}
        \beta_{k,s}\wedge \d\beta_{k,s}^h = \; & 
        s^h \d x \wedge (\d \alpha_k)^h + O(s^{h+1}) \\
        = \; & s^h \d x \wedge \left[ 
             -  k\sin(kx)e^{t_1+\ldots + t_n }\d x \wedge\d\theta_0 -\sum_{i=1}^n e^{-t_i}\d t_i \wedge\d\theta_i \right. \\
    & \left. + \cos(kx) e^{\sum_i t_i} \sum_{i=1}^n \d t_i \wedge \d\theta_0  +  k\cos(kx)\d x \wedge \d y \, 
        \right]^h + O(s^{h+1}) \\
        = \; & s^h\d x \wedge \left[ 
             -\sum_{i=1}^n e^{-t_i}\d t_i \wedge\d\theta_i 
    + \cos(kx) e^{\sum_i t_i} \sum_{i=1}^n \d t_i \wedge \d\theta_0 \, 
        \right]^h + O(s^{h+1}) \, .
    \end{align*}
This shows that one has $\frac{1}{s^h}\beta_{k,s}\wedge \d\beta_{k,s}^h \to \d x \wedge  \mu_h$ with 
\[
\mu_h = \left[ 
             -\sum_{i=1}^n e^{-t_i}\d t_i \wedge\d\theta_i 
    + \cos(kx) e^{\sum_i t_i} \sum_{i=1}^n \d t_i \wedge \d\theta_0 \, 
        \right]^h \, .
\]
The only thing left to prove is hence that $\d x \wedge \mu_h\wedge \omega^{n+1-h}>0$ everywhere.
For this, we compute
\begin{align*}
    \d x \wedge \mu_h\wedge \omega^{n+1-h} =  & \, 
    \d x \wedge \left[ 
             -\sum_{i=1}^n e^{-t_i}\d t_i \wedge\d\theta_i 
    + \cos(x) e^{\sum_i t_i} \sum_{i=1}^n \d t_i \wedge \d\theta_0 \, 
        \right]^h  \\
        & \qquad\qquad  \wedge \left[e^{t_1+\ldots + t_n}\d\theta_0 \wedge \d y - \sum_{i=1}^n e^{-t_i} \d t_i \wedge \d\theta_i \right]^{n+1-h} \\
        \overset{(*)}{=} & \, (n+1-h) e^{t_1+\ldots + t_n}\d x \wedge \d\theta_0 \wedge \d y \left[- \sum_{i=1}^n e^{-t_i} \d t_i \wedge \d\theta_i\right]^n  \\
        = & \, (-1)^{n+1} (n+1-h) n! \d x \wedge \d y \wedge \d\theta_0 \wedge \d t_1 \wedge \d \theta_1 \wedge \ldots\wedge \d t_n \wedge \d\theta_n \\
        = & \, (-1)^{n} (n+1-h) n! \d\theta_0 \wedge \d t_1 \wedge \d \theta_1 \wedge \ldots\wedge \d t_n \wedge \d\theta_n \wedge  \d x \wedge \d y 
\end{align*}
where $(*)$ comes from the following facts: first, at least one term $\d y$, necessarily coming from $\omega$, should be present in any non-zero contribution; second, the term containing $\d\theta_0 $ in $\mu_h$ will contribute just trivial terms; and third, regrouping the remaining terms from $\mu^h$ and $\omega^{n-h}$ that are $h$-th and $(n-h)$-th powers of the same term $\sum_{i=1}^n e^{-t_i}\d t_i \wedge \d\theta_i$.
Now, the above term is indeed positive, as $M$ is oriented in such a way that $\alpha_+$ is a positive contact form, and we can compute 
\[
\alpha_+\wedge \d\alpha_+^n = (-1)^{n} (n+1)! \d\theta_0 \wedge \d t_1 \wedge \d \theta_1 \wedge \ldots\wedge \d t_n \wedge \d\theta_n \, .
\]

It remains to consider the top-degree case \(h=n+1\). We have
\begin{align*}
\beta_{k,s}\wedge(\d\beta_{k,s})^{n+1}
={}&
s^{n+1}\d x\wedge(\d\alpha_k)^{n+1}
+
s^{n+2}\alpha_k\wedge(\d\alpha_k)^{n+1}.
\end{align*}
The first term vanishes: after wedging with \(\d x\), the remaining
part of \(\d\alpha_k\) has no \(\d y\)-component, and hence its
\((n+1)\)-st power is zero. Therefore,
\[
\beta_{k,s}\wedge(\d\beta_{k,s})^{n+1}
=
s^{n+2}\alpha_k\wedge(\d\alpha_k)^{n+1}.
\]
Since \(\alpha_k\) is a contact form, the right-hand side is positive
for \(s>0\). Thus, after rescaling by \(s^{-(n+2)}\), it converges to
the required positive top-degree form.

This concludes the proof that $\beta_{k,s}$ defines a contact strict deformation, as desired. 
\hfill \qedsymbol


\subsection{Stability of weak fillability}
\label{sec:stability_weak_fillability}

In \Cref{sec:proof_stability_weak_fill_under_contact_tame_approx}
we prove  \Cref{thm:weak_fillability_stable_under_tame_contact_approx} stating the stability of weak fillability under contact tame approximations, as defined in \Cref{def:contact_tame_approx_and_deform}.
Afterwards, in \Cref{sec:proof_stability_weak_fill_under_contact_strict_approx} we prove \Cref{thm:weak_fillability_stable_under_strict_contact_approx}, dealing with the analogous stability for the contact approximations of strict type, as defined in \Cref{def:contact_strict_approx_and_deform}.

\subsubsection{The tame case}
\label{sec:proof_stability_weak_fill_under_contact_tame_approx}

We use here the notations in \Cref{thm:weak_fillability_stable_under_tame_contact_approx} and \Cref{def:contact_tame_approx_and_deform}.

According to \cite[Definition 4 and Theorem 2.4]{MassotNiederkrugerWendl_WeakFillings}, in order to prove that $\xi_m$ is weakly filled by $(W,\Omega)$ it is equivalent to prove that there is a complex structure $J_m$ on $\xi_m$ that is tamed at the same time by $\CS_{\xi_m}$ and by $\Omega\vert_{\xi_m}$.
We are hence going to prove that this is true, for $m$ large enough, for the sequence $J_m$ in \Cref{def:contact_tame_approx_and_deform}.

The fact that $J_m$ is tamed by $\CS_{\xi_m}$ is already in the \Cref{def:contact_tame_approx_and_deform}. What needs to be proven is hence that $J_m$ is tamed also by $\Omega\vert_{\xi_m}$.

For this, consider $v_m \in \xi_m$ a vector of unit norm, w.r.t.\ an auxiliary Riemannian metric $g$ which will also be used to define the extension $\widetilde J_m$.
Then, we can decompose $v_m = u_m + e_m \in \eta\oplus \eta^\perp$.
Up to passing to a subsequence, we can also assume that  $e_m\to 0$ and $u_m\to u_\infty\in \eta$, as the unit sphere bundle of $\xi_m$ converges to that of $\eta$ because of the assumption that $\xi_m\xrightarrow{C^k}\eta$.

As $\widetilde J_m \xrightarrow{C^k} \widetilde J$ by assumption of contact tame $C^k$-approximation, we then have
\begin{equation*}
    \Omega\vert_{\xi_m}(v_m,J_m v_m) = \Omega(u_m + e_m , \widetilde J_m(u_m) + \widetilde J_m (e_m)) \xrightarrow{m\to \infty} \Omega(u_\infty,\widetilde J u_\infty)=\Omega\vert_\eta(u_\infty, J u_\infty) \, .
\end{equation*}

Now, by definition of weak filling $(W,\Omega)$ of $(\eta,\CS_{\eta,\mu})$ we know that $\Omega\vert_{\cK_\eta}= f\mu\vert_{\cK_\eta}$ for some function $f>0$.
According to \cite[Lemma 2.46]{CardonaGironella_ConfoliationsHighDim}, we moreover know that, up to replacing $\Omega\vert_M$ by  $\Omega\vert_M + T \d\beta$ for some large $T>0$ (which can be realized by attaching a long cylindrical end on top of $(W,\Omega)$ according to the normal form \cite[Lemma 2.45]{CardonaGironella_ConfoliationsHighDim} near the boundary), we can moreover arrange that $\omega:=\Omega\vert_M + T \d\beta$ on $M$ realizes that the symmetric bilinear form $(\omega\vert_\eta)(\cdot , J \cdot)$ appearing in the right hand side of the above equation (or rather its shifted version) is strictly positive.

Once this is achieved, by compactness of the sphere bundle of $\eta$, there is $m_0>0$ so that, for each $m>m_0$, 
\[
    \omega\vert_{\xi_m}(\cdot,J_m \cdot)\colon S\xi_m\to \bR
\]
on the sphere bundle $S\xi_m$ is lower bounded by a positive real constant, as desired.
\hfill \qedsymbol

\subsubsection{The strict case}
\label{sec:proof_stability_weak_fill_under_contact_strict_approx}

We use here the notations in \Cref{thm:weak_fillability_stable_under_strict_contact_approx} and \Cref{def:contact_strict_approx_and_deform}.
Up to rescaling $\alpha_m$, we can also arrange that $\alpha_m\overset{C^1}{\to}\beta$ uniformly in $m$.

Recall also that, by definition of $(W,\Omega)$ being a weak filling of $(M,\eta,\CS_{\eta,\mu})$, we have that $\Omega_M\vert_{\cK_{\eta}}= f \mu\vert_{\cK_\eta}$ for some positive function $f\colon M \to \bR_{>0}$, where $\Omega_M := \Omega\vert_M$. \Cref{rmk:contact_approx_and_deform_only_depends_on_omega_on_Keta} then guarantees that \Cref{item:contact_approx_and_deform__conformal_convergence_other_mixed_terms} holds for $\Omega_M$ as choice of extension of $\mu\vert_{\cK_\eta}$ to $M$.

Let us also point out that, more formally, the conclusion of \Cref{thm:weak_fillability_stable_under_strict_contact_approx} can be formulated as follows:
\begin{equation}
\label{eqn:stability_weak_fillability_statement}
\tag{$\dagger$}
\exists\, m_0>0 \text{ such that }
\forall\, m\ge m_0,\quad
\alpha_m\wedge(\Omega_M+t\,\d\alpha_m)^n>0
\quad \forall\, t\ge0.
\end{equation}

\noindent
We are hence going to prove the following property, by reverse induction on $i$, from $i=n$ to $i=0$, where $\Vol$ is a fixed auxiliary volume form on $M$:
\begin{equation}
\label{eqn:stability_weak_fillability_statement_i}
\tag{$\dagger_i$}
\begin{aligned}
&\exists\, m_0^i>0,\;  \exists \,\delta_i>0 , \;
\exists \text{ neighborhood } V_i \text{ of } \cup_{j=i}^n C_j
\text{ such that} \\
&\qquad
\forall\, m\ge m_0^i,\;
\forall\, t\ge0,\;
\alpha_m\wedge(\Omega_M+t\,\d\alpha_m)^n>\delta_i \Vol
\text{ on } V_i \, .
\end{aligned}
\end{equation}
As \Cref{thm:weak_fillability_stable_under_strict_contact_approx} is simply a rephrasing of $(\dagger_0)$, proving
\eqref{eqn:stability_weak_fillability_statement_i} for all $i$ is enough to conclude.

\bigskip

Let's start by proving \eqref{eqn:stability_weak_fillability_statement_i} with $i=n$.
For this, we expand 
\begin{equation}    \label{eqn:proof_stability_weak_fillability_under_approx__base_step}
    \alpha_m\wedge(\Omega_M+t\,\d\alpha_m)^n = \sum_{h=0}^n \binom{n}{h} t^h \alpha_m\wedge\d\alpha_m^h\wedge \Omega_M^{n-h} \, .  
\end{equation}
At all points of $C_n$, where $\rk\cK_\eta=2n$ is maximal, according to \Cref{def:contact_strict_approx_and_deform} (and to the definition of conformal convergence of partitioned forms), we have a finite family (independent of $m$!) of open sets $O_{j,n}\subset M$ for $j\in J_n$ such that $O_n:=\cup_{j\in J_n} O_{j,n}$ covers $C_n$ and such that, for each $j\in J_n$ and $h=1,\ldots , n$, on $O_{j,n}$ there is a sequence (indexed by $m$) of positive functions $f_{m,h,n,j}$ and differential forms $\mu_{h,n,j}$ of degree $2h$ (recall that $\mu_{h,n,j}=1$ if its degree is zero, by convention), such that 
\[  f_{m,h,n,j} \alpha_m \wedge \d \alpha_m^h \xrightarrow{C^0} \rho_{h,n,j}  \]
where $\rho_{h,n,j}$ is an extension over $O_{j,n}$ of $\beta \wedge \mu_{h,n,j}$ on $O_{j,n}\cap C_n$, and so that
\[ \beta\wedge\mu_{h,n,j}\wedge \Omega_M^{n-h} >0 \text{ at all points of }C_n \, . \]
As $C_n$ is closed in $M$ that is compact, it is itself compact, and hence we know that there is $\delta_{n,h}>0$ such that
\[ \beta\wedge\mu_{h,n,j}\wedge \Omega_M^{n-h} > 2\delta_{n,h} \Vol \text{ at all points of }C_n \, , \]
and up to taking $\delta_n:=\min_{0\leq h \leq n} \delta_{n,h}$ and to shrinking $O_n$ to a smaller open neighborhood $V_n$ of $C_n$, we have
\[
    \rho_{h,n,j}\wedge \Omega_M^{n-h}>\frac{3}{2}\delta_n \Vol \text{ at all points of }V_n \, .
\]
Going back to \eqref{eqn:proof_stability_weak_fillability_under_approx__base_step}, we can then rewrite 
\begin{equation}    \label{eqn:proof_stability_weak_fillability_under_approx__base_step_bis}
    \alpha_m\wedge(\Omega_M+t\,\d\alpha_m)^n = \alpha_m\wedge \Omega_M^n +  \sum_{h=1}^n \binom{n}{h} \frac{t^h}{f_{m,h,n,j}} f_{m,h,n,j} \alpha_m\wedge\d\alpha_m^h\wedge \Omega_M^{n-h} 
\end{equation}
and, for all $h=1,\ldots, n$, we have $f_{m,h,n,j}>0$ and 
\[
f_{m,h,n,j} \alpha_m\wedge\d\alpha_m^h\wedge \Omega_M^{n-h} \xrightarrow{m\to \infty} \rho_{h,n,j}\wedge \Omega_M^{n-h} \text{ over }V_n \, .
\]
As the right hand side is bounded below by $\frac{3}{2}\delta_n \Vol$ on $V_n$, we have that, for each $m$ big enough and up to shrinking slightly $V_n$ (but still denoting it the same), each coefficient of the polynomial in $t$ from \eqref{eqn:proof_stability_weak_fillability_under_approx__base_step_bis} is bounded below by $\delta_n \Vol$.
In particular, there is $m_0>0$ such that for all $m>m_0$, $\alpha_m\wedge (\Omega_M+t\d\alpha_m)^n>\delta_n \Vol$ on all of $V_n$, as desired in $(\dagger_n)$.
(Note that this relies \emph{in a fundamental way} on the fact that the top degree coefficient of the polynomial in \eqref{eqn:proof_stability_weak_fillability_under_approx__base_step_bis} is strictly positive uniformly in $m$, for $m$ big enough; c.f.\ \Cref{exa:Mori_definition}.)

\bigskip

It is now left to prove the inductive step: namely, we assume that $(\dagger_{i+1})$ is true, and prove \eqref{eqn:stability_weak_fillability_statement_i}.
This is very similar to the base step, so it will be treated more succinctly.

As we already know that each coefficient of the polynomial $\alpha_m\wedge (\Omega_M+t\d\alpha_m)^n$ is strictly positive at all points in $V_{i+1}$ by the inductive assumption, it is enough to prove that the $t$-coefficients are also strictly positive on $C_i\setminus V_{i+1}$.

For this, we decompose the polynomial $P(t)$ at points of $C_i\setminus V_{i+1}$ as follows:
\begin{equation}
\label{eqn:proof_stability_weak_fillability_under_approx__inductive_step}
    \alpha_m\wedge(\Omega_M+t\,\d\alpha_m)^n = \sum_{h=0}^{n-i} \binom{n}{h} t^h \alpha_m\wedge\d\alpha_m^h\wedge \Omega_M^{n-h} \;+ \sum_{h=n-i+1}^{n} \binom{n}{h} t^h \alpha_m\wedge\d\alpha_m^h\wedge \Omega_M^{n-h} \, .  
\end{equation}

We now discuss positivity of the first sum in the right hand side above.

Now, because $\alpha_m \xrightarrow{C^1} \beta$ (that we previously arranged up to rescaling), we have that, at all points of $C_i\setminus V_{i+1}$,
\[
    \sum_{h=0}^{n-i} \binom{n}{h} t^h \alpha_m\wedge\d\alpha_m^h\wedge \Omega_M^{n-h} \xrightarrow{m\to \infty} \sum_{h=0}^{n-i} \binom{n}{h} t^h \beta\wedge\d\beta^h\wedge \Omega_M^{n-h} \overset{(*)}{=} \beta\wedge(\Omega_M + t\d\beta)^n
\]
where $(*)$ comes from the fact that $\cK_\eta$ has rank $i$ at points of $C_i\setminus V_{i+1}$.
At this point, either directly in the case of quasi-strong filling or up to attaching a long cylindrical end and using \Cref{lem:strict_weak_filling_c-symplectic_confoliation_positive_coefficients_polynomial} in the case of strict weak filling, 
we have that $\beta\wedge(\Omega_M + t\d\beta)^n>0$ for all $t\geq 0$. 
Moreover, as this is a polynomial of degree exactly $n-i$ at all points of $C_i\setminus V_{i+1}$, just like 
$\sum_{h=0}^{n-i} \binom{n}{h} t^h \alpha_m\wedge\d\alpha_m^h\wedge \Omega_M^{n-h}$ is, together with the fact that the above limit holds, we can deduce that, for all $p\in C_i\setminus V_{i+1}$ there is $m_i^p>0$ 
such that for all $m>m_i^p$, 
$\sum_{h=0}^{n-i} \binom{n}{h} t^h \alpha_m\wedge\d\alpha_m^h\wedge \Omega_M^{n-h}$ is strictly positive at $p$, say bounded below by $\delta_{i,p}'\Vol$. 
Lastly, as $C_i\setminus V_{i+1}$ is a closed subset of $M$ (as $\rk\cK_\eta$ is upper semi-continuous) and hence it is compact, one can find uniform constants $m_i'>0$ and $\delta_i'>0$ guaranteeing the same positivity property for all points $p\in C_i\setminus V_{i+1}$.

As far as the $t$-coefficients of the terms in the second sum in \eqref{eqn:proof_stability_weak_fillability_under_approx__inductive_step} are concerned, 
by the same argument as in the base case, relying on the convergence property guaranteed by \Cref{def:contact_strict_approx_and_deform} and on compactness of $C_i\setminus V_{i+1}$, one can prove that there is $m_i''>0$ so that for all $m>m_i''$ the concerned $t$-coefficients in the second sum of the right hand side of \eqref{eqn:proof_stability_weak_fillability_under_approx__inductive_step} are also strictly positive, say bounded below by $\delta_i''\Vol$.

Now, by compactness of $C_{i}\setminus V_{i+1}$ and by the fact that the $t^n$-coefficient is strictly positive on this subset, the lower bound will for each of the $t$-coefficients in \eqref{eqn:proof_stability_weak_fillability_under_approx__inductive_step} will in fact hold on a neighborhood $V_{i}'$ of $C_i\setminus V_{i+1}$.

Then, choosing $V_i:=V_i' \cup  V_{i+1}$, $m_i:=\max(m_i',m_i'',m_{i+1})$ and $\delta_i = \min (\delta_i',\delta_i'',\delta_{i+1})$ we get the conclusion of \eqref{eqn:stability_weak_fillability_statement_i}, as desired in the inductive step.

\medskip

This concludes the proof of \Cref{thm:weak_fillability_stable_under_strict_contact_approx}.
\hfill\qedsymbol


\subsection{A symplectic foliation without contact approximations}
\label{sec:sympl_fol_without_contact_approx}

In this section we prove \Cref{cor:no_contact_approximations_product_foliation} from the Introduction.
In fact, this follows immediately from the combination of \Cref{lem:strong_sympl_fol_are_quasi-strong_sympl_semi-fillable} and \Cref{thm:rank_char_fol_at_least_two,thm:weak_fillability_stable_under_tame_contact_approx,thm:weak_fillability_stable_under_strict_contact_approx}. 

\smallskip

More precisely, \Cref{lem:strong_sympl_fol_are_quasi-strong_sympl_semi-fillable} guarantees that there is a weak (in fact strong) symplectic filling $W$ of 
\[
(\overline M,\overline \cF,\omega) \sqcup (M,\cF,\omega) \, ,
\]
where $M=\bS^1\times \bS^2\times X$, $\cF=\ker \d\theta$ with $\theta\in \bS^1$, and $\omega= \omega_{\bS^2}+\omega_X$.

Let's assume by contradiction that there is a sequence $(\xi_m)_{m\in \bN}$ of contact structures on $M$ approximating, in the tame or strict sense, $(\cF,\CS_{\cF,\omega})$.
Then, according to \Cref{thm:weak_fillability_stable_under_tame_contact_approx,thm:weak_fillability_stable_under_strict_contact_approx}, for $m$ large enough, one can see $W$ as a weak filling of 
\[
(\overline M,\overline \cF,\CS_{\cF,\omega}) \sqcup (M,\xi_m) \, .
\]

Now, as $\overline\cF$ is a foliation by manifolds $\bS^2\times X$, with $\omega_X$ symplectically aspherical, \Cref{thm:rank_char_fol_at_least_two} applies, and it implies that there is no point on $\partial W$ where the hyperplane distribution is contact.
This clearly contradicts the fact that one boundary component of $W$ is entirely made of contact points (for $\xi_m$).
This concludes the proof of \Cref{cor:no_contact_approximations_product_foliation}.
\hfill \qedsymbol


\bibliographystyle{alpha}
\bibliography{biblio_v1}

\Addresses

\end{document}